\documentclass[12pt]{amsart}

\usepackage[english]{babel}

\usepackage[colorinlistoftodos]{todonotes}

\usepackage[pdftex,paper=a4paper,portrait=true,textwidth=450pt,textheight=675pt,tmargin=3cm,marginratio=1:1]{geometry}

\usepackage{amsfonts}

\usepackage{amsmath}

\usepackage{bbm}

\usepackage{amsthm}

\usepackage{amssymb}
\usepackage{enumerate}

\usepackage{eufrak}

\usepackage{cancel}

\usepackage{color}

\usepackage[curve]{xypic}

\newtheorem{theorem}{Theorem}[section]

\newtheorem{corollary}[theorem]{Corollary}

\newtheorem{proposition}[theorem]{Proposition}

\newtheorem{proposition-definition}[theorem]{Proposition-Definition}

\newtheorem{lemma}[theorem]{Lemma}

\newtheorem{remark}[theorem]{Remark}

\theoremstyle{definition}

\newtheorem{definition}[theorem]{Definition}

\theoremstyle{property}

\makeatletter
\newcommand{\contraction}[5][1ex]{%
  \mathchoice
    {\contraction@\displaystyle{#2}{#3}{#4}{#5}{#1}}%
    {\contraction@\textstyle{#2}{#3}{#4}{#5}{#1}}%
    {\contraction@\scriptstyle{#2}{#3}{#4}{#5}{#1}}%
    {\contraction@\scriptscriptstyle{#2}{#3}{#4}{#5}{#1}}}%
\newcommand{\contraction@}[6]{%
  \setbox0=\hbox{$#1#2$}%
  \setbox2=\hbox{$#1#3$}%
  \setbox4=\hbox{$#1#4$}%
  \setbox6=\hbox{$#1#5$}%
  \dimen0=\wd2%
  \advance\dimen0 by \wd6%
  \divide\dimen0 by 2%
  \advance\dimen0 by \wd4%
  \vbox{%
    \hbox to 0pt{%
      \kern \wd0%
      \kern 0.5\wd2%
      \contraction@@{\dimen0}{#6}%
      \hss}%
    \vskip 0.5ex
    \vskip\ht2}}

\newcommand{\contraction@@}[3][0.05em]{%
  \hbox{%
    \vrule width #1 height 0pt depth #3%
    \vrule width #2 height 0pt depth #1%
    \vrule width #1 height 0pt depth #3%
    \relax}}
\makeatother

\DeclareFontFamily{OT1}{rsfs}{}
\DeclareFontShape{OT1}{rsfs}{n}{it}{<-> rsfs10}{}
\DeclareMathAlphabet{\curly}{OT1}{rsfs}{n}{it}

\renewcommand\L{\mathcal L}

\renewcommand\S{\mathcal S}
\renewcommand\O{\mathcal O}
\newcommand\PP{\mathbb P}

\newcommand\cE{\mathcal E}

\newcommand\F{\mathcal F}

\newcommand\G{\mathcal G}

\newcommand\U{\mathfrak U}

\newcommand\X{\mathfrak X}

\newcommand\C{\mathbb C}

\newcommand\cB{\mathcal B}

\newcommand\cQ{\mathcal Q}

\newcommand\Q{\mathbb Q}

\newcommand\Z{\mathbb Z}

\newcommand\ch{\mathrm{ch}}

\newcommand\td{\mathrm{td}}

\newcommand\vb{\mathrm{vb}}

\newcommand\Into{\ar@{^(->}[r]<-.3ex>}

\newcommand\rk{\operatorname{rank}}

\newcommand\id{\operatorname{id}}

\newcommand\Pic{\operatorname{Pic}}

\newcommand\Spec{\operatorname{Spec}\,}

\newcommand\beq[1]{\begin{equation}\label{#1}}
\newcommand\eeq{\end{equation}}
\newcommand\beqa{\begin{eqnarray*}}
\newcommand\eeqa{\end{eqnarray*}}

\renewcommand{\id}{\operatorname{id}}
\renewcommand{\P}{\PP(a,b,c)}
\newcommand{\bs}{\boldsymbol}
\newcommand{\bmu}{\bs{\mu}}

\let\polishl\l
\renewcommand{\l}{\ell}

\renewcommand{\i}{\hat{i}}
\renewcommand{\j}{\hat{j}}
\renewcommand{\k}{\hat{k}}
\newcommand{\rG}{\mathsf{G}}
\newcommand{\rH}{\mathsf{H}}

\newcommand{\rc}{\mathsf{c}}
\newcommand{\T}{\mathbf{T}}

\DeclareRobustCommand{\SkipTocEntry}[4]{}

\begin{document}
\title[Weighted projective planes and modular forms]{Sheaves on weighted projective planes and modular forms}
\author[A.~Gholampour, Y.~Jiang and M. Kool]{Amin Gholampour, Yunfeng Jiang and Martijn Kool}
\maketitle
\begin{abstract}
We give an explicit description of toric sheaves on the weighted projective plane $\PP(a,b,c)$ viewed as a toric Deligne-Mumford stack. The integers $(a,b,c)$ are not necessarily chosen coprime or mutually coprime allowing for gerbe and root stack structures. 

As an application, we describe the fixed point locus of the moduli scheme of stable rank 1 and 2 torsion free sheaves on $\PP(a,b,c)$ with fixed $K$-group class. Summing over all $K$-group classes, we obtain explicit formulae for generating functions of the topological Euler characteristics.

In the case of stable rank 2 locally free sheaves on $\PP(a,b,c)$ with $a,b,c \leq 2$ the generating functions can be expressed in terms of Hurwitz class numbers and give rise to modular forms of weight $3/2$. This generalizes Klyachko's computation on $\PP^2$ and is consistent with $S$-duality predictions from physics. 
\end{abstract}

\tableofcontents

\section{Introduction}

A toric sheaf on a smooth toric variety $X$ is a coherent sheaf together with a lift of the action of the torus $\T$ on $X$. Toric sheaves on smooth toric varieties have been well-studied, e.g.~in the work of A.~A.~Klyachko and M.~Perling \cite{Kly1, Per1} among others. The idea is to cover $X$ by affine open $\T$-invariant subsets $U_\sigma \cong \C^d$, corresponding to the cones $\sigma$ of dimension $d$ in the fan $\Sigma$ of $X$ and restrict $\F$ to each of these open subsets. Over these open subsets, the toric sheaf is given by a module with an $X(\T)$-grading, where $X(\T)$ is the character group of $\T$. This local data becomes particularly explicit in the case $\F$ is in addition locally free, reflexive, torsion free or just pure. Supplementing this local data with gluing conditions gives a very explicit description of $\F$. This can be used to compute Chern classes of $\F$ and describe moduli of toric sheaves \cite{Kly2, Per2, Pay, Koo1}. This in turn can be used to compute the topological Euler characteristic of the entire moduli space of stable sheaves on $X$ using localization \cite{Kly2, Koo2}.

Our goal is to generalize the above to toric Deligne-Mumford stacks\footnote{We only consider the case in which the stack $\X$ is smooth as in \cite{BCS}.} $\X$ introduced by L.~A.~Borisov, L.~Chen and G.~G.~Smith \cite{BCS}. Such stacks are described by a stacky fan $\mathbf{\Sigma}$ \cite{BCS, FMN}. Again, one can consider the cones $\sigma$ of dimension $d$ in $\mathbf{\Sigma}$, which gives a cover by open $\T$-invariant substacks $\mathfrak{U}_\sigma \cong [\C^d/N(\sigma)]$ \cite[Prop.~4.1]{BCS}. Here $N(\sigma)$ is a finite abelian group acting on $\C^d$. We can then restrict a toric sheaf $(\F,\Phi)$ on $\X$ to each of these open substacks. Viewing these open substacks as groupoids, we get similar data as before, but with two new phenomena:
\begin{enumerate} [i)]
\item The action of the torus $\T$ on $\C^d$ can be \emph{non-primitive}. On a smooth toric variety, one can always choose coordinates on $U_\sigma$ such that the action becomes $$(\lambda_1, \ldots, \lambda_d)\cdot(x_1, \ldots, x_d) = (\lambda_1 x_1, \ldots, \lambda_d x_d).$$ However in the current setup higher powers of  $\lambda_i$ can occur.
\item The sheaf $\F|_{\mathfrak{U}_\sigma}$ is described by a module carrying a grading by $X(\T)$ and each weight space carries a further \emph{fine grading} by the character group $X(N(\sigma))$ of $N(\sigma)$.
\end{enumerate}
This local data becomes particularly explicit when $\F$ is in addition locally free, reflexive, torsion free etc. 

To obtain a global description, one has to impose gluing conditions on the local data. In order to avoid overly cumbersome notation and because our primary objective is to compute generating functions of topological Euler characteristics of moduli spaces of stable sheaves, we restrict attention to an example: stacky weighted projective planes $\PP(a,b,c)$ for \emph{any} positive integers $a,b,c$. 
The integers $(a,b,c)$ are not necessarily chosen coprime or mutually coprime allowing for gerbe and root stack structures. 

In the case of smooth toric varieties, the gluing proceeds by intersecting any two open affine invariant subsets $U_\sigma$, $U_\tau$ and matching the local data. This involves pulling back the local data along \emph{open immersions}. This times, we have to ``intersect'' any two open substacks $\mathfrak{U}_\sigma \cong [\C^d/N(\sigma)]$, $\mathfrak{U}_\tau \cong [\C^d/N(\tau)]$ and pull-back the local data to the stack theoretic intersection. At the level of groupoids, this turns out to correspond to pulling back the local data along certain \emph{\'etale morphisms}. In this way, we obtain an explicit description of toric sheaves on $\PP(a,b,c)$. This can be used to study toric locally free and toric torsion free sheaves on $\PP(a,b,c)$ and describe their moduli.

As an application, we
consider the moduli spaces of stable rank 1 and 2 torsion free sheaves on $\PP(a,b,c)$. These moduli spaces\footnote{As opposed to Nironi, we consider $\mu$-stability instead of Gieseker stability and we fix the $K$-group class of the sheaf instead of the modified Hilbert polynomial. Here $\mu$ is the linear term of the \emph{modified} Hilbert polynomial divided by its quadratic term. This means $\mu$ depends on a choice of polarization and \emph{generating sheaf}.}, which are in fact schemes, have been introduced by F.~Nironi \cite{Nir} for any rank and any (not necessarily toric) projective DM stack. We briefly recall the relevant parts of Nironi's construction at the beginning of Section 7. The action of the torus $\T$ on $\PP(a,b,c)$ lifts to the moduli space and we use the explicit description of toric sheaves on $\PP(a,b,c)$ to compute its fixed point locus and its Euler characteristic. The rank 1 case leads to counting of coloured partitions (e.g.~as in \cite{DS}). 

In the rank 2 case, generating functions of Euler characteristics of moduli spaces of stable locally free sheaf on $\PP^2$ were first computed by Klyachko in his beautiful papers \cite{Kly2, Kly3}. He expresses his answer in terms of Hurwitz class numbers $H(\Delta)$ and the sum of divisors function $\sigma_0(n)$. For any $\Delta>0$, $H(\Delta)$ is the number of (equivalence classes
of) positive definite integer binary quadratic forms $AX^2+BXY+CY^2$ with discriminant $B^2 - 4AC = -\Delta$ and weighted by the size of its automorphisms group\footnote{This means forms equivalent to $\lambda(X^2+Y^2)$ and $\lambda(X^2+XY+Y^2)$ are counted with weights $1/2$ and $1/3$ respectively.}. 
\begin{theorem} [Klyachko] \label{Klyachko}
Let $N_{\PP^2}(2,c_1,\chi)$ be the moduli space of rank 2 stable locally free sheaves on $\PP^2$ with first Chern class $c_1 \in \Z$ and holomorphic Euler characteristic $\chi$. Let
$$
\rH^{\vb}_{c_1}(q) := \sum_{\chi} e(N_{\PP^2}(2,c_1,\chi)) q^\chi,
$$
where $e(\cdot)$ denotes topological Euler characteristic. Then $\rH^{\vb}_{c_1}(q)$ equals $q^{\frac{1}{4}c_{1}^{2}+\frac{3}{2}c_1+2}$ times
$$
\sum_{n=1}^{\infty} 3H(4n-1) q^{\frac{1}{4}-n}
$$
when $c_1$ is odd and
$$
\sum_{n=1}^{\infty} 3\Big(H(4n) - \frac{1}{2} \sigma_0(n) \Big) q^{-n}
$$
when $c_1$ is even.
\end{theorem}
\noindent In the physics literature, C.~Vafa and E.~Witten \cite{VW} observed that in the case $c_1$ is odd Klyachko's expression and work of D.~Zagier  \cite{Zag} (see also \cite{HZ}) imply $\rH^{\vb}_{c_1}(q)$ ($q=e^{2 \pi i z}$, $\mathrm{Im}(z)>0$) is the holomorphic part of a modular form of weight $3/2$ for $\Gamma_0(4)$ (up to replacing $q$ by $q^{-1}$ and up to an overall power of $q$ in front). In the case $c_1$ is even one only obtains modularity after correctly adding strictly semistable sheaves to the moduli space. Their contribution turns out to cancel the sum of divisors term\footnote{This follows from a computation of K.~Yoshioka as mentioned in \cite{VW}.}. Modularity of these generating functions is predicted by a conjecture from physics called the $S$-duality conjecture \cite{VW}.

We compute the analog of the generating function\footnote{For the precise definition, see Section 7.2. In our case we fix $c_1 \in \Z$ and $\lambda \in \{0, \ldots, d-1\}$, where $d = \gcd(a,b,c)$. Here $\lambda$ is, roughly speaking, the eigenvalue with respect to the overall stabilizer $\bmu_d$.} $\rH^{\vb}_{c_1,\lambda}(q)$ for \emph{any} weighted projective plane $\PP(a,b,c)$ (Corollary \ref{unrefined}). Moreover, in the case of $a,b,c \leq 2$ we show $\rH^{\vb}_{c_1,\lambda}(q)$ can again be expressed in terms of $H(\Delta)$ and $\sigma_0(n)$.
\begin{theorem} \label{intro}
For $\PP(1,1,2)$ the generating function $\rH^{\vb}_{c_1,0}(q)$ equals $q^{\frac{1}{4}c_{1}^{2}+\frac{5}{2}c_1+6}$ times
$$
 \sum_{n=1}^{\infty} 2H(8n-1) q^{\frac{1}{4}-2n}
$$
when $c_1$ is odd and 
$$
\sum_{n=1}^{\infty} \Big(H(4n) + 2H(n) - \frac{1}{2} \sigma_0(n) \Big) q^{-n} - \sum_{n=1}^{\infty} \sigma_0(n) q^{-4n}
$$
when $c_1$ is even. 

For $\PP(1,2,2)$ the generating function $\rH^{\vb}_{c_1,0}(q)$ equals $q^{\frac{1}{8} c_{1}^{2}+\frac{3}{2}c_1+\frac{17}{4}}$ times
$$
\sum_{n=1}^{\infty} H(8n-1) q^{\frac{1}{8}-n}
$$
when $c_1$ is odd and $q^{\frac{1}{8} c_{1}^{2}+\frac{3}{2}c_1+4}$ times
$$
q^{\frac{1}{2}} \sum_{n=1}^{\infty} 3H(4n-1) q^{\frac{1}{2}- 2n} + \sum_{n=1}^{\infty} 3\Big(H(4n) - \frac{1}{2} \sigma_0(n) \Big) q^{-2n}
$$
when $c_1 \equiv 0 \ \mathrm{mod} \ 4$ and $q^{\frac{1}{8} c_{1}^{2}+\frac{3}{2}c_1+4}$ times
$$
\sum_{n=1}^{\infty} 3H(4n-1) q^{\frac{1}{2}- 2n} + q^{\frac{1}{2}} \sum_{n=1}^{\infty} 3\Big(H(4n) - \frac{1}{2} \sigma_0(n) \Big) q^{-2n}
$$
when $c_1 \equiv 2 \ \mathrm{mod} \ 4$.

For $\PP(2,2,2)$ the generating function $\rH^{\vb}_{c_1,\lambda}(q)$ is zero unless $c_1$ is even, so assume this is the case. For $\lambda$ even, $\rH^{\vb}_{c_1,\lambda}(q)$ is equal to the generating function $\rH^{\vb}_{\frac{c_1}{2}}(q)$ of Theorem \ref{Klyachko}. For $\lambda$ odd, $\rH^{\vb}_{c_1,\lambda}(q)$ is equal to the generating function $\rH^{\vb}_{\frac{c_1}{2}+1}(q)$ of Theorem \ref{Klyachko}.

In each of these cases, when the expression for $\rH^{\vb}_{c_1,\lambda}(q)$ does \emph{not} involve the sum of divisors function there are no strictly semistable locally free sheaves in the moduli spaces and $\rH^{\vb}_{c_1,\lambda}(q)$ ($q=e^{2 \pi i z}$, $\mathrm{Im}(z)>0$) is the holomorphic part of a modular form of weight $\frac{3}{2}$ (up to replacing $q$ by $q^{-1}$ or $q^{-\frac{1}{2}}$ and up to an overall power of $q$ in front).
\end{theorem}


We conjecture that in the absence of strictly semistable locally free sheaves in the moduli space, the generating function $\rH^{\vb}_{c_1,\lambda}(q)$ computed in Corollary \ref{unrefined} always has modular properties. 
In Section 7.2 we also give a general expression for the generating function of Euler characteristics of moduli spaces of stable locally free sheaves on any $\PP(a,b,c)$ which sums over $K$-group classes instead of holomorphic Euler characteristic (Theorem \ref{mainrank2}). These generating functions depend on more than one variable and are \emph{refinements} of $\rH^{\vb}_{c_1,\lambda}(q)$. It would be interesting to study their modular properties. \\

\noindent \textbf{Convention.} Throughout this paper we fix $\PP := \PP(a,b,c)$ with $a,b,c$ \emph{arbitrary} positive integers. \\

\noindent \textbf{Acknowledgments.} We thank Paul Johnson and Benjamin Young for useful discussions on counting of coloured partitions. We also thank Jan Manschot and Richard Thomas for helpful discussions. 

A.~G.~was partially supported by NSF grant DMS-1406788. M.~K.~was supported by EPSRC grant EP/G06170X/1, ``Applied derived categories''.

\section{$G$-equivariant sheaves on DM stacks with $G$-action}

Let $\X$ be a DM stack\footnote{All our stacks and schemes are defined over $\C$. 
}. A (quasi-)coherent sheaf $\F$ on $\X$ is determined by specifying a (quasi-)coherent sheaf $\F_U$ over each atlas $U \rightarrow \X$ together with compatibility isomorphisms between the $\F_U$ \cite[Def.~7.18]{Vis}. We are interested in the case $\X$ carries the action of an affine algebraic group $G$ and $\F$ has a $G$-equivariant structure. Note that the affine algebraic group $G$ is taken to be a variety so is non-stacky. 
\begin{definition} \cite[Def.~1.3, 4.3]{Rom} \label{def:G-action} 
Let $\X$ be a DM stack and $G$ an affine algebraic group with multiplication $\mu : G \times G \rightarrow G$, inverse map $\iota : G \rightarrow G$ and identity $e : \Spec \C \rightarrow G. $ A \emph{$G$-action} on $\X$ is a morphism $\bs{\sigma} : G \times \X \rightarrow \X$ together with 2-isomorphisms $\bs{\alpha}$ and $\bs{a}$ 
\[
\xymatrix
{
G \times G \times \X \ar[r]^-{\mu \times \bs{\mathrm{id}}_\X} \ar[d]_{\mathrm{id}_G \times \bs{\sigma}}  & G \times \X \ar[d]^{\bs{\sigma}} & &   \Spec \C\times  \X \ar[r]^-{e \times \bs{\mathrm{id}}_\X} \ar@{=}[dr] & G \times \X \ar[d]^{\bs{\sigma}} \ar@2+<-10mm,-7mm>^-{\bs{a}}  \\
G \times \X \ar[r]_{\bs{\sigma}} \ar@{=>}[ur]^{\bs{\alpha}}  & \X & & & \X,
}
\]
where the 2-isomorphisms $\bs{\alpha}$, $\bs{a}$ satisfy some natural compatibility conditions (see \cite[Def.~1.3]{Rom}).

Next, let $\F$ be a quasi-coherent sheaf on $\X$. Denote by $\bs{p}_2 : G \times \X \rightarrow \X$ and $\bs{p}_{23} : G \times G \times \X \rightarrow G \times \X$ projection on the second and last two factors. A \emph{$G$-equivariant structure} on $\F$ is an isomorphism $\Phi : \bs{\sigma}^{*} \F \rightarrow \bs{p}^{*}_{2} \F$ satisfying the cocycle identity 
\[
 \bs{p}_{23}^* \Phi \circ (\mathrm{id}_G \times \bs{\sigma})^* \Phi=(\mu \times \bs{\mathrm{id}}_\X)^* \Phi.
\]
When $\Phi$ is a $G$-equivariant structure on $\F$, we refer to the pair $(\F,\Phi)$ as a \emph{$G$-equivariant sheaf}.  A \emph{$G$-equivariant morphism} between $G$-equivariant sheaves $(\F,\Phi)$, $(\G,\Psi)$ is a morphism $\theta : \F \rightarrow \G$ satisfying $\bs{p}^{*}_{2}\theta \circ \Phi = \Psi \circ \bs{\sigma}^{*}\theta$. \hfill $\oslash$
\end{definition}

We leave various straight-forward notions such as ``$G$-equivariant morphisms between DM stacks with $G$-action'', ``$G$-invariant open substacks'' and ``pull-back of $G$-equivariant sheaves along $G$-equivariant morphisms'' to the reader. 

Let $\X$ be a DM stack with $G$-action and $\{\U_i\}$ a cover by $G$-invariant open substacks. We want to think of $G$-equivariant quasi-coherent sheaves on $\X$ as $G$-equivariant quasi-coherent sheaves on the $\U_i$ together with gluing conditions. The idea is that in applications, $G$-equivariant quasi-coherent sheaves on the $\U_i$ are much easier to describe, e.g.~when the $\U_i$ are quotient stacks. We use the notation $|_i$, $|_{ij}$, $|_{ijk}$ for restriction $|_{\U_i}$, $|_{\U_{ij}}$, $|_{\U_{ijk}}$, where $\U_{ij}:=\U_i \times_\X \U_j$ and $\U_{ijk}:=\U_i \times_\X \U_j \times_\X \U_k$. Using the definition of a quasi-coherent sheaf on a DM stack \cite[Def.~7.18]{Vis} and \cite[Exc.~II.1.22]{Har} on atlases, gives:
\begin{proposition} \label{glue}
Let $\X$ be a DM stack with $G$-action and let $\{\U_i\}$ be a cover by $G$-invariant open substacks. Let $\{(\F_i,\Phi_i)\}$ be a collection of $G$-equivariant quasi-coherent sheaves on the $\U_i$ and assume there are $G$-equivariant isomorphisms $\phi_{ij} : \F_i|_{ij} \rightarrow \F_j|_{ij}$ satisfying $\phi_{ii} = \mathrm{id}$ and the cocycle identity $\phi_{ik}|_{ijk} = \phi_{jk}|_{ijk} \circ \phi_{ij} |_{ijk}$. Then there exists a unique $G$-equivariant quasi-coherent sheaf $(\F,\Phi)$ on $\X$ together with $G$-equivariant isomorphisms $\phi_{i} : \F|_i \rightarrow \F_i$ satisfying $\phi_j|_{ij} = \phi_{ij} \circ \phi_i|_{ij} $.
\end{proposition}

Let $H$ be an affine algebraic group acting on a scheme $X$ and assume the stabilizers of the $\C$-points are reduced and finite. Then $\X = [X/H]$ is a DM stack \cite[Ex.~7.17]{Vis} and we refer to such stacks as quotient DM stacks. For later use, it is more convenient to work with the \emph{groupoid presentation} \cite[Sect.~7]{Vis} of a quotient stack $[X/H]$
$$\xymatrix{H\times X \ar[r]_-{\tau} \ar@<1ex>[r]^-{p_2} & X,}$$ where $p_2$ (the source map of the groupoid) is projection to the first factor and $\tau$ (the target map of the groupoid) is the given action of $H$ on $X$. We suppress notation of the other three structure morphisms: composition, unit, inverse (see \cite[Sect.~7]{Vis} and reference therein). 

\begin{proposition} \cite[Ex.~7.21]{Vis} \label{prop:equivalence} The category of quasi-coherent sheaves on a quotient DM stack $[X/H]$ is equivalent to the category of $H$-equivariant quasi-coherent sheaves on $X$. 
\end{proposition}
It is easy to generalise this to the $G$-equivariant setting. We start with the following lemma.
\begin{lemma} \label{lem:G-action}
Let $\X = [X/H]$ be a quotient DM stacks. Any $G$-action on $X$ commuting with $H$ induces a canonical $G$-action on $\X$. 
\end{lemma}
\begin{proof}
Let $\sigma: G\times X\to X$ be the $G$-action on $X$. The stack $G\times \X$ admits a groupoid presentation $$\xymatrix{G \times H\times X \ar[r]_-{\id_G \times \tau} \ar@<1ex>[r]^-{\id_G \times p_2} & G\times X,}$$ which is induced from the natural groupoid presentation of $\X$. The other structure maps are the obvious maps.
Let $\sigma_0 := \sigma$ and $$\sigma_1:=p_2 \times (\sigma \circ p_{13}):G\times H \times X \to H\times X.$$ The commutativity of the actions of $H$ and $G$ on $X$ implies that $(\sigma_1,\sigma_0)$ is  a morphism of groupoids $$\xymatrix{G \times H\times X \ar[d]_{\sigma_1} \ar[r]_-{\mathrm{id}_G \times \tau} \ar@<1ex>[r]^-{\mathrm{id}_G \times p_2} & G\times X \ar[d]^{\sigma_0} \\ H\times X \ar[r]_-{\tau} \ar@<1ex>[r]^-{p_2} & X,}$$ which gives rise to the morphism between the associated stacks. It is straightforward to check that this morphism indeed defines a $G$-action on $\X$ as in Definition \ref{def:G-action} with $\bs{\alpha}$ and $\bs{a}$ both taken to be the identity. 
\end{proof}

Given a scheme $X$ with commuting $G$- and $H$-actions and a quasi-coherent sheaf $\F$ on $X$, we can consider commuting $G$- and $H$-equivariant structures $\Phi$ and $\Psi$ on $\F$. Denoting the actions by $\sigma : G \times X \rightarrow X$, $\tau : H \times X \rightarrow X$ and projections by $p_{13} : G \times H \times X \rightarrow G \times X$ and $p_{23} : G \times H \times X \rightarrow H \times X$, this means
\begin{equation} \label{equ:GHequiv}
p_{13}^* \Phi \circ (\mathrm{id}_H \times \sigma)^* \Psi =  p_{23}^* \Psi \circ (\mathrm{id}_G \times \tau)^* \Phi.  
\end{equation}
Triples $(\F,\Phi,\Psi)$ with $\Phi$ and $\Psi$ commuting can be made into a category by considering morphisms which intertwine both the $G$- and $H$-equivariant structure. 
\begin{proposition}
Let  $\X = [X/H]$ be a quotient DM stack with $G$-action coming from a $G$-action on $X$ commuting with $H$. Then the category of $G$-equivariant quasi-coherent sheaves on $\X$ is equivalent to the category of quasi-coherent sheaves on $X$ with commuting $G$- and $H$-equivariant structures. 
\end{proposition}
\begin{proof} The proof is a straightforward application of the equivalence of categories in Proposition \ref{prop:equivalence}. We denote the $G$-action on $X$ by $\sigma$ and the $H$-action on $X$ by $\tau$ and use $p_i$, $p_{ij}$ for the various projections as before. 

Let $\bs{\sigma} ,\bs{p}: G\times \X \to \X$ be the $G$-action and projection. These two morphisms of stacks are obtained from the corresponding morphisms $(p_1, p_0 ) = (p_{23}, p_2)$ and $$(\sigma_1, \sigma_0) = (p_2 \times (\sigma \circ p_{13}),\sigma)$$ of the groupoid presentations. I.e.~ 
$$\xymatrix{G \times H\times X \ar[d]_{\sigma_1} \ar@<1ex>[d]^-{p_1} \ar[r]_-{t'} \ar@<1ex>[r]^-{s'} & G\times X \ar[d]_{\sigma_0} \ar@<1ex>[d]^-{p_0}\\ H\times X \ar[r]_-{t} \ar@<1ex>[r]^-{s} & X,}$$ where $s = p_2$, $t = \tau$, $s' = \mathrm{id}_G \times p_2$ and $t' = \mathrm{id}_G \times \tau$. By Proposition \ref{prop:equivalence}, giving a $G$-equivariant quasi-coherent sheaf $\F$ on $\X$ is equivalent to giving a quasi-coherent sheaf $\F_X$ on $X$ together with four isomorphisms  $$\Psi: t^*\F_X \to s^*\F_X, \;\;\; A: t^{\prime *}\sigma_0^*\F_X \to s^{\prime *}\sigma_0^*\F_X,$$ 
$$B: t^{\prime *}p_0^*\F_X \to s^{\prime *}p_0^*\F_X, \;\;\;\Phi:  \sigma_0^*\F_X\to p_0^*\F_X ,$$ such that: 
\begin{enumerate} [i)]
\item $\Psi$ satisfies the cocycle identity, 
\item $A = \sigma_1^*\Psi$, $B = p_1^*\Psi$,  
\item $s^{\prime *}\Phi \circ A =B \circ t^{\prime *} \Phi$,
\item $\Phi$ satisfies the cocycle identity.  
\end{enumerate}
The isomorphism $\Psi$ (with cocycle condition) defines an $H$-equivariant structure on $\F_X$ and hence a quasi-coherent sheaf $\F$ on $\X$. The sheaves $\sigma_0^*\F_X$ and $p_0^*\F_X$ together with the isomorphisms $A$ and $B$ (satisfying condition ii)) correspond to the quasi-coherent sheaves $\bs{\sigma}^*\F$ and $\bs{p}^*\F$ on $G\times \X$. The isomorphism $\Phi$ with condition iii) corresponds to an isomorphism $\bs{\sigma}^*\F \cong \bs{p}^*\F$ on $G \times \X$. It satisfies the cocycle identity if and only if iv) holds. This defines a $G$-equivariant structure on $\F_X$. The proof follows by noting that condition iii) is equivalent to equation (\ref{equ:GHequiv}). 
\end{proof}

Consider the setup of the previous proposition. If in addition $X = \Spec R$ is affine, $G$ and $H$ are diagonalizable with character groups $X(G)$ and $X(H)$, then we get a further equivalence with the category of $R$-modules with $X(G) \times X(H)$-grading as follows. The global section functor $H^0(X,\cdot)$ gives an equivalence between the category of quasi-coherent sheaves on $X$ and the category $R$-modules. If a quasi-coherent sheaf $\F$ has in addition a $G$- and $H$-equivariant structure, then this can be used to define a linear $G$ and $H$-action on $H^0(X,\F)$ in the following way. Let $i_g : X \hookrightarrow G \times X$ be defined by inclusion of an element $g$ into $G$ and set $\Phi_g :=i_g^* \Phi$. Then for any $s \in H^0(X,\F)$, define $$g \cdot s := \Phi_{g^{-1}}((g^{-1})^* s) \in H^0(X,\F),$$ where $(g^{-1})^*s \in H^0(X,(g^{-1})^*\F)$ is the canonical lift (and similarly for $H$). Note that $\Phi$ and $\Psi$ commute, so the $G$- and $H$-action on $H^0(X,\F)$ commute. Since $G$ is diagonalizable, the module $H^0(X,\F)$ decomposes into $G$-eigenspaces 
\[
H^0(X,\F) = \bigoplus_{\chi \in X(G)} H^0(X,\F)_{\chi}.
\]
Each $H^0(X,\F)_m$ has an induced $H$-action. Since $H$ is diagonalizable, we get a further decomposition into eigenspaces
\[
H^0(X,\F) = \bigoplus_{\chi \in X(G)} \bigoplus_{\psi \in X(H)} H^0(X,\F)_{\chi,\psi}.
\]
For $H$ trivial, T.~Kaneyama \cite{Kan} proves $H^0(X,\cdot)$ gives an equivalence between the category of $G$-equivariant quasi-coherent sheaves on $X$ and the category of $X(G)$-graded $R$-modules (see also \cite[Prop.~2.31]{Per3}). It is not hard to show that in our current setting, $H^0(X,\cdot)$ gives an equivalence between the category of quasi-coherent sheaves on $X$ with commuting $G$- and $H$-equivariant structures and the category of $X(G) \times X(H)$-graded $R$-modules. We find it convenient in our applications to refer to the $X(H)$-grading as the \emph{fine grading}. Summarizing:
\begin{corollary} \label{grading}
Let  $\X = [(\Spec R)/H]$ be a quotient DM stack with $G$-action coming from a $G$-action on $\Spec R$ commuting with $H$. Assume $G$ and $H$ are diagonalizable. Then the category of quasi-coherent sheaves on $\X$ with commuting $G$- and $H$-equivariant structures is equivalent to the category of $R$-modules with an $X(G)$-grading and $X(H)$-fine grading. 
\end{corollary}

This corollary is useful because in applications moduli of finitely generated $R$-modules with $X(G)$-grading and $X(H)$-fine grading can be explicitly described.

\section{Toric sheaves on affine toric DM stacks}

Let $\C^d$ be affine space with linear $\T = \C^{*d}$-action. Let $H$ be a finite abelian group. Suppose $\C^d$ also has an $H$-action and the $\T$ and $H$ actions commute. In this section, we describe toric sheaves, i.e.~$\T$-equivariant coherent sheaves, on quotient DM stacks $\X = [\C^d/H]$. This boils down to applying Corollary \ref{grading} and repackaging the grading data somewhat. We also discuss what torsion free and reflexive sheaves look like in this picture. 

Throughout this paper, we make the following notational conventions. Let $D$ be any diagonalizable group. Then $$D \cong \C^{*d} \times \bmu_{a_1} \times \cdots \times \bmu_{a_s},$$ where $\bmu_{a_i}$ is the group of $a_i$-th roots of unity and the $a_i$'s are prime powers. The numbers $d, a_i$ are unique with this property. The group operation of the latter group (and its factors) is written multiplicatively and we denote its elements by $\lambda = (\lambda_1, \ldots, \lambda_d, \mu_1, \ldots, \mu_s)$. We denote the character group of $D$ by $X(D)$. We get an isomorphism
\[
X(D) \cong \Z^d \oplus \Z_{a_1} \oplus \cdots \oplus \Z_{a_s}.
\]
The group operation of the group on the RHS (and its factors) is written additively and its elements are denoted by $m=(\ell_1, \ldots, \ell_d, l_1, \ldots, l_s)$. The isomorphism maps $m=(\ell_1, \ldots, \ell_d, l_1, \ldots, l_s)$ to the character 
\[
\chi(m) : D \longrightarrow \C^{*}, \ (\lambda_1, \ldots, \lambda_d,\mu_1, \ldots, \mu_s) \mapsto \lambda_{1}^{\ell_{1}} \cdots \lambda_{d}^{\ell_{d}} \mu_{1}^{l_{1}} \cdots \mu_{s}^{l_{s}}. 
\]
We often switch between additive notation $m,m', \ldots$, $+$ and multiplicative notation $\chi(m), \chi(m'), \ldots$, $\cdot$ for $X(D)$. With this notation $$\chi(m) \cdot \chi(m') = \chi(m+m').$$

\subsection{Non-degenerate torus actions.} 

We start with some elementary remarks on the action of $\T$ on $\C^d$.  
\begin{definition}
Let $\T$ act linearly\footnote{It is conjectured that \emph{any} $\T$-action on $\C^d$ is linear after a suitable change of coordinates (linearization conjecture). This has been proved for $d\leq 3$ \cite{Gut, KR1, KR2}.} on $\C^d$. Then there exist coordinates $x_1, \ldots, x_d$ and characters $\chi(m_1), \ldots, \chi(m_d)$ such that $\lambda \cdot x_i = \chi(m_i)(\lambda) x_i$ for all $i=1, \ldots, d$ and $\lambda \in \T$. If the $m_i$'s are dependent, then the action is said to be \emph{degenerate}. If the action is non-degenerate and the $m_i$ generate the lattice $X(\T)$, then the action is said to be \emph{primitive}. These notions do not depend on the choice of coordinates. \hfill $\oslash$
\end{definition}
In the case of smooth toric varieties, the opens in the cover $\{U_\sigma\}$ corresponding to the cones $\sigma$ of maximal dimension have primitive linear torus actions. In the case of smooth toric stacks, the opens in the cover $\{[\C^d / N(\sigma)]\}$ corresponding to cones $\sigma$ of maximal dimension can have a torus action induced by a non-degenerate but possible non-primitive linear torus action on $\C^d$. We are interested in such actions. 
\begin{definition} \label{linnondeg}
Let $\T$ act linearly and non-degenerately on $\C^d$. Suppose the action is written as 
\begin{equation} \label{action}
\lambda \cdot x_i = \chi(m_i)(\lambda) x_i.
\end{equation} 
Note that this choice of coordinates is unique\footnote{In the case of toric varieties, such a choice does not have to be made since $\T \subset \C^d$ and the action of $\T$ on $\C^d$ restricts to the canonical action $\T \times \T \rightarrow \T$.} up to scaling and reordering the $x_i$. The \emph{box} associated to the action is the subset $\cB_\T \subset X(\T)$ of all elements of the form $\sum_i q_i m_i \in X(\T)$ with $0 \leq q_1, \ldots, q_d  < 1$ rational. Note that $\cB_\T= 0$ if and only if the $\T$-action is primitive. We also denote an element $\sum_i q_i m_i \in\cB_\T$ by $(q_1, \ldots, q_d) \in \Q^d$. \hfill $\oslash$
\end{definition}
For $\C^d$ with non-degenerate linear $\T$-action, the weight spaces $H^0(\C^d,\O_{\C^d})_m$ are 0 or 1-dimensional. The collection of $m \in X(\T)$ for which $$H^0(\C^d,\O_{\C^d})_m \neq 0$$ forms a semigroup $S \subset X(\T)$. If the action is given as \eqref{action}, then $S$ is the cone spanned by $m_1, \ldots, m_d \in X(\T)$. The category of $\T$-equivariant quasi-coherent sheaves on $\C^d$ is equivalent to the category of $X(\T)$-graded modules (Corollary \ref{grading}). We repackage this data neatly in the same way Perling does in the case of toric varieties \cite{Per1}. 

Given a $\T$-equivariant quasi-coherent sheaf $\F$ on $\C^d$, let $$H^0(\C^d,\F) = \bigoplus_{m \in X(\T)} F(m)$$ be the corresponding $X(\T)$-graded module. This gives us a collection of vector spaces $\{F(m)\}_{m \in X(\T)}$. Multiplication by $x_i$ gives linear maps 
\begin{equation} \label{mult}
\chi_i(m) : F(m) \longrightarrow F(m+m_i),
\end{equation}
for all $m \in X(\T)$. These maps satisfy 
\begin{equation} \label{trans}
\chi_j(m+m_i) \circ \chi_i(m) = \chi_i(m+m_j) \circ \chi_j(m),
\end{equation}
for any $m \in X(\T)$ and $i,j=1, \ldots, d$. Abstractly, we refer to a collection of vector spaces $\{F(m)\}_{m \in X(\T)}$ and linear maps $\{\chi_i(m)\}_{m \in X(\T),i=1,\ldots,d}$ satisfying (\ref{trans}) as an \emph{$S$-family} $\hat{F}$:

\begin{definition}
Let $X(\T)$ be the character group of an algebraic torus $\T$ of dimension $d$. Suppose we are given $\Z$-independent elements $m_1, \ldots, m_d \in X(\T)$. An \emph{$S$-family} $\hat{F}$ consists of the following data: a collection of vector spaces $\{F(m)\}_{m \in X(\T)}$ and linear maps $\{\chi_i(m)\}_{m \in X(\T),i=1,\ldots,d}$ satisfying (\ref{trans}). (The letter $S$ stands for the semigroup given by the cone generated by $m_1, \ldots, m_d \in X(\T)$.) A \emph{morphism of $S$-families} $\hat{F}$, $\hat{G}$ is a family $\hat{\phi}$ of linear maps $$\{\phi(m) : F(m) \longrightarrow G(m)\}_{m \in X(\T)}$$ commuting with the $\chi_i(m)$'s. \hfill $\oslash$
\end{definition}

The $S$-family $\hat{F}$ associated to a $\T$-equivariant quasi-coherent sheaf $\F$ contains all the data of the graded module $H^0(\F)$. It is easy to see that the category of $\T$-equivariant quasi-coherent sheaves on $\C^d$ is equivalent to the category of $S$-families. The following is also obvious. 
\begin{proposition}
Let $\T$ act non-degenerately on $\C^d$. Then each $\T$-equivariant quasi-coherent sheaf $\F$ on $\C^d$ with corresponding $S$-family $\hat{F}$ decomposes equivariantly according to the box elements 
\[
\F \cong \bigoplus_{b \in \cB_\T} {}_{b}\F, \ \hat{F} \cong \bigoplus_{b \in \cB_\T} {}_{b}\hat{F}.
\]
\end{proposition}
\begin{proof}
It suffices to show the statement for $S$-families. Take an element $b \in \cB_\T$ then define $_{b}\hat{F}$ to be the $S$-subfamily consisting of all vector spaces $F(m+b)$ where $m = \sum_i \l_i m_i$ for any $\ell_i \in \Z$. 
\end{proof}

\noindent \textbf{Summary.} A $\T$-equivariant quasi-coherent sheaf on affine space $\C^d$ with non-degenerate linear $\T$-action is just a family of vector spaces indexed by the lattice points of $X(\T)$ and (compatible linear) maps between them encoding the module structure. When the action is non-primitive, the sheaf decomposes according to the box elements $b \in \cB_\T$. \\

We end this section by mentioning what the $S$-family of a coherent, torsion free and reflexive sheaf looks like. Let $\F$ be a $\T$-equivariant quasi-coherent sheaf on $\C^d$ with $S$-family $\hat{F}$. Any $m \in X(\T)$ can be uniquely written as $$b+\sum_{i=1}^{d} \l_i m_i,$$ where $b \in \cB_T$, the $\l_i$ are integers and the $m_i$ are defined in \eqref{action}. We often write $${}_{b}F(\l_1, \ldots, \l_d):={}_{b}F(m).$$ Moreover, the element $b \in \cB_\T$ can be uniquely written as $\sum_i q_i m_i$ for rational numbers $0 \leq q_i <1$ and we also often write $${}_{(q_1, \ldots, q_d)}F(\l_1, \ldots, \l_d) := {}_{b}F(\l_1, \ldots, \l_d).$$ 
Similarly, the maps $\chi_i(b+m)$ (see \eqref{mult}) are denoted by $${}_{b}\chi_{i}(\l_1, \ldots, \l_d), \ {}_{(q_1, \ldots, q_d)}\chi_{i}(\l_1, \ldots, \l_d),$$ 
or simply $x_i$ when we suppress the domain.

The sheaf $\F$ is \emph{coherent} if and only if the corresponding module is finitely generated. In terms of $\hat{F}$ this means the following \cite[Def.~5.10, Prop.~5.11]{Per1}:
\begin{enumerate} [i)]
\item all vector spaces ${}_{b}F(\vec{\l})$ are finite dimensional,
\item for each $b \in \cB_\T$ and $\l_{1}, \ldots, \l_{d}$ sufficiently negative ${}_{b}F(\vec{\l}) = 0$,
\item for each $b \in \cB_\T$, there are only finitely many $\vec{L} \in \Z^d$ for which
\[
{}_{b}F(\vec{L}) \neq \mathrm{span}_{\C}\{x_{1}^{L_1 - \l_1} \cdots x_{d}^{L_d - \l_d} s \ | \ s \in {}_{b}F(\vec{\l}) \ \mathrm{and} \ L_i - \l_i \geq 0 \ \mathrm{not \ all \ zero}\}.
\]
\end{enumerate} 
We refer to $\T$-equivariant coherent sheaves as \emph{toric sheaves}. 

The sheaf $\F$ is \emph{torsion free} if and only if it is coherent and all maps $x_i$ are injective. This is proved in \cite[Prop.~5.13]{Per1} (or \cite[Prop.~2.8]{Koo1}, which also discusses \emph{pure} sheaves of lower dimension). Up to an equivalence of categories, this means we can take all $x_i$'s to be inclusion. Decomposing $\hat{F} \cong \bigoplus_{b \in \cB_\T} {}_{b}\hat{F}$, we conclude that a toric torsion free sheaf on $\C^d$ is specified by the following data:
\begin{enumerate} [i)]
\item finite-dimensional vector spaces denoted by ${}_{b}F(\vec{\infty})$ for all $b \in \cB_\T$, 
\item for each $b \in \cB_\T$: a multifiltration $\{{}_{b}F(\vec{\l})\}_{\vec{\l} \in \Z^d}$ of ${}_{b}F(\vec{\infty})$ such that $${}_{b}F(\vec{\l}) = {}_{b}F(\vec{\infty})$$ for sufficiently large $\vec{\l} \in \Z^d$.
\end{enumerate} 
The rank of such a torsion free sheaf is $$\sum_{b \in \cB_\T} \mathrm{dim}_{\C}({}_{b}F(\vec{\infty})).$$ 

How do we characterize the \emph{reflexive} sheaves among the torsion free sheaves? The following describe the $S$-families corresponding to toric reflexive sheaves. Suppose for each $b \in \cB_\T$ we are given a finite-dimensional vector space ${}_{b}V(\infty)$ (not all zero) and a filtration
\[
\cdots \subset {}_{b}V(\l-1) \subset {}_{b}V(\l) \subset {}_{b}V(\l+1) \subset \cdots,
\]
such that ${}_{b}V(\l) = 0$ for $\l$ sufficiently small and ${}_{b}V(\l) = {}_{b}V(\infty)$ for $\l$ sufficiently large. Set 
\begin{equation}
{}_{b}F(\l_1, \ldots, \l_d) = {}_{b}V(\l_1) \cap \cdots \cap {}_{b}V(\l_d) \subset {}_{b}V(\infty). \label{reflexive}
\end{equation}
Multi-filtrations obtained in this way exactly correspond to the $S$-families of toric reflexive sheaves \cite[Thm.~5.19]{Per1}.

\subsection{Toric sheaves on $[\C^d/H]$.} 

We now describe toric sheaves on the DM stack $[\C^d/H]$, where  $\T = \C^{*d}$ and $H$ is a finite abelian group, both acting on $\C^d$. We assume $\T$ acts linearly and non-degenerately (Definition \ref{linnondeg}) and the actions of $\T$ and $H$ commute. By Corollary \ref{grading}, the category of $\T$-equivariant sheaves on $[\C^d/H]$ is equivalent to the category of $H^0(\C^d,\O_{\C^d})$-modules with $X(\T)$-grading and $X(H)$-fine-grading. 

Next, we choose coordinates $x_i$ on $\C^d$ such that the $\T$-action is given by \eqref{action} $$\lambda \cdot x_i = \chi(m_i)(\lambda) x_i.$$ As mentioned before, this choice of coordinates is unique up to scaling and reordering the $x_i$. A $\T$-equivariant quasi-coherent sheaf $\F$ on $[\C^d/H]$ gives rise to the $H^0(\C^d,\O_{\C^d})$-module $H^0(\C^d,\F)$ with $X(\T)$-grading and $X(H)$-fine-grading. In turn, this provides us with an $S$-family $\hat{F}$ as described in Section 3.1. For each $m \in X(\T)$, the vector space $F(m)$ has an $X(H)$-grading
\[
F(m) = \bigoplus_{n \in X(H)} F(m)_n.
\]
Since the actions of $\T$ and $H$ commute, $H$ acts by $$h \cdot x_i = \chi(n_i)(h) x_i,$$ for a unique $n_i \in X(H)$. Each multiplication map $\chi_i(m)$ maps $F(\vec{\lambda})_n$ to $F(m+m_i)_{n+n_i}$. In other words, we have the following data:
\begin{enumerate} [i)]
\item a collection of vector spaces $\{F(m)_n\}_{m \in X(\T), n \in X(H)}$,
\item a collection of linear maps $$\{\chi_i(m) : F(m) \longrightarrow F(m+m_i)\}_{i=1, \ldots, d, m \in X(\T)}$$ satisfying
\begin{align*}
&\chi_i(m) : F(m)_n \longrightarrow F(m+m_i)_{n+n_i}, \ \chi_j(m+m_i) \circ \chi_i(m) = \chi_i(m+m_j) \circ \chi_j(m),
\end{align*}
for all $i,j = 1, \ldots, d$, $m \in X(\T)$ and $n \in X(H)$.
\end{enumerate}
We refer to this data $\hat{F}$ as a \emph{stacky $S$-family}. There is an obvious notion of morphism between stacky $S$-families: a morphism of $S$-families respecting the fine-grading. This repackaging can be summarized in the following proposition.
\begin{proposition} \label{stackyfam}
Let $\T = \C^{*d}$ and $H$ be a finite abelian group acting on $\C^d$. Assume $\T$ acts linearly and non-degenerately and the actions of $\T$ and $H$ commute. Then the category of $\T$-equivariant quasi-coherent sheaves on $[\C^d/H]$ is equivalent to the category of stacky $S$-families.
\end{proposition}

Note that a stacky $S$-family decomposes according to the elements of the box $\cB_\T$ as in Section 3.1. 
Coherent, torsion free and reflexive $\T$-equivariant quasi-coherent sheaves on $[\C^d/H]$ correspond to stacky $S$-families with underlying $S$-family satisfying the properties described in Section 3.1. We give two examples. \\

\noindent \textbf{Example 1.} Let $\T = \C^{*2}$ act by $(\kappa,\lambda)\cdot (x,y)=(\kappa^2x,\lambda y)$ and $H = \bmu_2$ by $-1 \cdot (x,y)=(-x,-y)$. The box in this example is $\cB_\T = \{(0,0),(1,0)\}$. A rank 1 toric torsion free sheaf $\F$ on $[\C^2/\bmu_2]$ corresponds to the following data. Firstly, $\F = {}_{(0,0)}\F$ or $\F = {}_{(1,0)}\F$, because the rank is 1. In either case, ${}_{b}\F$ is described by specifying a double filtration of $\C$ on $\Z^2$ corresponding to the $S$-family ${}_{b}F(\l_1,\l_2)$. The ``stackiness'' comes in by picking any $(L_1,L_2) \in \Z^2$ such that ${}_{b}F(L_1,L_2) = \C$ and giving this vector space weight $0$ or $1$ with respect to $X(H) = \Z_2$. Since all maps are the identity (or zero), this fixes all the $X(H)$-weights. A typical ${}_{b}\hat{F}$ looks like:
\begin{displaymath}
\xy
(0,0)*{} ; (30,0)*{} **\dir{} ; (0,5)*{} ; (30,5)*{} **\dir{.} ; (0,10)*{} ; (30,10)*{} **\dir{.} ; (0,15)*{} ; (30,15)*{} **\dir{.} ; (0,20)*{} ; (30,20)*{} **\dir{.} ; (0,25)*{} ; (30,25)*{} **\dir{.} ; (0,30)*{} ; (30,30)*{} **\dir{} ;                   (0,0)*{} ; (0,30)*{} **\dir{} ; (5,0)*{} ; (5,30)*{} **\dir{.} ; (10,0)*{} ; (10,30)*{} **\dir{.} ; (15,0)*{} ; (15,30)*{} **\dir{.} ; (20,0)*{} ; (20,30)*{} **\dir{.} ; (25,0)*{} ; (25,30)*{} **\dir{.} ; (30,0)*{} ; (30,30)*{} **\dir{} ;       (10,30)*{} ; (10,20)*{} **\dir{-} ; (10,20)*{} ; (15,20)*{} **\dir{-} ; (15,20)*{} ; (15,15)*{} **\dir{-} ; (15,15)*{} ; (20,15)*{} **\dir{-} ; (20,15)*{} ; (20,5)*{} **\dir{-} ; (20,5)*{} ; (30,5)*{} **\dir{-} ;                 (10,25)*!<-7pt,-7pt>{0} ; (15,25)*!<-7pt,-7pt>{1} ; (20,25)*!<-7pt,-7pt>{0} ; (25,25)*!<-7pt,-7pt>{1} ; (10,20)*!<-7pt,-7pt>{1} ; (15,20)*!<-7pt,-7pt>{0} ; (20,20)*!<-7pt,-7pt>{1} ; (25,20)*!<-7pt,-7pt>{0} ; (15,15)*!<-7pt,-7pt>{1} ; (20,15)*!<-7pt,-7pt>{0} ; (25,15)*!<-7pt,-7pt>{1} ; (20,10)*!<-7pt,-7pt>{1} ; (25,10)*!<-7pt,-7pt>{0} ; (20,5)*!<-7pt,-7pt>{0} ; (25,5)*!<-7pt,-7pt>{1} ; (-20,15)*{{}_{b}\hat{F}}
\endxy 
\end{displaymath}
where the origin $(\l_1,\l_2)=(0,0)$ can be located anywhere. Here the vector spaces associated to the lattice points on or above the solid line are $\C$ and zero otherwise. Moreover, $0,1$ refers to the $\bmu_2$-weight of the vector space associated to the lattice point in the left-bottom corner of the box. \\

\noindent \textbf{Example 2.} Let $[\C^2/\bmu_2]$ be as in the previous example. Let $\F$ be a rank 2 toric torsion free sheaf on $[\C^2/\bmu_2]$. Decompose $$\F = {}_{(0,0)}\F \oplus {}_{(0,1)}\F.$$ There are three possibilities:
\begin{enumerate}
\item[(I)] The sheaves ${}_{(0,0)}\F$ and ${}_{(0,1)}\F$ are both non-zero: then each is a rank 1 toric torsion free sheaf as in example 1.
\item[(II)] Only one ${}_{b}\F$ is non-zero, then $${}_{b}F(L_1,L_2) = {}_{b}F(\infty,\infty) = \C^2$$ for sufficiently large $(L_1,L_2) \in \Z^2$ and this space decomposes as a (nontrivial) sum of a weight $0$ and $1$ space $${}_{b}F(L_1,L_2) = {}_{b}F(L_1,L_2)_0 \oplus {}_{b}F(L_1,L_2)_{1}.$$ The 1-dimensional spaces ${}_{b}F(L_1,L_2)_n$ both have a double filtration as in example 1 and ${}_{b}\F$ decomposes as a sum of two rank 1 toric torsion free sheaves.
\item[(III)] Only one ${}_{b}\F$ is non-zero, then $${}_{b}F(L_1,L_2) = {}_{b}F(\infty,\infty) = \C^2$$ for sufficiently large $(L_1,L_2) \in \Z^2$ and this space has weight $0$ or $1$. This time ${}_{b}F(\l_1,\l_2)$ is a double filtration of $\C^2$ on $\Z^2$ and the $X(H)$-weight of each ${}_{b}F(\l_1,\l_2)$ is fully determined by the $X(H)$-weight of ${}_{b}F(L_1,L_2)$. 
\end{enumerate} 
For a fixed choice of dimensions of the $X(\T)$-weight spaces, there are only a finite number of sheaves of type I and II, but the sheaves of type III have continuous moduli corresponding to different choices of inclusions $\C \subset \C^2$. A typical sheaf of type III looks like:
\begin{displaymath}
\xy
(0,0)*{} ; (30,0)*{} **\dir{} ; (0,5)*{} ; (30,5)*{} **\dir{.} ; (0,10)*{} ; (30,10)*{} **\dir{.} ; (0,15)*{} ; (30,15)*{} **\dir{.} ; (0,20)*{} ; (30,20)*{} **\dir{.} ; (0,25)*{} ; (30,25)*{} **\dir{.} ; (0,30)*{} ; (30,30)*{} **\dir{} ;                   (0,0)*{} ; (0,30)*{} **\dir{} ; (5,0)*{} ; (5,30)*{} **\dir{.} ; (10,0)*{} ; (10,30)*{} **\dir{.} ; (15,0)*{} ; (15,30)*{} **\dir{.} ; (20,0)*{} ; (20,30)*{} **\dir{.} ; (25,0)*{} ; (25,30)*{} **\dir{.} ; (30,0)*{} ; (30,30)*{} **\dir{} ;       (10,30)*{} ; (10,20)*{} **\dir{=} ; (10,20)*{} ; (15,20)*{} **\dir{=} ; (15,20)*{} ; (15,15)*{} **\dir{=} ; (15,15)*{} ; (20,15)*{} **\dir{=} ; (20,15)*{} ; (20,5)*{} **\dir{=} ; (20,5)*{} ; (30,5)*{} **\dir{=} ;                 (10,25)*!<-7pt,-7pt>{0} ; (15,25)*!<-7pt,-7pt>{1} ; (20,25)*!<-7pt,-7pt>{0} ; (25,25)*!<-7pt,-7pt>{1} ; (10,20)*!<-7pt,-7pt>{1} ; (15,20)*!<-7pt,-7pt>{0} ; (20,20)*!<-7pt,-7pt>{1} ; (25,20)*!<-7pt,-7pt>{0} ; (15,15)*!<-7pt,-7pt>{1} ; (20,15)*!<-7pt,-7pt>{0} ; (25,15)*!<-7pt,-7pt>{1} ; (20,10)*!<-7pt,-7pt>{1} ; (25,10)*!<-7pt,-7pt>{0} ; (20,5)*!<-7pt,-7pt>{0} ; (25,5)*!<-7pt,-7pt>{1} ; (-20,15)*{{}_{b}\hat{F}} ;          (0,30)*{} ; (0,25)*{} **\dir{-} ; (0,25)*{} ; (5,25)*{} **\dir{-} ; (5,25)*{} ; (5,20)*{} **\dir{-} ; (5,20)*{} ; (10,20)*{} **\dir{-} ; (10,20)*{} ; (10,10)*{} **\dir{-} ; (10,10)*{} ; (15,10)*{} **\dir{-} ; (15,10)*{} ; (15,5)*{} **\dir{-} ; (15,5)*{} ; (20,5)*{} **\dir{-} ; (25,5)*{} ; (25,0)*{} **\dir{-} ; (25,0)*{} ; (30,0)*{} **\dir{-} ;                    (0,25)*!<-7pt,-7pt>{0} ; (5,25)*!<-7pt,-7pt>{1} ; (5,20)*!<-7pt,-7pt>{0} ; (10,15)*!<-7pt,-7pt>{0} ; (10,10)*!<-7pt,-7pt>{1} ; (15,10)*!<-7pt,-7pt>{0} ; (15,5)*!<-7pt,-7pt>{1} ; (25,0)*!<-7pt,-7pt>{0}
\endxy 
\end{displaymath}
where the vector spaces below the single line are $0$-dimensional, on and above but below the double line $1$-dimensional and on or above the double line $\C^2$. The $0$'s and $1$'s in the diagram denote the $X(H)$-weights. The continuous moduli in this example are $(\PP^{1})^3$, because there are three connected components between the single and double line.

\section{Weighted projective planes} \label{sec:stack P(a,b,c)}

The weighted projective plane $\PP:=\P$ is by definition the quotient stack $[\C^3 \backslash \{0\}/\C^*]$, where $\C^*$ acts on $\C^3$ by $$\lambda \cdot (X,Y,Z)=(\lambda^a X,\lambda^b Y,\lambda^c Z).$$ This is a smooth complete toric DM stack in the sense of \cite{BCS,FMN}. Let $d:=\text{gcd}(a,b,c)$, $d_{12}:=\text{gcd}(a,b)$, $d_{13}:=\text{gcd}(a,c)$ and $d_{23}:=\text{gcd}(b,c)$. The coarse moduli scheme of $\PP$, denoted by $\textbf{P} = \textbf{P}(a,b,c)$, is the weighted projective plane in the classical sense \cite[Sect.~2.2]{Ful}. The toric variety $\textbf{P}(a,b,c)$ (in general singular) is isomorphic to $\textbf{P}(a/d,b/d,c/d)$. If $d_{12}=d_{23}=d_{13}=1$, then the structure map $\PP \to \bf{P}$ is an isomorphism away from those points among $(1:0:0), (0:1:0), (0:0:1) \in \bf{P}$ which are singularities\footnote{In this case $\P$ is called a canonical DM stack (see \cite{FMN}).}. In the case $d=1$, $\P$ is an orbifold meaning that the structure map $\PP \to \bf{P}$ is an isomorphism away from the lines $$(0:Y:Z),(X:0:Z), (X:Y:0) \subset \bf{P}.$$ In general, $\PP$ is a $B\bmu_d$-gerbe over the orbifold $\PP(a/d,b/d,c/d)$.

\subsection{Open cover and torus action} \label{sec:open cover}

$\PP$ can be covered by the standard open substacks $\U_1, \U_2, \U_3$ corresponding to the open sets $$\{X\neq 0 \},\{Y \neq 0\}, \{Z \neq0\} \subset \C^3\backslash \{0\}.$$ Define a map of sets $$\hat{\cdot}:\{1,2,3\} \to \Z_{>0}$$ sending $1 \mapsto a$, $2\mapsto b$ and $3\mapsto c$. By \cite{BCS}, we have an isomorphism of stacks
$$\U_i\cong [\C^2/\bmu_{\i}],$$ where $\mu \in \bmu_{\i}$ acts on $(x,y)\in \C^2$ by 
$$\begin{cases}(\mu^b x,\mu^c y) & \text{if}\;\; i=1, \\ (\mu^a x,\mu^c y) & \text{if} \;\;i=2, \\ (\mu^a x,\mu^b y) & \text{if} \;\; i=3. \end{cases}$$ The open immersion $\U_i \hookrightarrow \PP$ is induced by the natural inclusion $\bmu_{\i} \hookrightarrow \C^*$ as the group of $\i$-th roots of unity and the map $\C^2 \to \C^3\backslash \{0\}$ sending $(x,y)$ to $$\begin{cases}(1,x,y) & \text{if}\;\; i=1, \\ (x,1,y) & \text{if} \;\;i=2, \\ (x,y,1) & \text{if} \;\; i=3. \end{cases}$$ 

We compute the double stack theoretic intersections of the $\U_i$'s by taking the fiber product over $\PP$ via the maps defined above. The outcome is $$\U_{ij}:=\U_i \times_\PP \U_j \cong [\C^*\times \C/\bmu_{\i}\times \bmu_{\j}].$$ The action of $(\mu,\nu) \in \bmu_{\i}\times \bmu_{\j}$ on $(\gamma,z)\in \C^*\times \C$ is given by $$(\nu \mu^{-1} \gamma,\mu^{\k} z),$$ where $k\in \{1,2,3\} \backslash \{i,j\}.$ The open immersions $\U_{ij} \hookrightarrow \U_i$ and $\U_{ij} \hookrightarrow \U_j$ are induced respectively from the projections $$\bmu_{\i}\times \bmu_{\j} \to \bmu_{\i} \;\; \text{  and  } \;\;\bmu_{\i}\times \bmu_{\j} \to \bmu_{\j},$$ and the maps $\C^*\times \C\to  \C \times \C $ that send $(\gamma,z)$ to respectively 

$$\begin{cases}(\gamma^{-b},z) & \text{if}\;\; (i,j)=(1,2), \\ (z,\gamma^{-c} ) & \text{if} \;\;(i,j)=(2,3), \\ (z,\gamma^{-c}) & \text{if} \;\; (i,j)=(1,3), \end{cases} \quad \text{and} \quad \begin{cases}(\gamma^{a},z\gamma^c) & \text{if}\;\; (i,j)=(1,2), \\ (z\gamma^a,\gamma^{b} ) & \text{if} \;\;(i,j)=(2,3), \\ (\gamma^a,z\gamma^{b}) & \text{if} \;\; (i,j)=(1,3). \end{cases}$$

Similarly, one can show $$\U_{123}:=\U_1 \times_\PP \U_2 \times_\PP \U_3 \cong [\C^*\times \C^*/\bmu_a\times \bmu_b \times \bmu_c],$$ where the action of $(\mu,\nu,\xi) \in \bmu_a\times \bmu_b \times \bmu_c $ on $(\kappa,\lambda) \in \C^*\times \C^*$ is given by $$(\nu \mu^{-1} \kappa, \xi\nu^{-1} \lambda).$$ The open immersion $\U_{123} \hookrightarrow \U_{ij}$ is given by the corresponding projection $$\bmu_a\times \bmu_b \times \bmu_c \to \bmu_{\i} \times \bmu_{\j}$$ and the map $\C^*\times \C^* \to \C^*\times \C$ taking $(\kappa, \lambda)$ to $$\begin{cases}(\kappa,\kappa^{-c}\lambda^{-c}) & \text{if}\;\; (i,j)=(1,2), \\ (\lambda,\kappa^{a}  ) & \text{if} \;\;(i,j)=(2,3), \\ (\kappa\lambda,\kappa^{-b}) & \text{if} \;\; (i,j)=(1,3). \end{cases}$$ 

The stack $\PP$ is equipped with a natural action of a 2-dimensional torus obtained by rigidifying its open dense substack $\U_{123}\cong [\C^{*3}/\C^*]$ defined above (see \cite{FMN}). Denote this torus\footnote{In fact $\U_{abc} \cong \T\times B\bmu_d$.} by $\T\cong \C^* \times \C^*$.  One can easily define an action of $\T$ on the stacks $\U_i$ and $\U_{ij}$ such that all the open immersions defined above are $\T$-equivariant. In the table below we summarize the result by showing the corresponding weights of the $\T$-action on $\C\times \C$ for each $\U_i$ and on $\C^*\times \C$ for each $\U_{ij}$ (cf.~Lemma \ref{lem:G-action}):
\begin{table} \label{table}
\begin{center}
\begin{tabular}{|c|c|}
\hline
& $\T$-weights on $\C\times \C$ \\ \hline
$\U_1$ & $(b,0),(0,c)$\\\hline
$\U_2$ & $(-a,0), (-c,c)$\\\hline
$\U_3$ & $(0,-a),(b,-b)$\\\hline
\end{tabular}
\quad \quad
\begin{tabular}{|c|c|}
\hline
& $\T$-weights on $\C^*\times \C$ \\ \hline
$\U_{12}$ & $(-1,0),(0,c)$\\\hline
$\U_{23}$ & $(1,-1),(-a,0)$\\\hline
$\U_{13}$ & $(0,-1),(b,0)$\\\hline
\end{tabular}
\end{center}
\label{default}
\caption{$\T$-action on each $\U_i$ and $\U_{ij}$.}
\end{table}%

\noindent The $\T$-action on $\PP$ has three fixed points $P_1, P_2, P_3$ which correspond to the origins of the open substacks $\U_1, \U_2, \U_3$. 

\subsection{$K$-groups of weighted projective planes} \label{sec:K}
\newcommand{\tch}{\widetilde{\ch}}
\newcommand{\ttd}{\widetilde{\td}}
Denote by $K_0(\PP)$ be the Grothendieck group of coherent sheaves on $\PP$ and let $K_0(\PP)_\Q:=K_0(\PP)\otimes_{\Z} \Q$.
By \cite{BH} $$K_0(\PP)_\Q\cong \Q[g,g^{-1}]/(1-g^a)(1-g^b)(1-g^c).$$ Here $g:=[\O_{\PP}(-1)]$ is a generator of $\text{Pic}(\PP)\cong \Z$ \cite[Ex.~7.27]{FMN}. The classes of the structure sheaves of the fixed points of the $\T$-action are $$[\O_{P_i}]=(1-g^a)(1-g^b)(1-g^c)/(1-g^{\i}). $$We prove this in Proposition \ref{point} as an easy corollary of the general description of toric sheaves on $\PP$ developed in Section 5. We get the following identities in $K_0(\PP)_\Q$ $$[\O_{P_i}]=[\O_{P_i}]g^{\i}, \quad i=1,2,3.$$ The elements $[\O_{P_i}]g^j$ for $i=1,2,3$ and $j=0,\dots, \i-1$ generate the subgroup of $K_0(\PP)_\Q$ of $0$-dimensional coherent sheaves. It is not hard to prove the following relations among these elements:
\begin{align} \label{equ:relations}
\sum_{i=0}^{a/d-1}[\O_{P_1}]g^{id} &=\sum_{j=0}^{b/d-1}[\O_{P_2}]g^{jd}=\sum_{k=0}^{c/d-1}[\O_{P_3}]g^{kd},\\ \nonumber
\sum_{i=0}^{a/d-1}[\O_{P_1}]g^{id+1} &=\sum_{j=0}^{b/d-1}[\O_{P_2}]g^{jd+1}=\sum_{k=0}^{c/d-1}[\O_{P_3}]g^{kd+1}, \\ \nonumber
&\dots \\ \nonumber
\sum_{i=0}^{a/d-1}[\O_{P_1}]g^{id+d-1} &=\sum_{j=0}^{b/d-1}[\O_{P_2}]g^{jd+d-1}=\sum_{k=0}^{c/d-1}[\O_{P_3}]g^{kd+d-1},\\ \nonumber
\sum_{i=0}^{a/d_{12}-1}[\O_{P_1}]g^{id_{12}} &=\sum_{j=0}^{b/d_{12}-1}[\O_{P_2}]g^{jd_{12}} ,\dots \\ \nonumber
\sum_{j=0}^{b/d_{23}-1}[\O_{P_2}]g^{jd_{23}} &=\sum_{k=0}^{c/d_{23}-1}[\O_{P_3}]g^{kd_{23}} ,\dots \\ \nonumber
\sum_{i=0}^{a/d_{13}-1}[\O_{P_1}]g^{id_{13}} &=\sum_{k=0}^{c/d_{13}-1}[\O_{P_3}]g^{kd_{13}} ,\dots \nonumber
\end{align}
For example, the first equality in the first row follows from the following identity in $K_0(\PP)_\Q$  $$(1+g^d+\cdots+g^{a-d})(1-g^b)(1-g^c)=(1-g^a)(1+g^d+\cdots+g^{b-d})(1-g^c).$$ The first equality in the second row is obtained by multiplying both sides of this equation by $g$ etc. 

Let $I\PP$ be the inertia stack of $\PP$ and $\pi: I\PP\longrightarrow \PP$ the natural map. We index the connected components of $I\PP$ by the set $F$ and its elements by $f \in F$. For any coherent sheaf $\F$ on $\PP$, define $$\tch:K_0(\PP)_\Q \to A^*(I\PP)_{\bmu_\infty}, \ \tch(\F):=\sum_{f \in F} \sum_i \omega_{f,i} \cdot \ch(\F_{f,i}),$$ where $\F_{f}$ is the restriction of $\pi^* \F$ to the component corresponding to $f$, $\F_f = \bigoplus_{i} \F_{f,i}$ is its decomposition into eigenvectors and $\omega_{f,i} \in \bmu_\infty$ are the corresponding eigenvalues\footnote{Here $\bmu_\infty$ is the group of roots of unity. We will only be concerned with the case $\F = L$ is a line bundle, in which case each $L_f$ is an eigensheaf.}. For any $f\in F$, we denote by $\tch(\F)_f$ the restriction of $\tch(\F)$ to the component $Z$ of $I\PP$ corresponding to $f$. 

Define $$\ttd: \Pic(\PP) \to A^*(I\PP)_{\bmu_\infty}$$ as follows. Let $L\in \Pic(\PP)$ and let $f \in F$. Denoting $c_1(L_f)=x_f$ and denoting the corresponding eigenvalue of $L_f$ by $\omega_f$, the component $\ttd (L)_f$ is given by\footnote{Here $\delta_{1,\omega_f}$ is the Kronecker delta. This expression reduces to the usual Todd class of a line bundle when $\omega_f=1$.} $$\displaystyle \frac{x^{\delta_{1,\omega_f}}_f}{1-\omega_f e^{-x_f}}.$$ With these definitions one can compute the holomorphic Euler characteristic of $\F$ using To\"en-Riemann-Roch (TRR) \cite{Toe} 
\begin{align*}
\chi(\F) &= 
\int_{I\PP} \tch(\F) \cdot \ttd(\O_\PP(a)) \cdot \ttd(\O_\PP(b)) \cdot \ttd(\O_\PP(c)).
\end{align*}
Define $m:=\operatorname{lcm}(a,b,c)$. For any coherent sheaf $\F$ on $\PP$ and $t \in \Z_{>0}$, define $$P(\F,t):=\chi(\F\otimes \O_\PP(mt))$$ to be the Hilbert polynomial of $F$ with respect to\footnote{$\O_\PP(m)$ is the pull back of $\O_{\bf{P}}(1)$ from the coarse moduli space $\bf{P}$. Recall that $\O_\PP(1)$ was introduced at the beginning of this section.} $\O_\PP(m)$.  
\begin{proposition}\label{prop:hilb poly} For any $r \in \Z$: \begin{enumerate}[i)]
\item  if $d   \nmid r$ then $P(\O_\PP(r),t)=0$, \item  if $d \mid r$ then $P(\O_\PP(r),t)$ modulo the $t$-constant term is given by 
$$\frac{dm^2}{2abc}t^2+m\bigg(\frac{2r+a+b+c}{2abc}d 
+\sum_{1\le i<j \le3}\frac{1}{\i \j}\sum_{\tiny \begin{array}{c}  h\in \{1,\dots,d_{ij}-1\} \\ \frac{d_{ij}}{d} \not | h \end{array}} \frac{ \omega_{ij}^{hr}}{1-\omega_{ij}^{h\k}}  \bigg)t$$ where $\omega_{ij}=e^{2\pi \sqrt{-1} /d_{ij}}$ and $k \in \{1,2,3\} \setminus \{i,j\}$.
\end{enumerate}
\end{proposition}
\begin{proof}
We use the TRR formula above to compute the coefficients of $t^2$ and $t$ in $\chi(\O_\PP(r+mt))$. Only 2 and 1-dimensional components of $I\PP$ contribute to these terms. Let $$D:=\{l/d\}_{l=0,\dots,d-1}.$$ The 2-dimensional components of $I\PP$ are isomorphic to $\PP$  and they are in bijection with the elements of $D$. The eigenvalue of $\O_\PP(1)$ on $\PP_{l/d}$ is $e^{2\pi \sqrt{-1} l/d}$. For $(i,j)=(1,2),(1,3)$ or $(2,3)$, let $$D_{ij}:=\{l/d_{ij}\}_{l=0,\dots, d_{ij}-1} \setminus D.$$ The 1-dimensional components of $I\PP$ are all isomorphic to one of $\PP(a,b)$, $\PP(a,c)$ or $\PP(b,c)$. The components isomorphic to $\PP(\i,\j)$ are in bijection with the elements of $D_{ij}$. The eigenvalue of $\O_\PP(1)$ on $\PP(\i,\j)_{l/d_{ij}}$ is $e^{2\pi \sqrt{-1} l/d_{ij}}$. We find the contribution of each component separately. \\

\noindent \emph{Contribution of $\PP_{l/d}$}: The integrand in the TRR formula over $\PP_{l/d}$ is $$e^{2\pi \sqrt{-1} l(r+mt)/d}\cdot e^{(r+mt)x}\cdot  \frac{ax}{1-e^{-ax}} \cdot \frac{bx}{1-e^{-bx}} \cdot \frac{cx}{1-e^{-cx}},$$ where $x$ is the first Chern class of the pull back of $\O_\PP(1)$ to $I\PP$ and restricted to $\PP_{l/d}$. The integral over $\PP_{l/d}$ modulo the $t$-constant term is $$\frac{e^{2\pi \sqrt{-1} lr/d}}{abc}\left((m^2/2)t^2+(r+a/2+b/2+c/2)mt \right).$$ 

\noindent \emph{Contribution of $\PP(\i,\j)_{l/d_{ij}}$}: The integrand in the TRR formula over $\PP(\i,\j)_{l/d_{ij}}$ is equal to $$e^{2\pi \sqrt{-1}l(r+mt)/d_{ij}}\cdot e^{(r+mt)x}\cdot  \frac{\i x}{1-e^{-\i x}} \cdot \frac{\j x}{1-e^{-\j x}} \cdot \frac{1}{1-e^{2\pi \sqrt{-1}l \k/d_{ij}}\cdot e^{-\k x}},$$ where $x$ is the first Chern class of the pull back of $\O_\PP(1)$ to $\PP(\i,\j)_{l/d_{ij}}$ and $k\in\{1,2,3\}\setminus\{i,j\}$. The integral over $\PP(\i,\j)_{l/d_{ij}}$ modulo the $t$-constant term is equal to $$\frac{me^{2\pi \sqrt{-1}l r/d_{ij}}}{\i \j(1-e^{2\pi \sqrt{-1}l \k/d_{ij}})}t.$$ 
The proposition follows by adding the contributions of all the components above. \end{proof}

\begin{remark}
Since $\O_\PP(m)$ is the pull back of $\O_{\bf{P}}(1)$ from the coarse moduli space, $P(\F,t)$ is equal to the Hilbert polynomial of the push-forward of $\F$ to $\bf{P}$ with respect to $\O_{\bf{P}}(1)$. If $d \nmid r$ then the push forward of $\O_\PP(r)$ to $\bf{P}$ is zero, which gives another proof of the first part of Proposition \ref{prop:hilb poly}.
\end{remark}

\begin{definition} \label{indexsets}
In the proof of Proposition \ref{prop:hilb poly}, we defined the sets $$D:=\{l/d\}_{l=0,\ldots,d-1}, \ D_{ij}:=\{l/d_{ij}\}_{l=0, \ldots, d_{ij}-1} \setminus D,$$ for all $i<j$. For convenience we define $D_{ji}:=D_{ij}$. In addition, we define $$D_i:=\{l/\i\}_{l=0,\dots,\i-1} \setminus (D \cup D_{ij} \cup D_{ik}),$$ for all $\{i,j,k\} = \{1,2,3\}$. The 2-dimensional components of $I\PP$ are all isomorphic to $\PP$ and they are indexed by $D$. The 1-dimensional components of $I\PP$ are isomorphic to $\PP(\i,\j)$ and are index by $D_{ij}$. Finally, the 0-dimensional components of $I\PP$ are isomorphic to $\PP(\i)$ and are indexed by $D_i$. The sets $D$, $D_i$'s and $D_{ij}$'s give a partition of the set of connected components $F$ of $I\PP$. The eigenvalue of $\O_\PP(1)$ pulled back to $I\PP$ and restricted to the component corresponding to $f \in F$ is $e^{2 \pi \sqrt{-1} f}$. \hfill $\oslash$
\end{definition}

\begin{corollary} \label{Plinebundle}
Let $\displaystyle \cE=\oplus_{n=0}^{E-1}\O_\PP(n)$ and assume $d_{ij} \mid E$ for all $1\le i<j \le 3$. Then for any $r \in \Z$, the Hilbert polynomial $P(\O_\PP(r)\otimes \cE,t)$ modulo $t$-constant term is given by $$P(\O_\PP(r)\otimes \cE,t)=\frac{E m^2}{2abc}t^2+\sum_{i=1}^e\frac{2r+2n_i+a+b+c}{2abc}mdt,$$ where\footnote{In particular, if $d=1$ then $e=E=\rk \cE$.} $n_1,\dots, n_e$ are those $n \in \{0,\dots,E-1\}$ for which $d \mid r+n$. 
\end{corollary}

\section{Toric sheaves on weighted projective planes}  

The toric DM stack $\PP := \PP(a,b,c)$ has an action of an algebraic torus $\T = \C^{*2}$ as described in Section 4.1. In this section, we describe toric sheaves on $\PP$. The approach is as follows. We cover $\PP$ by the $\T$-invariant open substacks $\U_i \cong [\C^2/\bmu_{\i}]$ 
The action of $\T$ and $\bmu_{\i}$ on $\C^2$ are described explicitly in Section 4.1. A toric sheaf $\F$ on $\PP$ restricts to a toric sheaf $\F_i$ on $\U_i$ which is (up to equivariant isomorphism) determined by its stacky $S$-family $\hat{F}_i$ (Proposition \ref{stackyfam}). Conversely, given toric sheaves $\F_i$ on $\U_i$, we want to find gluing conditions on the stacky $S$-families $\hat{F}_i$ allowing us to glue the $\F_i$ to a toric sheaf $\F$ on $\PP$. 

Our approach is as follows. Let $\F_i|_{ij}$ be the pull-back of $\F_i$ along the inclusion $\U_{ij} = \U_i \times_{\PP} \U_j \hookrightarrow \U_i$. In Section 4.1, we described these inclusions explicitly at the groupoid level and we noted $\U_{ij} \cong [\C^* \times \C / \bmu_{\i} \times \bmu_{\j}]$. This will allows us to compute the stacky $S$-family $\hat{F}_{i,ij}$ of $\F_i|_{ij}$. Matching $\hat{F}_{i,ij}$ and $\hat{F}_{j,ij}$ for all $i,j$ allows us to glue sheaves (Proposition \ref{glue}).

\subsection{Graded tensor products}

Given a morphism of affine schemes $$f : \Spec S \longrightarrow \Spec R$$ and a quasi-coherent sheaf $\F$ corresponding to an $R$-module $H^0(\F)$ on $\Spec R$, the pull-back $f^* \F$ corresponds to the $S$-module \cite[Prop.~5.2]{Har} 
\begin{equation} \label{pullback}
H^0(f^*\F) \cong H^0(\F) \otimes_{R} S.
\end{equation}
Suppose the affine schemes have an action of an algebraic torus $\T$, $f$ is $\T$-equivariant and $\F$ is a toric sheaf. Then $f^*\F$ is a toric sheaf. The $\T$-equivariant structures of $\F$ and $f^* \F$ correspond to $X(\T)$-gradings on $H^0(\F)$ and $H^0(f^*\F)$. How can the latter grading be described in terms of the gradings of $H^0(\F)$, $R$ and $S$? We follow Perling's discussion \cite[Sect.~3]{Per1}.

Suppose we are given, in general, an abelian group $A$, an $A$-graded ring $R$ and $A$-graded $R$-modules $M$ and $N$. Then $M \otimes_R N$ has a natural $A$-grading as follows \cite[Sect.~3]{Per1}. The degree $a$ part $(M \otimes_R N)_a$ is the abelian subgroup generated by all elements
\[
m \otimes n, \ \mathrm{with \ } m \in M_{a'} \ \mathrm{and \ }  n \in N_{a''} \ \mathrm{such \ that \ } a'+a''=a. 
\]
Next, suppose we are given an $A$-graded ring $R$ and a $B$-graded ring $S$. A morphism of graded rings is a pair $(\phi, \psi)$ with $\phi : R \rightarrow S$ a ring homomorphism and $\psi : A \rightarrow B$ a group homomorphism satisfying $\phi(R_{a}) \subset S_{\psi(a)}$. Given such a morphism of graded rings, $R$ can be given a $B$-grading by 
\[
R_b := \bigoplus_{a \in \psi^{-1}(b)} R_a.
\]
Likewise, any $A$-graded $R$-module $M$ can be a seen as a $B$-graded $R$-module.

Back to the geometric setting. The $\T$-equivariant morphism $f $ induces a morphism of graded rings $$(\phi, \chi) : (R,X(\T)) \longrightarrow (S,X(\T)).$$ Using $\chi$ one can change the $X(\T)$-grading on the $R$-modules $H^0(\F)$ and $R$ as described above. Then $S$ becomes an $X(\T)$-graded $R$-module as well and $H^0(f^*\F)$ inherits a natural $X(\T)$-grading by \eqref{pullback}. \emph{Comment}: in our applications we are always in the easier situation where $\chi$ is the identity map.

\subsection{Gluing conditions}

We are now ready to describe toric sheaves on $\PP$. A toric sheaf $\F_i$ on $\U_i \cong [\mathbb{C}^2/\bmu_{\i}]$ is determined (up to equivariant isomorphism) by its stacky $S$-family $\hat{F}_i$ (Proposition \ref{stackyfam}). We want to compute the stacky $S$-family $\hat{F}_{i,ij}$ of the toric sheaf $\F_i|_{ij}$ on $\U_{ij} \cong [\C^* \times \C / \bmu_{\i} \times \bmu_{\j}]$ for all $i,j$. Matching $\hat{F}_{i,ij}$ and $\hat{F}_{j,ij}$ for all $i,j$ allows us to glue the $\F_i$ to a global toric sheaf $\F$ by Proposition \ref{glue}.

The torus $\T$ has character group $X(\T) = \Z^2$. This torus acts on each $\U_i$ and $\U_{ij}$. We denote the elements of $\T = \C^{*2}$ by $(\kappa, \lambda) \in \T$. For each $\U_i \cong [\C^2/\bmu_{\i}]$, there are coordinates $(x,y)$ on $\C^2$ such that $$(\kappa, \lambda)\cdot (x,y) = (\chi(m_1)(\kappa, \lambda)x,\chi(m_2)(\kappa, \lambda)y),$$ for independent $m_1, m_2 \in X(\T) $. The characters $m_1, m_2$ are unique up to \emph{order} and we choose the following orders\footnote{Compared to Table \ref{table} in Section 4.1, we have changed the order for $i=2$ because of symmetry reasons.} (see Table \ref{table} in Section 4.1)
\begin{align*}
i=1: \ (m_1,m_2) &= ((b,0),(0,c)), \\
i=2: \ (m_1,m_2) &= ((-c,c),(-a,0)), \\
i=3: \ (m_1,m_2) &= ((0,-a),(b,-b)).
\end{align*}
These vectors span the box $\cB_\T$ of the action of $\T$ on $\U_i$. Likewise, for each $\U_{ij} \cong [\C^* \times \C / \bmu_{\i} \times \bmu_{\j}]$ there are coordinates $(\gamma,z)$ on $\C^*\times \C$ such that $$(\kappa,\lambda)\cdot (\gamma,z) = (\chi(m_1)(\kappa,\lambda)\gamma,\chi(m_2)(\kappa,\lambda)z),$$ for independent $m_1, m_2 \in X(\T)$. 

We compute $\hat{F}_{1,12}$ in terms of $\hat{F}_1$. Recall that we denote the various vector spaces corresponding to $\hat{F}_1$ by (see Section 3) $${}_{(i/b,j/c)}F_1(\l_1,\l_2)_l.$$ Here the box $\cB_T = [0,b-1] \times [0,c-1] \subset X(\T)$ and its elements are written by $(i/b,j/c) \in \cB_T$. The inclusion $$\U_{12} \cong [\C^* \times \C / \bmu_a \times \bmu_b] \hookrightarrow \U_1 \cong [\C^2/\bmu_a]$$ corresponds to the \'etale morphism (Section 4.1)
\begin{equation} \label{map}
\C^* \times \C \longrightarrow \C^2, \ (\gamma,z) \mapsto (\gamma^{-b},z) = (x,y),
\end{equation}
which intertwines the $\bmu_a \times \bmu_b$- and $\bmu_a$-actions and the $\T$-actions. These actions are given by
\begin{align*}
(\kappa,\lambda)\cdot(\gamma,z) &= (\kappa^{-1}\gamma,\lambda^c z), \\ 
(\mu,\nu)\cdot(\gamma,z) &= (\mu^{-1}\nu \gamma, \mu^c z), \\ 
(\kappa,\lambda)\cdot(x,y) &= (\kappa^bx,\lambda^cy), \\
\mu\cdot(x,y) &= (\mu^b x, \mu^c y). 
\end{align*}
We denote $a$-th roots of unity by $\mu \in \bmu_a$, $b$-th roots of unite by $\nu \in \bmu_b$ and $c$-th roots of unity by $\xi \in \bmu_c$. 

We first restrict $\F_1$ to $\C^* \times \C \subset \C^2$. For a fixed $\l_2$, the vector spaces ${}_{(i/b,j/c)}F_1(\l_1,\l_2)$ stabilize for $\l_1 \gg 0$. This means they are all isomorphic for $\l_1 \gg 0$. Up to isomorphism of $\hat{F}_1$, we can take all the isomorphisms between these limiting spaces to be the identity map (though these maps in general do swap the fine-grading!). We denote the limit by $${}_{(i/b,j/c)}F_1(\infty,\l_2).$$ The sheaf $\F_1|_{\C^* \times \C}$ has an $S$-family $\hat{G}_1$ which is easily computed using the description of grading on a tensor product of Section 5.1
\[
{}_{(i/b,j/c)}G_1(\l_1,\l_2) = {}_{(i/b,j/c)}F_1(\infty,\l_2).
\]
The maps between these vector spaces are the identity in the $\l_1$-direction and the induced maps in the $\l_2$-direction. This is because multiplication by $x$ is an isomorphism. The fine-grading does depend on $\l_1$ albeit in a trivial way. This means that for $\l_1 \gg 0$ such that ${}_{(i/b,j/c)}F_1(\l_1,\l_2)={}_{(i/b,j/c)}F_1(\infty,\l_2)$, we have $${}_{(i/b,j/c)}G_1(\l_1,\l_2)_l = {}_{(i/b,j/c)}F_1(\l_1,\l_2)_l.$$ The fine-grading on ${}_{(i/b,j/c)}G_1(\l_1,\l_2)$ for any fixed $\l_1$ determines the fine-grading on ${}_{(i/b,j/c)}G_1(\l_1,\l_2)$ for all $\l_1$
\[
{}_{(i/b,j/c)}G_1(\l_1,\l_2)_l = {}_{(i/b,j/c)}G_1(0,\l_2)_{l-\l_1 b} \otimes \mu^{\l_1 b}.
\]
Here $- \otimes \mu^{l}$ stands for tensoring by the 1-dimensional representation of $\bmu_a$ of weight $l \in \Z_a$. 

Next we \emph{define} the fine-grading on the limit vector space ${}_{(i/b,j/c)}F_1(\infty,\l_2)$ by
\begin{equation} \label{finegrad}
{}_{(i/b,j/c)}F_1(\infty,\l_2)_l := {}_{(i/b,j/c)}G_1(0,\l_2)_l.
\end{equation}
The $S$-family $\hat{F}_{1,12}$ at the point $$(i/b + \l_1) m_1 + (j/c + \l_2) m_2 \in X(\T)$$
is equal to
\[
\bigoplus_{0 \leq i' \leq b-1} {}_{(i'/b,j/c)}F_1(\infty,\l_2).
\]
The maps between these vector spaces are the identity in the $\l_1$-direction and the induced maps in the $\l_2$-direction. This expression can be derived as follows. Using the description of grading on a tensor product of Section 5.1 and the explicit formula for the map (\ref{map}), we see that an element of $\hat{F}_{1,12}$ at $$(i/b + \l_1) m_1 + (j/c + \l_2) m_2 \in X(\T)$$
can be uniquely written in the form
\[
\bigoplus_{0 \leq i' \leq b-1} v_{i'} \otimes \gamma^{i'-i},
\]
where
\[
v_{i'} \in {}_{(i'/b,j/c)}G_1(\l_1,\l_2).
\]
Here the notation $- \otimes \gamma$ comes from the following. The module $F_1 = H^0(\F_1)$ is a finitely generated $\C[x,y]$-module. First, it gets restricted to $\C^* \times \C$, which makes it into a $\C[x^\pm,y]$-module $G_1$. Then it gets pulled-back along the \'etale morphism (\ref{map}), which gives a $\C[\gamma^\pm,z]$-module $G_1 \otimes \C[\gamma^\pm,z]$. The homogeneous elements $v_{i'} \otimes \gamma^{{i'}-i}$ are elements of this module.

Next we determine the fine-grading. Assuming $v_{i'}$ is homogenous of weight $l+i'-i \in \Z_a$ with respect to the action of $\mu_a$, the weight of $v_{i'} \otimes \gamma^{i'-i}$ with respect to the action of $\bmu_a \times \bmu_b$ is $$(l,i'-i) \in \Z_a \oplus \Z_b.$$ Therefore, the fine-grading on the weight space of $\hat{F}_{1,12}$ at the point $$(i/b + \l_1) m_1 + (j/c + \l_2) m_2 \in X(\T)$$
is equal to
\[
\bigoplus_{\tiny{\begin{array}{c} 0 \leq i' \leq b-1 \\ l \in \Z_a \end{array}}} {}_{(i'/b,j/c)} G_1(\l_1,\l_2)_{l+i'-i} \otimes \mu^{i-i'} \otimes \nu^{i'-i},
\]
or, using the definition of the fine-grading \eqref{finegrad},
\[
\bigoplus_{\tiny{\begin{array}{c} 0 \leq i' \leq b-1 \\ l \in \Z_a \end{array}}} {}_{(i'/b,j/c)} F_1(\infty,\l_2)_{l+i'-i} \otimes \mu^{i-i'+\l_1 b} \otimes \nu^{i'-i}.
\]
Here $- \otimes \mu^{l}$ stands for tensoring with the one-dimensional representation of $\bmu_a$ of weight $l \in \Z_a$ and $- \otimes \nu^{l}$ stands for tensoring with the one-dimensional representation of $\bmu_b$ of weight $l \in \Z_b$. 
\begin{theorem}[Gluing conditions] \label{main}
The category of toric sheaves on $\PP$ is equivalent to the category of triples $\{\hat{F}_i\}_{i=1,2,3}$ of stacky finite $S$-families on $\U_i \cong[\C^2/\bmu_{\i}]$ satisfying the following gluing conditions: 
\begin{align*}
\bigoplus_{\tiny{\begin{array}{c} 0 \leq i \leq b-1 \\ l \in \Z_a \end{array}}} \!\!\! {}_{(i/b,j/c)}F_1(\infty,\l)_{l} \otimes \mu^{-i} \otimes \nu^{i}  &= \!\!\! \bigoplus_{\tiny{\begin{array}{c} 0 \leq i \leq a-1 \\ l \in \Z_b \end{array}}} \!\!\! {}_{(j/c,i/a)}F_2(\l, \infty)_{l} \otimes \nu^{-i-j-\l c} \otimes \mu^{i+j+\l c}, \\
\bigoplus_{\tiny{\begin{array}{c} 0 \leq i \leq c-1 \\ l \in \Z_b \end{array}}} \!\!\! {}_{(i/c,j/a)}F_2(\infty,\l)_{l} \otimes \nu^{-i} \otimes \xi^{i}  &= \!\!\! \bigoplus_{\tiny{\begin{array}{c} 0 \leq i \leq b-1 \\ l \in \Z_c \end{array}}} \!\!\! {}_{(j/a,i/b)}F_3(\l, \infty)_{l} \otimes \xi^{-i-j-\l a} \otimes \nu^{i+j+\l a}, \\
\bigoplus_{\tiny{\begin{array}{c} 0 \leq i \leq a-1 \\ l \in \Z_c \end{array}}} \!\!\! {}_{(i/a,j/b)}F_3(\infty,\l)_{l} \otimes \xi^{-i} \otimes \mu^{i}  &= \!\!\! \bigoplus_{\tiny{\begin{array}{c} 0 \leq i \leq c-1 \\ l \in \Z_a \end{array}}} \!\!\! {}_{(j/b,i/c)}F_1(\l, \infty)_{l} \otimes \mu^{-i-j-\l b} \otimes \xi^{i+j+\l b}, 
\end{align*}
for all $\l \in \Z$ and similar gluing conditions between the linear maps.
\end{theorem}

\begin{proof}
The category of toric sheaves on $\PP$ is equivalent to the category of triples $\{\F_i\}_{i=1,2,3}$ of toric sheaves on $\U_i \cong [\C^2/\bmu_{\i}]$ satisfying the gluing conditions of Proposition \ref{glue}. In turn, a toric sheaf $\F_i$ is described (up to equivariant isomorphism) by its stacky $S$-family $\hat{F}_i$ (Proposition \ref{stackyfam}). Given toric sheaves $\{\F_i\}_{i=1,2,3}$ on $\U_i \cong[\C^2/\bmu_{\i}]$, consider the toric sheaves $\F_i|_{ij}$ on $$\U_{ij} \cong [\C^* \times \C / \bmu_{\i} \times \bmu_{\j}]$$ and compute the stacky $S$-family $\hat{F}_{i,ij}$ corresponding to $\F_i|_{ij}$. 

The case $\hat{F}_{1,12}$ was worked out above. At a point $(0,j+c\l) \in X(\T) = \Z^2$ we found its stacky $S$-family is given by
\[
\bigoplus_{\tiny{\begin{array}{c} 0 \leq i \leq b-1 \\ l \in \Z_a \end{array}}} {}_{(i/b,j/c)} F_1(\infty,\l)_{l} \otimes \mu^{-i} \otimes \nu^{i}.
\]
Similarly, the stacky $S$-family of $\hat{F}_{2,12}$ at $(-j-\l c,j+\l c) \in X(\T) = \Z^2$ is given by 
\[
\bigoplus_{\tiny{\begin{array}{c} 0 \leq i \leq a-1 \\ l \in \Z_b \end{array}}} {}_{(j/c,i/a)} F_2(\l, \infty)_{l} \otimes \nu^{-i} \otimes \mu^{i}.
\]
We cannot directly ``equate'' these two answers. This is because the first vector spaces live on the line $(0,j+c\l)$ and the second vector spaces live on the line $(-j-\l c,j+\l c) \in X(\T) = \Z^2$. Translating the last vector spaces to the line $(0,j+\l c) \in X(\T) = \Z^2$ gives the first equation. The other equations go similar. 

A priori the equalities in the theorem should be isomorphisms of graded vector spaces satisfying the cocycle identity. However, up to isomorphism (and hence up to equivalence of categories), we can choose them to be identities. \end{proof}

\section{Examples of toric sheaves on weighted projective planes}

In this section, we give some examples of toric sheaves on $\PP := \PP(a,b,c)$ as described by Theorem \ref{main}. The focus will be on rank 1 and 2 toric torsion free and locally free sheaves. \\

\noindent \textbf{Example 1}. Let $\F$ be a rank $1$ toric torsion free sheaf on $\PP^2$. In each chart $\U_i \cong \C^2$, $\F$ is described by its $S$-family $\hat{F}_i$ (Section 3)
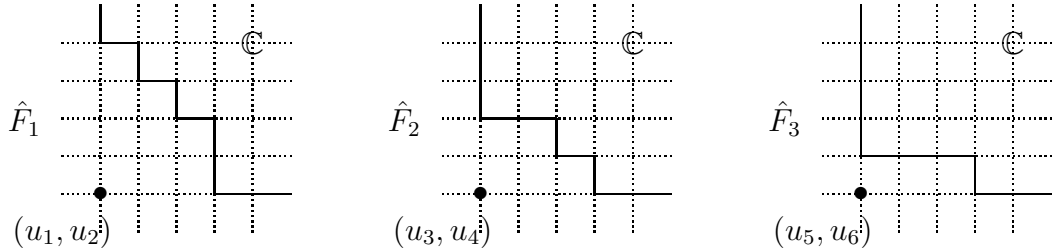
\begin{displaymath}
\xy
(0,0)*{} ; (30,0)*{} **\dir{} ; (0,5)*{} ; (30,5)*{} **\dir{.} ; (0,10)*{} ; (30,10)*{} **\dir{.} ; (0,15)*{} ; (30,15)*{} **\dir{.} ; (0,20)*{} ; (30,20)*{} **\dir{.} ; (0,25)*{} ; (30,25)*{} **\dir{.} ; (0,30)*{} ; (30,30)*{} **\dir{} ;                   (0,0)*{} ; (0,30)*{} **\dir{} ; (5,0)*{} ; (5,30)*{} **\dir{.} ; (10,0)*{} ; (10,30)*{} **\dir{.} ; (15,0)*{} ; (15,30)*{} **\dir{.} ; (20,0)*{} ; (20,30)*{} **\dir{.} ; (25,0)*{} ; (25,30)*{} **\dir{.} ; (30,0)*{} ; (30,30)*{} **\dir{} ; (5,30)*{} ; (5,25)*{} **\dir{-} ;  (5,25)*{} ; (10,25)*{} **\dir{-} ; (10,25)*{} ; (10,20)*{} **\dir{-} ; (10,20)*{} ; (15,20)*{} **\dir{-} ; (15,20)*{} ; (15,15)*{} **\dir{-} ; (15,15)*{} ; (20,15)*{} **\dir{-} ; (20,15)*{} ; (20,5)*{} **\dir{-} ; (20,5)*{} ; (30,5)*{} **\dir{-} ;  (-5,15)*{\hat{F}_1} ; (5,5)*{\bullet} ; (0,0)*{(u_1,u_2)} ;  (25,25)*{\mathbb{C}} ;                                                                                                        ; (50,0)*{} ; (80,0)*{} **\dir{} ; (50,5)*{} ; (80,5)*{} **\dir{.} ; (50,10)*{} ; (80,10)*{} **\dir{.} ; (50,15)*{} ; (80,15)*{} **\dir{.} ; (50,20)*{} ; (80,20)*{} **\dir{.} ; (50,25)*{} ; (80,25)*{} **\dir{.} ; (50,30)*{} ; (80,30)*{} **\dir{} ;                   (50,0)*{} ; (50,30)*{} **\dir{} ; (55,0)*{} ; (55,30)*{} **\dir{.} ; (60,0)*{} ; (60,30)*{} **\dir{.} ; (65,0)*{} ; (65,30)*{} **\dir{.} ; (70,0)*{} ; (70,30)*{} **\dir{.} ; (75,0)*{} ; (75,30)*{} **\dir{.} ; (80,0)*{} ; (80,30)*{} **\dir{} ;  (55,30)*{} ; (55,15)*{} **\dir{-} ;  (55,15)*{} ; (65,15)*{} **\dir{-} ; (65,15)*{} ; (65,10)*{} **\dir{-} ; (65,10)*{} ; (70,10)*{} **\dir{-} ; (70,10)*{} ; (70,5)*{} **\dir{-} ; (70,5)*{} ; (80,5)*{} **\dir{-} ;  (45,15)*{\hat{F}_2} ; (55,5)*{\bullet} ; (50,0)*{(u_3,u_4)} ;  (75,25)*{\mathbb{C}} ;                                                                                                                                                                                                                                   ; (100,0)*{} ; (130,0)*{} **\dir{} ; (100,5)*{} ; (130,5)*{} **\dir{.} ; (100,10)*{} ; (130,10)*{} **\dir{.} ; (100,15)*{} ; (130,15)*{} **\dir{.} ; (100,20)*{} ; (130,20)*{} **\dir{.} ; (100,25)*{} ; (130,25)*{} **\dir{.} ; (100,30)*{} ; (130,30)*{} **\dir{} ;                   (100,0)*{} ; (100,30)*{} **\dir{} ; (105,0)*{} ; (105,30)*{} **\dir{.} ; (110,0)*{} ; (110,30)*{} **\dir{.} ; (115,0)*{} ; (115,30)*{} **\dir{.} ; (120,0)*{} ; (120,30)*{} **\dir{.} ; (125,0)*{} ; (125,30)*{} **\dir{.} ; (130,0)*{} ; (130,30)*{} **\dir{} ;  (105,30)*{} ; (105,10)*{} **\dir{-} ; (105,10)*{} ; (120,10)*{} **\dir{-} ;  (120,10)*{} ; (120,5)*{} **\dir{-} ; (120,5)*{} ; (130,5)*{} **\dir{-} ;              (95,15)*{\hat{F}_3} ; (105,5)*{\bullet} ; (100,0)*{(u_5,u_6)} ;  (125,25)*{\mathbb{C}}                                                                                                                                 
\endxy 
\end{displaymath}
where on and above the solid line all weight spaces are $\C$ and all maps are the identity and below the solid line all weight spaces and maps are $0$. The solid dot indicates an element $(u_j, u_k) \in \Z^2$ and determines the position of the partition. According to Theorem \ref{main}, this data to a toric sheaf if and only if 
\[
u_2 = u_3, \ u_4 = u_5, \ u_6 = u_1.  
\]

\noindent \textbf{Example 2}. Let $\F$ be a rank 1 toric torsion free sheaf on $\PP(1,1,2)$. The charts $\U_1$ and $\U_2$ have a box of size $2$, whereas $\U_3$ has a trivial box. Since the rank is 1, for each $i=1,2,3$ there is only one $b_i \in \cB_\T$ for which the stacky $S$-family ${}_{b_i}\hat{F}_i$ is non-zero (Section 3). A priori this gives $2\cdot 2 = 4$ choices of $b_i$'s. The gluing conditions of Theorem \ref{main} reduce this to two choices
\[
b_1 = (0,i/2), \ b_2 = (i/2,0), \ b_3 = (0,0), \ i \in \{0,1\}.
\]
Fix one of these two choices. On each chart $\U_i$, the $S$-family ${}_{b_i}\hat{F}_i$ is given by the same data as in the previous example. Here we temporarily ignore the fact that the weight spaces in $\U_3$ carry a $\bmu_2$-action which we address in a moment. This local data is described by three partitions and six integers $u_1, \ldots, u_6 \in \Z$. The gluing conditions of Theorem \ref{main} impose $u_2 = u_3, \ u_4 = u_5, \ u_6 = u_1$ as in the previous example. On the chart $\U_3$, each weight space carries a $\bmu_2$-action. Fixing the weight of one non-zero weight space determines the weights of all other weight spaces. Since $\bmu_2$ acts by $(-1,-1)$, the weights alternate as in Example 1 of Section 3. For fixed partitions and integers $u_i$ as above, this a priori gives two choices  of $\bmu_2$-actions. However, the gluing conditions of Theorem \ref{main} allow precisely one choice! \\

\noindent \textbf{Example 3}. Let $\F$ be a rank 1 toric torsion free sheaf on $\PP(2,2,2)$. This time, each chart $\U_i$ has a box of size $2\cdot 2=4$. As before, for each $i=1,2,3$ there is only one $b_i \in \cB_\T$ for which ${}_{b_i}\hat{F}_i$ is non-zero (Section 3). A priori this gives $4\cdot 4\cdot 4 = 64$ choices of $b_i$'s. The gluing conditions of Theorem \ref{main} reduce this to eight choices
\[
b_1 = (i/2,j/2), \ b_2 = (j/2,k/2), \ b_3 = (k/2,i/2), \ i,j,k \in \{0,1\}.
\]
Fix any of these eight choices. On each chart $\U_i$, the $S$-family ${}_{b_i}\hat{F}_i$ is given by the same data as in the previous two examples (again, first ignoring the $\bmu_2$-actions). This local data  is described by three partitions and six integers $u_1, \ldots, u_6 \in \Z$. The gluing conditions of Theorem \ref{main} impose $u_2 = u_3, \ u_4 = u_5, \ u_6 = u_1$ as before. On each chart $\U_i$, the weight spaces carry a $\bmu_2$-action. Since the finite groups $\bmu_2$ act trivially, all weight spaces on a fixed chart $\U_i$ have the same character. For fixed partitions and integers $u_i$, this a priori gives $2\cdot2\cdot2=8$ choices  of $\bmu_2$-actions. However, the gluing conditions of Theorem \ref{main} again allow precisely one choice! \\

From Examples 1,2 and 3 the general picture of rank 1 toric torsion free sheaves on any $\PP$ is easily obtained. We start by describing toric line bundles. A toric line bundle $L$ is locally described by three stacky $S$-families $\hat{L}_i$ and for each $i=1,2,3$ there is only one $b_i \in \cB_\T$ for which ${}_{b_i}\hat{L}_{i}$ is non-zero. A priori this gives $bc+ac+ba$ choices, but the gluing conditions of Theorem \ref{main} reduce this to one of the following $abc$ choices
\begin{align*}
&b_1=(i/b,j/c), \ b_2 = (j/c,k/a), \ b_3=(k/a,i/b), \\ 
&0 \leq i \leq b-1, \ 0 \leq j \leq c-1, \ 0 \leq k \leq a-1. 
\end{align*}
Fix any such choice $i,j,k$. On each chart $\U_i$, the $S$-family ${}_{b_i}\hat{L}_i$ is fully determined by the weight of its unique homogeneous generator, i.e.~by two integers. Here we ignore the $\bmu_{\hat{i}}$-actions for the moment. A priori this gives six integers to choose, but the gluing conditions of Theorem \ref{main} reduce this to three. In particular, there are unique integers\footnote{In Examples 1--3 of this section $I,J,K$ are respectively denoted by $u_1,u_3,u_5$.} $I,J,K$ such that the homogeneous generators have torus weight
\begin{align*}
(i/b+I) m_1 + (j/c+J)m_2 &\in X(\T), \ \mathrm{where} \ (m_1,m_2) := ((b,0),(0,c)) \\
(j/c+J)m_1 + (k/a+K)m_2 &\in X(\T), \ \mathrm{where} \ (m_1,m_2) := ((-c,c),(-a,0)) \\
(k/a+K)m_1 + (i/b+I)m_2 &\in X(\T), \ \mathrm{where} \ (m_1,m_2) := ((0,-a),(b,-b)).
\end{align*}
Recall that in each chart $\U_i$ we use the ``basis'' $(m_1,m_2)$ indicated above (Section 5.2).
Defining 
\begin{equation} \label{ui}
u_1 = i+Ib, \ u_2 = j+Jc, \ u_3 = k+Ka,
\end{equation}
we see that the data so far is uniquely determined by the specification of three integers $u_1,u_2,u_3$. In other words, specifying $u_1,u_2,u_3 \in \Z$ determines the positions of the homogeneous generators in the charts $\U_1, \U_2,\U_3$ according to the above recipe. So far we have ignored the $\bmu_{\hat{i}}$-actions. For fixed $u_1,u_2,u_3 \in \Z$, there are unique finite group weights of the homogeneous generators for which the data glues to a toric line bundle according to Theorem \ref{main}, namely
 \begin{align*}
u_1+u_2+u_3 &\in \Z_a \ \mathrm{on} \ \U_1, \\
u_1+u_2+u_3 &\in \Z_b \ \mathrm{on} \ \U_2, \\
u_1+u_2+u_3 &\in \Z_c \ \mathrm{on} \ \U_3.
 \end{align*}
We denote the toric line bundle corresponding to $(u_1,u_2,u_3) \in \Z^3$ by $L_{(u_1,u_2,u_3)}$ and its stacky $S$-families by $\hat{L}_{(u_1,u_2,u_3),i}$. From the description of the grading on a tensor product (Section 5.1), it is easy to see that
\[
L_{(u_1,u_2,u_3)} \otimes L_{(u_{1}',u_{2}',u_{3}')} \cong L_{(u_1+u_{1}',u_2+u_{2}',u_3+u_{3}')}.
\]
\begin{proposition} \label{eqpic}
Let $\Pic^\T(\PP)$ be the equivariant Picard group of $\PP$. Then 
\[
(u_1,u_2,u_3) \in \Z^3 \mapsto L_{(u_1,u_2,u_3)} \in \Pic^\T(\PP)
\]
is a group isomorphism.
\end{proposition}
Recall that the (non-equivariant) Picard group of $\PP$ is $\Z$ \cite[Ex.~7.27]{FMN}. A rank 1 toric torsion free sheaf on $\PP$ is described by the same data as a toric line bundle on $\PP$ except that on each chart $\U_i$ a partition is cut out in the corner as in Examples 1,2 and 3. We now turn to rank 2. \\

\noindent \textbf{Example 4}. Let $\F$ be a rank 2 toric torsion free sheaf on $\PP^2$  (see also \cite{Kly2} or \cite{Koo2}). In each chart $\U_i \cong \C^2$, $\F$ is described by a double filtration of $\C^2$ (Section 3): 
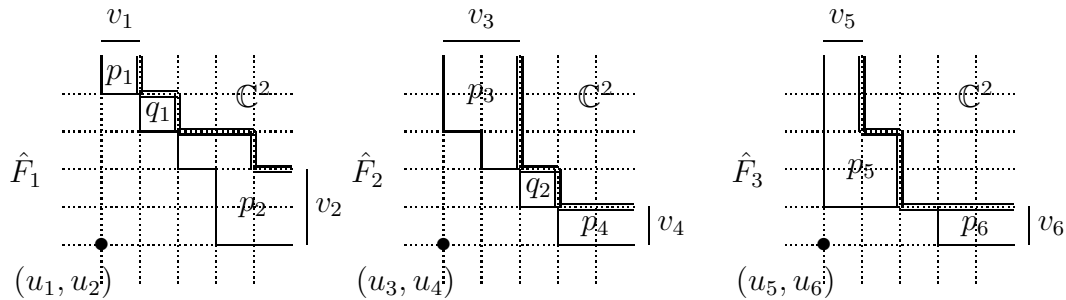
\begin{displaymath}
\xy
(0,0)*{} ; (30,0)*{} **\dir{} ; (0,5)*{} ; (30,5)*{} **\dir{.} ; (0,10)*{} ; (30,10)*{} **\dir{.} ; (0,15)*{} ; (30,15)*{} **\dir{.} ; (0,20)*{} ; (30,20)*{} **\dir{.} ; (0,25)*{} ; (30,25)*{} **\dir{.} ; (0,30)*{} ; (30,30)*{} **\dir{} ;                   (0,0)*{} ; (0,30)*{} **\dir{} ; (5,0)*{} ; (5,30)*{} **\dir{.} ; (10,0)*{} ; (10,30)*{} **\dir{.} ; (15,0)*{} ; (15,30)*{} **\dir{.} ; (20,0)*{} ; (20,30)*{} **\dir{.} ; (25,0)*{} ; (25,30)*{} **\dir{.} ; (30,0)*{} ; (30,30)*{} **\dir{} ; (5,30)*{} ; (5,25)*{} **\dir{-} ;  (5,25)*{} ; (10,25)*{} **\dir{-} ; (10,25)*{} ; (10,20)*{} **\dir{-} ; (10,20)*{} ; (15,20)*{} **\dir{-} ; (15,20)*{} ; (15,15)*{} **\dir{-} ; (15,15)*{} ; (20,15)*{} **\dir{-} ; (20,15)*{} ; (20,5)*{} **\dir{-} ; (20,5)*{} ; (30,5)*{} **\dir{-} ;  (-5,15)*{\hat{F}_1} ; (5,5)*{\bullet} ; (0,0)*{(u_1,u_2)} ;   (10,30)*{} ; (10,25)*{} **\dir{=} ; (10,25)*{} ; (15,25)*{} **\dir{=} ; (15,25)*{} ; (15,20)*{} **\dir{=} ; (15,20)*{} ; (25,20)*{} **\dir{=} ;     (25,20)*{} ; (25,15)*{} **\dir{=} ; (25,15)*{} ; (30,15)*{} **\dir{=}  ; (7.5,27.5)*{p_1} ;  (25,10)*{p_2} ;  (7.5,35)*{v_1} ; (35,10)*{v_2} ; (12.5,22.5)*{q_1} ;  (25,25)*{\mathbb{C}^2} ; (5,32)*{} ; (10,32)*{} **\dir{-} ; (32,15)*{} ; (32,5)*{} **\dir{-} ;                                                                                            ; (45,0)*{} ; (75,0)*{} **\dir{} ; (45,5)*{} ; (75,5)*{} **\dir{.} ; (45,10)*{} ; (75,10)*{} **\dir{.} ; (45,15)*{} ; (75,15)*{} **\dir{.} ; (45,20)*{} ; (75,20)*{} **\dir{.} ; (45,25)*{} ; (75,25)*{} **\dir{.} ; (45,30)*{} ; (75,30)*{} **\dir{} ;                   (45,0)*{} ; (45,30)*{} **\dir{} ; (50,0)*{} ; (50,30)*{} **\dir{.} ; (55,0)*{} ; (55,30)*{} **\dir{.} ; (60,0)*{} ; (60,30)*{} **\dir{.} ; (65,0)*{} ; (65,30)*{} **\dir{.} ; (70,0)*{} ; (70,30)*{} **\dir{.} ; (75,0)*{} ; (75,30)*{} **\dir{} ;  (50,30)*{} ; (50,20)*{} **\dir{-} ; (50,20)*{} ; (55,20)*{} **\dir{-} ; (55,20)*{} ; (55,15)*{} **\dir{-} ; (55,15)*{} ; (60,15)*{} **\dir{-} ; (60,15)*{} ; (60,10)*{} **\dir{-} ; (60,10)*{} ; (65,10)*{} **\dir{-} ; (65,10)*{} ; (65,5)*{} **\dir{-} ; (65,5)*{} ; (75,5)*{} **\dir{-} ;  (40,15)*{\hat{F}_2} ; (50,5)*{\bullet} ; (45,0)*{(u_3,u_4)} ;  (60,30)*{} ; (60,15)*{} **\dir{=} ;  (60,15)*{} ; (65,15)*{} **\dir{=} ; (65,15)*{} ; (65,10)*{} **\dir{=} ; (65,10)*{}; (75,10)*{} **\dir{=} ;  (55,25)*{p_3} ;  (62.5,12.5)*{q_2} ; (70,7.5)*{p_4};  (55,35)*{v_3} ; (80,7.5)*{v_4} ;  (70,25)*{\mathbb{C}^2} ;  (50,32)*{} ; (60,32)*{} **\dir{-} ; (77,10)*{} ; (77,5)*{} **\dir{-} ;                                                                                                                                                                                                                        ; (95,0)*{} ; (125,0)*{} **\dir{} ; (95,5)*{} ; (125,5)*{} **\dir{.} ; (95,10)*{} ; (125,10)*{} **\dir{.} ; (95,15)*{} ; (125,15)*{} **\dir{.} ; (95,20)*{} ; (125,20)*{} **\dir{.} ; (95,25)*{} ; (125,25)*{} **\dir{.} ; (95,30)*{} ; (125,30)*{} **\dir{} ;                   (95,0)*{} ; (95,30)*{} **\dir{} ; (100,0)*{} ; (100,30)*{} **\dir{.} ; (105,0)*{} ; (105,30)*{} **\dir{.} ; (110,0)*{} ; (110,30)*{} **\dir{.} ; (115,0)*{} ; (115,30)*{} **\dir{.} ; (120,0)*{} ; (120,30)*{} **\dir{.} ; (125,0)*{} ; (125,30)*{} **\dir{} ;  (100,30)*{} ; (100,10)*{} **\dir{-} ; (100,10)*{} ; (110,10)*{} **\dir{-} ;  (115,10)*{} ; (115,5)*{} **\dir{-} ; (115,5)*{} ; (125,5)*{} **\dir{-} ;              (90,15)*{\hat{F}_3} ; (100,5)*{\bullet} ; (95,0)*{(u_5,u_6)} ; (105,30)*{} ; (105,20)*{} **\dir{=} ;  (105,20)*{} ; (110,20)*{} **\dir{=} ; (110,20)*{} ; (110,10)*{} **\dir{=} ;  (110,10)*{} ; (125,10)*{} **\dir{=} ; (105,15)*{p_5} ; (120,7.5)*{p_6} ;  (102.5,35)*{v_5} ; (130,7.5)*{v_6} ; (120,25)*{\mathbb{C}^{2}} ; (100,32)*{} ; (105,32)*{} **\dir{-} ; (127,10)*{} ; (127,5)*{} **\dir{-} ;                                                                                                                          
\endxy 
\end{displaymath}
Here $u_1, \ldots, u_6 \in \Z$ fix the position of the filtrations and $v_1, \ldots, v_6 \in \Z_{\geq 0}$ are the widths of the indicated regions. Furthermore, $p_1, \ldots, p_6, q_1, q_2 \subset \C^2$ are $1$-dimensional linear subspaces. In each region, the vector spaces on the associated lattice points are as indicated and all maps are inclusions. Also, the lattice points below the single solid line have 0 as their associated vector space. According to Theorem \ref{main}, this glues to a toric sheaf if and only if
\begin{align*}
&u_2 = u_3, \ u_4 = u_5, \ u_6 = u_1, \ v_2 = v_3, \ v_4 = v_5, \ v_6 = v_1, \\
&p_2 = p_3 \subset \C^2, \ p_4 = p_5 \subset \C^2, \ p_6 = p_1 \subset \C^2.
\end{align*}
Note that the vector spaces $q_1, q_2$ are not involved in the gluing conditions. A general rank 2 toric torsion free sheaf can have many more (or none) $q_i$'s sitting in the corners. Also, the widths $v_i$ are allowed to be 0 in which case the corresponding $p_i$'s do not appear in the gluing conditions. Rank 2 toric locally free sheaves are fully determined by integers $u_1,u_2,u_3 \in \Z$, $v_1,v_2,v_3 \in \Z_{\geq 0}$ and a point $(p_1,p_2,p_3) \in (\PP^1)^3$ (Section 3, Equation (\ref{reflexive})). Abbreviating $p_{ij}:=p_i \cap p_j \subset \C^2$, a typical rank 2 toric locally free sheaf looks like:
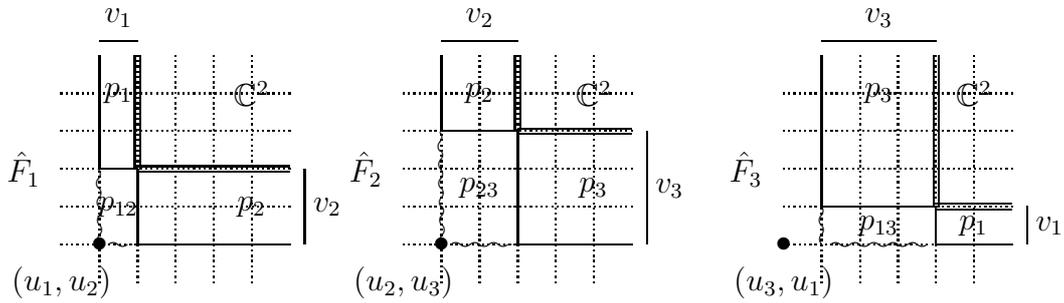
\begin{displaymath}
\xy
(0,0)*{} ; (30,0)*{} **\dir{} ; (0,5)*{} ; (30,5)*{} **\dir{.} ; (0,10)*{} ; (30,10)*{} **\dir{.} ; (0,15)*{} ; (30,15)*{} **\dir{.} ; (0,20)*{} ; (30,20)*{} **\dir{.} ; (0,25)*{} ; (30,25)*{} **\dir{.} ; (0,30)*{} ; (30,30)*{} **\dir{} ;                   (0,0)*{} ; (0,30)*{} **\dir{} ; (5,0)*{} ; (5,30)*{} **\dir{.} ; (10,0)*{} ; (10,30)*{} **\dir{.} ; (15,0)*{} ; (15,30)*{} **\dir{.} ; (20,0)*{} ; (20,30)*{} **\dir{.} ; (25,0)*{} ; (25,30)*{} **\dir{.} ; (30,0)*{} ; (30,30)*{} **\dir{} ;  (-5,15)*{\hat{F}_1} ; (5,5)*{\bullet} ; (0,0)*{(u_1,u_2)} ;  (5,30)*{} ; (5,15)*{} **\dir{-} ; (5,15)*{} ; (5,5)*{} **\dir{~} ;  (5,5)*{} ; (10,5)*{} **\dir{~} ; (10,5)*{} ; (30,5)*{} **\dir{-} ; (10,30)*{} ; (10,15)*{} **\dir{=} ; (10,15)*{} ; (10,5)*{} **\dir{-} ; (5,15)*{} ; (10,15)*{} **\dir{-} ; (10,15)*{} ; (30,15)*{} **\dir{=} ;  (7.5,25)*{p_1} ; (7.5,10)*{p_{12}} ;  (25,10)*{p_2} ;  (7.5,35)*{v_1} ;  (35,10)*{v_2} ;  (25,25)*{\mathbb{C}^{2}} ; (5,32)*{} ; (10,32)*{} **\dir{-} ; (32,15)*{} ; (32,5)*{} **\dir{-} ;                                                                                         (45,0)*{} ; (75,0)*{} **\dir{} ; (45,5)*{} ; (75,5)*{} **\dir{.} ; (45,10)*{} ; (75,10)*{} **\dir{.} ; (45,15)*{} ; (75,15)*{} **\dir{.} ; (45,20)*{} ; (75,20)*{} **\dir{.} ; (45,25)*{} ; (75,25)*{} **\dir{.} ; (45,30)*{} ; (75,30)*{} **\dir{} ;  (45,0)*{} ; (45,30)*{} **\dir{} ; (50,0)*{} ; (50,30)*{} **\dir{.} ; (55,0)*{} ; (55,30)*{} **\dir{.} ; (60,0)*{} ; (60,30)*{} **\dir{.} ; (65,0)*{} ; (65,30)*{} **\dir{.} ; (70,0)*{} ; (70,30)*{} **\dir{.} ; (75,0)*{} ; (75,30)*{} **\dir{} ;  (40,15)*{\hat{F}_2} ; (50,5)*{\bullet} ; (45,0)*{(u_2,u_3)} ;   (50,30)*{} ; (50,20)*{} **\dir{-} ; (50,20)*{} ; (50,5)*{} **\dir{~} ; (50,5)*{} ; (60,5)*{} **\dir{~} ; (60,5)*{} ; (75,5)*{} **\dir{-} ; (60,30)*{} ; (60,20)*{} **\dir{=} ; (60,20)*{} ; (60,5)*{} **\dir{-} ; (50,20)*{} ; (60,20)*{} **\dir{-} ; (60,20)*{} ; (75,20)*{} **\dir{=} ;  (55,25)*{p_2} ; (55,12.5)*{p_{23}} ;  (70,12.5)*{p_3} ;  (55,35)*{v_2} ;  (80,12.5)*{v_3}  ; (70,25)*{\mathbb{C}^{2}} ; (50,32)*{} ; (60,32)*{} **\dir{-} ; (77,20)*{} ; (77,5)*{} **\dir{-} ;                                                                                                                                                                                                                                  ; (95,0)*{} ; (125,0)*{} **\dir{} ; (95,5)*{} ; (125,5)*{} **\dir{.} ; (95,10)*{} ; (125,10)*{} **\dir{.} ; (95,15)*{} ; (125,15)*{} **\dir{.} ; (95,20)*{} ; (125,20)*{} **\dir{.} ; (95,25)*{} ; (125,25)*{} **\dir{.} ; (95,30)*{} ; (125,30)*{} **\dir{} ;                   (95,0)*{} ; (95,30)*{} **\dir{} ; (100,0)*{} ; (100,30)*{} **\dir{.} ; (105,0)*{} ; (105,30)*{} **\dir{.} ; (110,0)*{} ; (110,30)*{} **\dir{.} ; (115,0)*{} ; (115,30)*{} **\dir{.} ; (120,0)*{} ; (120,30)*{} **\dir{.} ; (125,0)*{} ; (125,30)*{} **\dir{} ;   (90,15)*{\hat{F}_3} ; (95,5)*{\bullet} ; (95,0)*{(u_3,u_1)} ;    (100,30)*{} ; (100,10)*{} **\dir{-} ; (100,10)*{} ; (100,5)*{} **\dir{~} ; (100,5)*{} ; (115,5)*{} **\dir{~} ; (115,5)*{} ; (125,5)*{} **\dir{-} ; (115,30)*{} ; (115,10)*{} **\dir{=} ; (115,10)*{} ; (115,5)*{} **\dir{-} ;  (100,10)*{} ; (115,10)*{} **\dir{-} ; (115,10)*{} ; (125,10)*{} **\dir{=} ;  (107.5,25)*{p_3} ; (107.5,7.5)*{p_{13}} ;  (120,7.5)*{p_1} ;  (107.5,35)*{v_3} ;  (130,7.5)*{v_1} ; (120,25)*{\mathbb{C}^{2}} ; (100,32)*{} ; (115,32)*{} **\dir{-} ; (127,10)*{} ; (127,5)*{} **\dir{-} ;                                                                                                                               
\endxy 
\end{displaymath}
Here the regions labelled by $p_{ij}$ either have associated vector space $0$ (if $p_i \neq p_j$) or $p_i=p_j \subset \C^2$ on the lattice points. (So the wiggly line is either not present or a single solid line.) Moreover, the widths $v_i$ are allowed to be 0, in which case the corresponding $p_i$'s do not appear. \\

\noindent \textbf{Example 5}. Let $\F$ be a rank 2 toric torsion free sheaf on $\PP(1,1,2)$. We focus on the case $\F$ is locally free. The description of rank 2 toric torsion free sheaves can be deduced by the reader. The stacky $S$-families decompose according to the box elements $$\hat{F}_1 = {}_{(0,0)} \hat{F}_1 \oplus {}_{(0,1/2)} \hat{F}_1, \ \hat{F}_2 = {}_{(0,0)} \hat{F}_2 \oplus {}_{(1/2,0)} \hat{F}_2, \ \hat{F}_3 = {}_{(0,0)} \hat{F}_3.$$ Since the rank is 2, either each summand is a rank 1 toric line bundle or one summand is zero and the other is a rank 2 toric locally free sheaf. We distinguish the following cases: 
 \begin{enumerate}
 \item [(I)] One box summand of $\hat{F}_i$ is zero for both $i=1,2$. Let ${}_{b_i}\hat{F}_i \neq 0$ then
 \[
 b_1 = (0,i/2), \ b_2=(i/2,0), \ b_3=(0,0), \ i \in\{0,1\},
 \] 
so there are two choices. Fix the $b_i$'s. Disregarding the $\bmu_2$-actions for the moment, the $S$-families $\hat{F}_i$ are exactly as in the case of toric locally free sheaves in the previous example, i.e.~determined by integers $u_1,u_2,u_3 \in \Z$, $v_1,v_2,v_3 \in \Z_{\geq 0}$ and a point $(p_1,p_2,p_3) \in (\PP^1)^3$. Again, the widths $v_i$ are allowed to be 0, in which case the corresponding $p_i$'s do not appear. For any such fixed choice, there is a unique choice of $\bmu_2$-actions on the weight spaces of $\hat{F}_3$ making this into a stacky $S$-family satisfying the gluing conditions of Theorem \ref{main}. For this choice, the limiting vector spaces $\C^2$ do \emph{not} decompose as the sum of 1-dimensional spaces with different weights. 
 \item [(II)] None of the ${}_{b_i}\hat{F}_i$ is zero. Start by ignoring the $\bmu_2$-actions. The $S$-families ${}_{(0,0)}\hat{F}_{1}$ and ${}_{(0,0)}\hat{F}_{2}$ look like
 
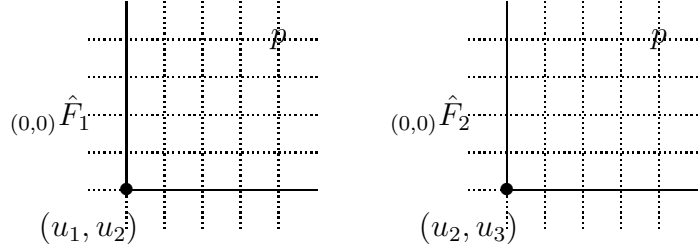
\begin{displaymath}
\xy
(0,0)*{} ; (30,0)*{} **\dir{} ; (0,5)*{} ; (30,5)*{} **\dir{.} ; (0,10)*{} ; (30,10)*{} **\dir{.} ; (0,15)*{} ; (30,15)*{} **\dir{.} ; (0,20)*{} ; (30,20)*{} **\dir{.} ; (0,25)*{} ; (30,25)*{} **\dir{.} ; (0,30)*{} ; (30,30)*{} **\dir{} ;                   (0,0)*{} ; (0,30)*{} **\dir{} ; (5,0)*{} ; (5,30)*{} **\dir{.} ; (10,0)*{} ; (10,30)*{} **\dir{.} ; (15,0)*{} ; (15,30)*{} **\dir{.} ; (20,0)*{} ; (20,30)*{} **\dir{.} ; (25,0)*{} ; (25,30)*{} **\dir{.} ; (30,0)*{} ; (30,30)*{} **\dir{} ;  (-5,15)*{{}_{(0,0)}\hat{F}_1} ; (5,5)*{\bullet} ; (0,0)*{(u_1,u_2)} ;  (5,30)*{} ; (5,5)*{} **\dir{-} ;  (5,5)*{} ; (30,5)*{} **\dir{-} ;  (25,25)*{p} ;                                                                                         (50,0)*{} ; (80,0)*{} **\dir{} ; (50,5)*{} ; (80,5)*{} **\dir{.} ; (50,10)*{} ; (80,10)*{} **\dir{.} ; (50,15)*{} ; (80,15)*{} **\dir{.} ; (50,20)*{} ; (80,20)*{} **\dir{.} ; (50,25)*{} ; (80,25)*{} **\dir{.} ; (50,30)*{} ; (80,30)*{} **\dir{} ;  (50,0)*{} ; (50,30)*{} **\dir{} ; (55,0)*{} ; (55,30)*{} **\dir{.} ; (60,0)*{} ; (60,30)*{} **\dir{.} ; (65,0)*{} ; (65,30)*{} **\dir{.} ; (70,0)*{} ; (70,30)*{} **\dir{.} ; (75,0)*{} ; (75,30)*{} **\dir{.} ; (80,0)*{} ; (80,30)*{} **\dir{} ;  (45,15)*{{}_{(0,0)}\hat{F}_2} ; (55,5)*{\bullet} ; (50,0)*{(u_2,u_3)} ;   (55,30)*{} ; (55,5)*{} **\dir{-} ;  (55,5)*{} ; (80,5)*{} **\dir{-} ;  (75,25)*{p}           
\endxy 
\end{displaymath}
where $p$ is a 1-dimensional vector space and $u_1,u_2,u_3 \in \Z$. Here we have already glued ${}_{(0,0)}\hat{F}_{1}$ and ${}_{(0,0)}\hat{F}_{2}$  along $\U_{12}$. A similar picture holds for ${}_{(0,1/2)}\hat{F}_1$ and ${}_{(1/2,0)}\hat{F}_2$ replacing $p,u_i$ by $p', u_{i}'$. It is easy to work out what the rest of the gluing conditions for this locally free sheaf are. We do one case. Assume $u_1 < u_{1}'$, $u_2 > u_{2}'$ and $u_3>u_{3}'$. In order to satisfy the gluing conditions (still ignoring the $\bmu_2$-actions), the vector spaces $p,p'$ must be distinct linear subspaces of $\C^2$ (so $\C^2 = p \oplus p'$) and the $S$-family $\hat{F}_3$ must look as follows:  
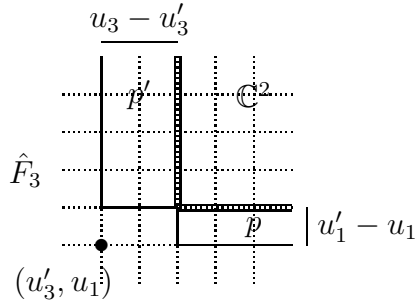
\begin{displaymath}
\xy
(100,0)*{} ; (130,0)*{} **\dir{} ; (100,5)*{} ; (130,5)*{} **\dir{.} ; (100,10)*{} ; (130,10)*{} **\dir{.} ; (100,15)*{} ; (130,15)*{} **\dir{.} ; (100,20)*{} ; (130,20)*{} **\dir{.} ; (100,25)*{} ; (130,25)*{} **\dir{.} ; (100,30)*{} ; (130,30)*{} **\dir{} ;                   (100,0)*{} ; (100,30)*{} **\dir{} ; (105,0)*{} ; (105,30)*{} **\dir{.} ; (110,0)*{} ; (110,30)*{} **\dir{.} ; (115,0)*{} ; (115,30)*{} **\dir{.} ; (120,0)*{} ; (120,30)*{} **\dir{.} ; (125,0)*{} ; (125,30)*{} **\dir{.} ; (130,0)*{} ; (130,30)*{} **\dir{} ;   (95,15)*{\hat{F}_3} ; (105,5)*{\bullet} ; (100,0)*{(u_{3}',u_1)} ;    (105,30)*{} ; (105,10)*{} **\dir{-} ;  (105,5)*{} ; (130,5)*{} **\dir{} ; (115,30)*{} ; (115,10)*{} **\dir{=} ; (115,10)*{} ; (130,10)*{} **\dir{=} ; (115,10)*{} ; (115,5)*{} **\dir{-} ;  (105,10)*{} ; (115,10)*{} **\dir{-} ;   (110,35)*{u_3-u_{3}'} ;  (140,7.5)*{u_{1}'-u_1} ; (115,5)*{} ; (130,5)*{} **\dir{-} ;  (125,7.5)*{p} ; (110,25)*{p'} ;(125,25)*{\C^2}  ; (105,32)*{} ; (115,32)*{} **\dir{-} ; (132,10)*{} ; (132,5)*{} **\dir{-} ;                                                                                                       
\endxy 
\end{displaymath}
From Theorem \ref{main} one can see that there is a unique choice of $\bmu_2$-actions on the weight spaces which makes this into a rank 2 toric locally free sheaf on $\PP(1,1,2)$. Such toric locally free sheaves are always equivariantly \emph{decomposable} as a sum of two toric line bundles. Even when ignoring the $\bmu_2$-actions, there is no analog to this kind of toric locally free sheaf on $\PP^2$. All other type II cases go similar and are also equivariantly decomposable as a direct sum of two toric line bundles. \\
\end{enumerate}

\noindent \textbf{Example 6}. Let $\F$ be a rank 2 toric torsion free sheaf on $\PP(2,2,2)$. Again, we focus on the case $\F$ is locally free and leave the torsion free case to the reader. Let $\hat{F}_i$ be the stacky $S$-families. There are three cases:
\begin{enumerate}
\item[(I)] Exactly one box summand of $\hat{F}_i$ is non-zero for each $i=1,2,3$. A priori there are $4\cdot4\cdot4=64$ choices of $b_i$'s, but only $4\cdot2\cdot1=8$ of them satisfy the gluing conditions of Theorem \ref{main}. The rest of this case is exactly as in Example 5.I. In particular, the limiting vector spaces $\C^2$ do not decompose into 1-dimensional spaces with different $\bmu_2$-weights. 
\item[(II)] There is only one $i=1,2,3$ for which only one box summand of $\hat{F}_i$ is non-zero. Say without loss of generality $i=3$.  A priori there are $6\cdot6\cdot4=144$ choices of $b_i$'s, but only four of them satisfy the gluing conditions of Theorem \ref{main}. The rest is as in Example 5.II. In particular, $\F$ is globally the direct sum of two toric line bundles. 
\item[(III)] For each $i=1,2,3$ two box summands of $\hat{F}_i$ are non-zero. Again, $\F$ is globally the direct sum of two toric line bundles. \\
\end{enumerate}

\begin{remark} Example 6 readily generalizes to a classification of rank 2 toric locally free sheaves on any $\PP$. In particular, we can always distinguish types I, II, III. Types II and III are always equivariantly decomposable as a direct sum of two toric line bundles whereas type I could be equivariantly indecomposable. Rank 2 toric torsion free sheaves are obtained from this description by putting 1-dimensional vector spaces $q_i$ ``in the corners'' much like in Example 4. 

The reader might have noticed that for all rank 1 and 2 toric torsion free sheaves on $\PP$, we first glued sheaves according to Theorem \ref{main} whilst completely ignoring the finite group actions. This is much simpler. Then we noted there exist \emph{unique} finite group actions on the weight spaces which glue this into a toric sheaf. We expect this to be true for toric torsion free sheaves of any rank $r>0$. However, this is not true for general toric coherent sheaves. For example, take $\PP(1,1,2)$ and take the zero sheaf on charts $\U_1, \U_2$ and put a single $\C$ on an arbitrary lattice point of $X(\T)$ associated to $\U_3$ and zeros elsewhere. In this case there are no gluing conditions and we can endow $\C$ with either weight $0$ or $1$ with respect to the $\bmu_2$-action. 
\end{remark}

\section{Rank 1 and 2 torsion free sheaves on weighted projective planes} \label{sec:rank 1}

Let $\X$ be a smooth proper DM stack. Let $\pi : \X \rightarrow X$ be the map to the coarse moduli scheme and assume $X$ is projective. Fix a polarization $\O_X(1)$ on $X$ and a \emph{generating sheaf} $\cE$ on $\X$. A generating sheaf on $\X$ is a $\pi$-very ample locally free sheaf on $\X$. This notion was introduced by M.~Olsson and J.~Starr \cite{OS} (see also \cite[Sect.~2]{Nir} for a discussion). In \cite{Nir}, Nironi constructs the moduli space $M$ of Gieseker stable (w.r.t.~$\O_X(1)$ and $\cE$) sheaves on $\X$ with fixed \emph{modified} Hilbert polynomial. Here the modified Hilbert polynomial of a coherent sheaf $\F$ on $\X$ is the polynomial
\[
P_{\O_X(1),\cE}(\F,t) = \chi(\X, \pi^* \O_X(t) \otimes \cE^* \otimes \F).
\]
The modified Hilbert polynomial is used to define Gieseker stability in the usual way. The reason $\cE$ is used is because otherwise $$\chi(\X,\pi^*\O_X(t) \otimes \F) = \chi(X,\O_X(t) \otimes \pi_*\F)$$ only ``sees'' the coarse moduli scheme $X$. The moduli space $M$ is in fact constructed using GIT and it is a quasi-projective $\C$-scheme \cite[Sect.~6]{Nir}. In the case there are no strictly Gieseker semistable sheaves $M$ is projective.

Let $\PP := \PP(a,b,c)$ as before, let $\O_\mathbf{P}(1)$ be any polarization on the coarse moduli scheme $\mathbf{P}$ and let $\cE$ be any generating sheaf on $\PP$. For any class $\rc \in K_0(\PP)_\Q$ of a torsion free sheaf on $\PP$, define $M_\cE(\rc)$ to be the moduli scheme of $\mu$-stable sheaves on $\PP$ with class $c$. Note that there are two differences with Nironi \cite{Nir}. Firstly, we fix $K$-group class (Section \ref{sec:K}) instead of modified Hilbert polynomial. This is a more refined topological invariant (see also the discussion in the introduction of \cite{Nir}). Secondly, we use $\mu$- instead of Gieseker stability. Here we mean $\mu$-stability defined using the modified Hilbert polynomial, i.e.~the linear term of the modified Hilbert polynomial divided by the quadratic term of the modified Hilbert polynomial. The moduli scheme $M_\cE(\rc)$ does not depend on the choice of $\O_\mathbf{P}(1)$ but does (in general) depend on the choice of generating sheaf $\cE$. (For polarizations on $\mathbf{P}$ and a general survey on weighted projective spaces viewed as toric varieties, see \cite{RT}.)

Note that we consider moduli spaces of $\mu$-stable sheaves \emph{only} and we do not consider strictly $\mu$-semistable sheaves. For $K$-group classes $\rc$ for which the associated modified Hilbert polynomial satisfies certain numerical constraints on the coefficients, one can ensure there are no strictly $\mu$-semistables around. E.g.~see examples 1--4 at the end of Section 7.2. \\

\noindent \textbf{Convention.} Throughout this section, we fix the standard polarization $\O_\mathbf{P}(1)$ on $\mathbf{P}$ and we choose a generating sheaf $\cE := \bigoplus_{n=0}^{E-1} \O_\PP(-n)$, where $E$ is any positive integer such that the least common multiple $m$ of $a,b,c$ divides $E$. Note that this $\cE$ admits equivariant structures. This is not the most general choice of generating sheaf on $\PP$, but we are getting the nicest formulae in this case. Most of the time, we suppress the dependence on  $\O_\mathbf{P}(1)$ and $\cE$.

\subsection{Rank 1}

In this section, $\rc$ denotes the class of a rank 1 torsion free sheaf on $\PP$ and we denote the corresponding moduli space by $M(\rc)$. In the rank 1 case $M(\rc)$ is independent of the choice of polarization as well as generating sheaf because $\mu$-stability is automatic. We are interested in computing the topological Euler characteristics $e(M(\rc))$, where we let $\rc \in K_0(\PP)_\Q$ run over all classes of rank 1 torsion free sheaves on $\PP$ with fixed first Chern class $\beta \in A^1(\PP)$
$$\rG_\beta(q)=\sum_{c_1(\rc)=\beta}e(M(\rc))q^\rc.$$ Here $q$ is a formal variable keeping track of the class $\rc$ and $A^*(\PP)$ denotes the Chow ring of $\PP$ as defined by A.~Vistoli \cite{Vis}.

In practice, $q$ stands for the following variables. We introduce the formal variables $p_0,\dots,p_{a-1}$, $q_0,\dots,q_{b-1}$ and $r_0,\dots,r_{c-1}$ respectively keeping tracks of the classes  $$[\O_{P_1}],\dots, [\O_{P_1}]g^{a-1}, [\O_{P_2}],\dots, [\O_{P_2}]g^{b-1} \;\text{and}\; [\O_{P_3}],\dots, [\O_{P_3}]g^{c-1}$$ in $K_0(\PP)_\Q$ defined in Section \ref{sec:K}. Referring to (\ref{equ:relations}), we impose the following relations among these variables:
\begin{align} \label{equ:varrel} p_0 p_d \cdots p_{a-d}&=q_0 q_d \cdots q_{b-d}=r_0 r_d \cdots r_{c-d},\\ \nonumber
p_1 p_{d+1} \cdots p_{a-d+1}&=q_1 q_{d+1} \cdots q_{b-d+1}=r_1 r_{d+1} \cdots r_{c-d+1}, \\ \nonumber
&\dots \\ \nonumber
p_{d-1} p_{2d-1} \cdots p_{a-1}&=q_{d-1} q_{2d-1} \cdots q_{b-1}=r_{d-1} r_{2d-1} \cdots r_{c-1},\\ \nonumber
p_0 p_{d_{12}} \cdots p_{a-d_{12}}&=q_0 q_{d_{12}} \cdots q_{b-d_{12}}, \dots \\ \nonumber
p_0 p_{d_{13}} \cdots p_{a-d_{13}}&=r_0 r_{d_{13}} \cdots r_{c-d_{13}}, \dots\\\nonumber
q_0q_{d_{23}} \cdots q_{b-d_{23}}&=r_{0} r_{d_{23}} \cdots r_{c-d_{23}}, \dots \nonumber \end{align}

As a toric sheaf, $\O_{P_1}$ is described by stacky $S$-families $\hat{F}_1$, $\hat{F}_2$, $\hat{F}_3$, where $\hat{F}_2=\hat{F}_3=0$ and $\hat{F}_1$ consists of a single vector space $\C$ with trivial $\bmu_a$-action at $(0,0) \in X(\T) = \Z^2$ and zeros everywhere else. The toric sheaves $\O_{P_2}$, $\O_{P_3}$ have similar descriptions replacing $1$ by $2$ and $1$ by $3$ respectively.
Denote by $$[\O_{P_1} \otimes \mu], \ [\O_{P_2} \otimes \nu], \ [\O_{P_3} \otimes \xi]$$ the generators of the Grothendieck groups of the residue gerbes at the points $P_1$, $P_2$ and $P_3$, i.e.~$K(B\bmu_{\hat{i}})\cong \operatorname{Rep}(\bmu_{\hat{i}})$ for $i=1,2,3$. We use the same notation for the push-forwards of these classes to $K_0(\PP)_\Q$.
\begin{proposition} \label{point}
In $K_0(\PP)_\Q$ we have equalities
\begin{align*}
[\O_{P_1} \otimes \mu^{i}] &= (1-g^b)(1-g^c) g^i, \\
[\O_{P_2} \otimes \nu^{j}] &= (1-g^a)(1-g^c) g^j, \\
[\O_{P_3} \otimes \xi^{k}] &= (1-g^a)(1-g^b)g^k,
\end{align*}
where $g:=[\O_{\PP}(-1)]$.
\end{proposition}
\begin{proof}
We compute $[\O_{P_1} \otimes \mu^{i}]$ and leave the other cases to the reader. Recall the characterization of toric line bundles $L_{(u_1,u_2,u_3)}$, $u_1,u_2,u_3 \in \Z$ of Section 6. Note that $L_{(1,0,0)}$, $L_{(0,1,0)}$, $L_{(0,0,1)}$ generate the equivariant Picard group (Proposition \ref{eqpic}) and their $K$-group classes are $g = [\O_\PP(-1)]$. 
Using the description of toric line bundles and $\O_{P_1} \otimes \mu^i$ in terms of stacky $S$-families (Section 6), one readily constructs a $\T$-equivariant short exact sequence 
\[
0 \longrightarrow L_{(b,c,0)} \longrightarrow L_{(b,0,0)} \oplus L_{(0,c,0)} \longrightarrow L_{(0,0,0)} \longrightarrow \O_{P_1} \longrightarrow 0,
\]
where the first map is $v \mapsto (v,v)$ and the second map is $(v,w) \mapsto v-w$. Since $\O_{P_1} \otimes \mu^{i} \cong \O_{P_1 }\otimes L_{(0,0,i)}$ (Section 6), we obtain 
\[
[\O_{P_1} \otimes \mu^{i}] = L_{(0,0,i)} - L_{(b,0,i)} - L_{(0,c,i)} + L_{(b,c,i)} = g^i - g^{b+i} - g^{c+i} + g^{b+c+i}.
\]
\end{proof} 

Slightly more general than Proposition \ref{point}, one can consider a stacky $S$-family with $\hat{F}_2=\hat{F}_3=0$ and $\hat{F}_1$ of the following form: put a single vector space $\C$ with representation $\mu^\alpha$ on the lattice point $$(i/b+I)m_1 + (j/c+J)m_2 \in X(\T),$$
where we recall that on chart $\U_1$ we use ``basis'' $(m_1,m_2)=((b,0),(0,c))$ (Section 5.2). The corresponding toric sheaf is easily seen to be
\[
\O_{P_1} \otimes L_{(i+Ib,j+Jc,\alpha-i-Ib-j-Jc)},
\]
so its $K$-group class is $(1-g^b)(1-g^c)g^{\alpha}$. Similarly on the other charts $\U_2$, $\U_3$.

Let $\F$ be any rank 1 toric torsion free sheaf on $\PP$. The reflexive hull of $\F$ is a toric line bundle $L_{(u_1,u_2,u_3)}$. This is clear from the toric picture of Section 6. The cokernel
\[
0 \longrightarrow \F \longrightarrow L_{(u_1,u_2,u_3)} \longrightarrow \cQ \longrightarrow 0
\]
is a $0$-dimensional toric sheaf which can be described by three coloured (2-dimensional) partitions as follows. In each chart $\U_i$, the stacky $S$-family of $\cQ$ has only one non-zero component $\hat{Q}_i$ in its box decomposition. The set
\[
\{(\l_1,\l_2) \in \Z^2 \ | \ Q_i(\l_1,\l_2) \neq 0\}
\]
defines a partition. Suppose we are on $\U_1$ and write uniquely $$u_1 = i + Ib, \ u_2 = j + Jc, \ u_3 = k + Ka$$ as in equations \eqref{ui}. The partition associated to $\cQ|_{\U_1}$ is nonempty if and only if $Q_1(I,J) \neq 0$ in which case $Q_1(I,J)$ has $\bmu_a$-weight 
\[
u_1+u_2+u_3 \ \mathrm{mod} \ a.
\]
Since any non-zero element of $Q_1(I,J)$ generates the whole module $\cQ|_{\U_1}$, we know the $\bmu_a$-representation of any non-zero $Q_1(\l_1,\l_2)$, i.e.~
\[
u_1+u_2+u_3+ \l_1 b + \l_2 c \ \mathrm{mod} \ a.
\] 
This describes a $\Z_a$-coloured partition $\lambda_1(\F)$. Similarly, coloured partitions $\lambda_2(\F)$, $\lambda_3(\F)$ arise from the charts $\U_2$, $\U_3$. For $i=1,2,3$, we denote by $$\Pi_{(u_1,u_2,u_3)}(i)$$ the collection of all partitions with $\Z_{\i}$-colouring arising from some $\lambda_i(\F)$, where $\F$ is a rank 1 toric torsion free sheaf on $\PP$ with reflexive hull $L_{(u_1,u_2,u_3)}$. Depending on the choice of $a,b,c$, not all $\i$ colours might appear. \\

\noindent \textbf{Example 1}. Suppose $(a,b,c) = (1,1,2)$ and let $u_1,u_2,u_3 \in \Z$. The charts $\U_1$ and $\U_2$ give rise to the collections $\Pi_{(u_1,u_2,u_3)}(1)$, $\Pi_{(u_1,u_2,u_3)}(2)$ of (uncoloured) partitions.  Moreover, chart $\U_3$ gives $\Pi_{(u_1,u_2,u_3)}(3)$, i.e.~the collection of $\Z_2$-coloured partitions where the first block is $0$ if $u_1+u_2+u_3$ is even and $1$ if $u_1+u_2+u_3$ is odd:
\begin{displaymath}
\xy
(-20,25)*{(a,b,c)=(1,1,2)} ; (-20,15)*{u_1+u_2+u_3 \ \mathrm{even}} ; (50,25)*{(a,b,c)=(1,1,2)} ; (50,15)*{u_1+u_2+u_3 \ \mathrm{odd}} ; (0,0)*{} ; (30,0)*{} **\dir{} ; (0,5)*{} ; (30,5)*{} **\dir{.} ; (0,10)*{} ; (30,10)*{} **\dir{.} ; (0,15)*{} ; (30,15)*{} **\dir{.} ; (0,20)*{} ; (30,20)*{} **\dir{.} ; (0,25)*{} ; (30,25)*{} **\dir{.} ; (0,30)*{} ; (30,30)*{} **\dir{} ;                   (0,0)*{} ; (0,30)*{} **\dir{} ; (5,0)*{} ; (5,30)*{} **\dir{.} ; (10,0)*{} ; (10,30)*{} **\dir{.} ; (15,0)*{} ; (15,30)*{} **\dir{.} ; (20,0)*{} ; (20,30)*{} **\dir{.} ; (25,0)*{} ; (25,30)*{} **\dir{.} ; (30,0)*{} ; (30,30)*{} **\dir{} ; (5,5)*{} ; (25,5)*{} **\dir{-} ;  (25,5)*{} ; (25,10)*{} **\dir{-} ; (25,10)*{} ; (20,10)*{} **\dir{-} ; (20,10)*{} ; (20,15)*{} **\dir{-} ; (20,15)*{} ; (10,15)*{} **\dir{-} ; (10,15)*{} ; (10,25)*{} **\dir{-} ; (10,25)*{} ; (5,25)*{} **\dir{-} ; (5,25)*{} ; (5,5)*{} **\dir{-} ;                                    (5,5)*!<-7pt,-7pt>{0} ; (10,5)*!<-7pt,-7pt>{1} ; (15,5)*!<-7pt,-7pt>{0} ; (20,5)*!<-7pt,-7pt>{1} ; (5,10)*!<-7pt,-7pt>{1} ; (10,10)*!<-7pt,-7pt>{0} ; (15,10)*!<-7pt,-7pt>{1} ; (5,15)*!<-7pt,-7pt>{0} ; (5,20)*!<-7pt,-7pt>{1} ;                                                                      (70,0)*{} ; (100,0)*{} **\dir{} ; (70,5)*{} ; (100,5)*{} **\dir{.} ; (70,10)*{} ; (100,10)*{} **\dir{.} ; (70,15)*{} ; (100,15)*{} **\dir{.} ; (70,20)*{} ; (100,20)*{} **\dir{.} ; (70,25)*{} ; (100,25)*{} **\dir{.} ; (70,30)*{} ; (100,30)*{} **\dir{} ;                   (70,0)*{} ; (70,30)*{} **\dir{} ; (75,0)*{} ; (75,30)*{} **\dir{.} ; (80,0)*{} ; (80,30)*{} **\dir{.} ; (85,0)*{} ; (85,30)*{} **\dir{.} ; (90,0)*{} ; (90,30)*{} **\dir{.} ; (95,0)*{} ; (95,30)*{} **\dir{.} ; (100,0)*{} ; (100,30)*{} **\dir{} ; (75,5)*{} ; (95,5)*{} **\dir{-} ;  (95,5)*{} ; (95,10)*{} **\dir{-} ; (95,10)*{} ; (90,10)*{} **\dir{-} ; (90,10)*{} ; (90,15)*{} **\dir{-} ; (90,15)*{} ; (80,15)*{} **\dir{-} ; (80,15)*{} ; (80,25)*{} **\dir{-} ; (80,25)*{} ; (75,25)*{} **\dir{-} ; (75,25)*{} ; (75,5)*{} **\dir{-} ;                                    (75,5)*!<-7pt,-7pt>{1} ; (80,5)*!<-7pt,-7pt>{0} ; (85,5)*!<-7pt,-7pt>{1} ; (90,5)*!<-7pt,-7pt>{0} ; (75,10)*!<-7pt,-7pt>{0} ; (80,10)*!<-7pt,-7pt>{1} ; (85,10)*!<-7pt,-7pt>{0} ; (75,15)*!<-7pt,-7pt>{1} ; (75,20)*!<-7pt,-7pt>{0} ;
\endxy 
\end{displaymath}

\noindent \textbf{Example 2}. Suppose $(a,b,c)=(2,2,2)$ and let $u_1,u_2,u_3 \in \Z$. Then $$\Pi_{(u_1,u_2,u_3)}(1), \ \Pi_{(u_1,u_2,u_3)}(2), \ \Pi_{(u_1,u_2,u_3)}(3)$$ are collections of $\Z_2$-coloured partitions, but only one colour actually appears. This colour is 0 when $u_1+u_2+u_3$ is even and 1 when $u_1+u_2+u_3$ is odd
\begin{displaymath}
\xy
(-20,25)*{(a,b,c)=(2,2,2)} ; (-20,15)*{u_1+u_2+u_3 \ \mathrm{even}} ; (50,25)*{(a,b,c)=(2,2,2)} ; (50,15)*{u_1+u_2+u_3 \ \mathrm{odd}} ; (0,0)*{} ; (30,0)*{} **\dir{} ; (0,5)*{} ; (30,5)*{} **\dir{.} ; (0,10)*{} ; (30,10)*{} **\dir{.} ; (0,15)*{} ; (30,15)*{} **\dir{.} ; (0,20)*{} ; (30,20)*{} **\dir{.} ; (0,25)*{} ; (30,25)*{} **\dir{.} ; (0,30)*{} ; (30,30)*{} **\dir{} ;                   (0,0)*{} ; (0,30)*{} **\dir{} ; (5,0)*{} ; (5,30)*{} **\dir{.} ; (10,0)*{} ; (10,30)*{} **\dir{.} ; (15,0)*{} ; (15,30)*{} **\dir{.} ; (20,0)*{} ; (20,30)*{} **\dir{.} ; (25,0)*{} ; (25,30)*{} **\dir{.} ; (30,0)*{} ; (30,30)*{} **\dir{} ; (5,5)*{} ; (25,5)*{} **\dir{-} ;  (25,5)*{} ; (25,10)*{} **\dir{-} ; (25,10)*{} ; (20,10)*{} **\dir{-} ; (20,10)*{} ; (20,15)*{} **\dir{-} ; (20,15)*{} ; (10,15)*{} **\dir{-} ; (10,15)*{} ; (10,25)*{} **\dir{-} ; (10,25)*{} ; (5,25)*{} **\dir{-} ; (5,25)*{} ; (5,5)*{} **\dir{-} ;                                    (5,5)*!<-7pt,-7pt>{0} ; (10,5)*!<-7pt,-7pt>{0} ; (15,5)*!<-7pt,-7pt>{0} ; (20,5)*!<-7pt,-7pt>{0} ; (5,10)*!<-7pt,-7pt>{0} ; (10,10)*!<-7pt,-7pt>{0} ; (15,10)*!<-7pt,-7pt>{0} ; (5,15)*!<-7pt,-7pt>{0} ; (5,20)*!<-7pt,-7pt>{0} ;                                                                      (70,0)*{} ; (100,0)*{} **\dir{} ; (70,5)*{} ; (100,5)*{} **\dir{.} ; (70,10)*{} ; (100,10)*{} **\dir{.} ; (70,15)*{} ; (100,15)*{} **\dir{.} ; (70,20)*{} ; (100,20)*{} **\dir{.} ; (70,25)*{} ; (100,25)*{} **\dir{.} ; (70,30)*{} ; (100,30)*{} **\dir{} ;                   (70,0)*{} ; (70,30)*{} **\dir{} ; (75,0)*{} ; (75,30)*{} **\dir{.} ; (80,0)*{} ; (80,30)*{} **\dir{.} ; (85,0)*{} ; (85,30)*{} **\dir{.} ; (90,0)*{} ; (90,30)*{} **\dir{.} ; (95,0)*{} ; (95,30)*{} **\dir{.} ; (100,0)*{} ; (100,30)*{} **\dir{} ; (75,5)*{} ; (95,5)*{} **\dir{-} ;  (95,5)*{} ; (95,10)*{} **\dir{-} ; (95,10)*{} ; (90,10)*{} **\dir{-} ; (90,10)*{} ; (90,15)*{} **\dir{-} ; (90,15)*{} ; (80,15)*{} **\dir{-} ; (80,15)*{} ; (80,25)*{} **\dir{-} ; (80,25)*{} ; (75,25)*{} **\dir{-} ; (75,25)*{} ; (75,5)*{} **\dir{-} ;                                    (75,5)*!<-7pt,-7pt>{1} ; (80,5)*!<-7pt,-7pt>{1} ; (85,5)*!<-7pt,-7pt>{1} ; (90,5)*!<-7pt,-7pt>{1} ; (75,10)*!<-7pt,-7pt>{1} ; (80,10)*!<-7pt,-7pt>{1} ; (85,10)*!<-7pt,-7pt>{1} ; (75,15)*!<-7pt,-7pt>{1} ; (75,20)*!<-7pt,-7pt>{1} ;
\endxy 
\end{displaymath}

Back to the general situation. For $i=1,2,3$, the partitions of $\Pi_{(u_1,u_2,u_3)}(i)$ are coloured by $\Z_{\i}$, though not all colours actually have to appear (e.g.~Example 2). For any $l \in \Z_{\i}$ and $\lambda \in \Pi_{(A,B,C)}(i)$, we denote by $\#_l \lambda$ the number of boxes of $\lambda$ with colour $l$. The previous discussion together with Propositions \ref{eqpic}, \ref{point} gives:
\begin{proposition} \label{ideal}
Let $\F$ be a rank 1 toric torsion free sheaf on $\PP$. Let $L_{(u_1,u_2,u_3)}$ be the reflexive hull of $\F$ and let $\cQ$ be the cokernel of the inclusion $\F \subset L_{(u_1,u_2,u_3)}$. Then $\cQ$ is described by three coloured partitions $\lambda_i(\F) \in \Pi_{(u_1,u_2,u_3)}(i)$, $i=1,2,3$ and the $K$-group class of $\F$ is 
\begin{align*}
\Bigg(1 - \sum_{i=1}^3\sum_{l=0}^{\i-1} \#_l \lambda_i(\F) \cdot g^l(1-g^a)(1-g^b)(1-g^c)/(1-g^{\i})\Bigg)g^{u_1+u_2+u_3},
\end{align*}
where $g:=[\O_{\PP}(-1)]$.
\end{proposition}

Suppose $\F$ is a rank 1 toric torsion free sheaf on $\PP$ with reflexive hull $L_{(u_1,u_2,u_3)}$. When counting $\T$-fixed points of moduli spaces $M(\rc)$ below, we are only interested in the coherent sheaf $\F$ and not its equivariant structure. A toric line bundle $L_{(u_{1}',u_{2}',u_{3}')}$ has trivial underlying line bundle if and only if $$u_{1}'+u_{2}'+u_{3}' = 0.$$ By tensoring $\F$ with any $L_{(-u_1,-u_2,u_1+u_2)}$, we change its equivariant structure, but not the underlying coherent sheaf. Therefore, it suffices to only consider rank 1 toric torsion free sheaves $\F$ on $\PP$ with reflexive hulls of the form $L_{(0,0,u)}$. In this case $c_1(\F) = -u$. For any $\beta \in A^1(\PP) \cong \Z$ and $i=1,2,3$, we define $$\Pi_\beta(i) := \Pi_{(0,0,-\beta)}(i).$$ 
\begin{theorem} \label{thm:generating function}
For any $\beta \in A^1(\PP)$, we have
\[
\rG_\beta = \left( \sum_{\lambda \in \Pi_\beta(1)} \prod_{l=0}^{a-1} p_{l}^{\#_l \lambda} \right) \left( \sum_{\lambda \in \Pi_\beta(2)} \prod_{l=0}^{b-1} q_{l}^{\#_l \lambda} \right) \left( \sum_{\lambda \in \Pi_\beta(3)} \prod_{l=0}^{c-1} r_{l}^{\#_l \lambda} \right),
\]
where the variable $p_l, q_l, r_l$ satisfy relations (\ref{equ:varrel}).
\end{theorem}
\begin{proof}
Fix a class $\rc \in K_0(\PP)_\Q$ with $c_1(\rc) = \beta$. The moduli scheme $M(\rc)$ has a torus action, which at the level of closed points is given by\footnote{To construct the action of $\T$ on $M(\rc)$ as a morphism, one extends the argument of \cite[Prop.~4.2]{Koo1} to Nironi's setting \cite{Nir}.}
\[
t\cdot [\F] := [t^* \F], \ t \in \T, \ \F \in M(\rc).
\]
By localization $e(M(\rc)) = e(M(\rc)^\T)$. The closed points of $M(\rc)^\T$ are exactly the isomorphism classes of rank 1 torsion free sheaves with class $\rc$ satisfying $t^* \F \cong \F$ for all $t \in \T$. Since any such $\F$ is simple, it admits a $\T$-equivariant structure \cite[Prop.~4.4]{Koo1} and this $\T$-equivariant structure is unique up to tensoring by a character of $\T$ \cite[Prop.~4.5]{Koo1}. Consequently, $M(\rc)^\T$ is equal to the set\footnote{Note that this only describes a bijection of finite sets. We are actually dealing with an isomorphism between two finite collections of reduced points by \cite[Thm.~4.9]{Koo1} (extended to this setting).} of \emph{equivariant} isomorphism classes of rank 1 toric torsion free sheaves with class $\rc$, where two such classes $[\F], [\F']$ are identified whenever $\F$ and $\F'$ differ by a character, i.e.~whenever $\F \cong \F' \otimes L_{(u_1,u_2,u_3)}$ for some $u_1+u_2+u_3 = 0$.

Hence any element of $M(\rc)^\T$ is \emph{uniquely} represented by stacky $S$-families $\{\hat{F}_i\}_{i=1,2,3}$ of a rank 1 toric torsion free sheaf with reflexive hull $L_{(0,0,-\beta)}$. By Proposition \ref{ideal}, the elements of $M(\rc)^\T$ are in bijective correspondence with triples of partitions $(\lambda_1,\lambda_2,\lambda_3) \in \Pi_\beta(1) \times \Pi_\beta(2) \times \Pi_\beta(3)$ satisfying
\begin{equation*}
\Bigg(1 - \sum_{i=1}^3\sum_{l=0}^{\i-1} \#_l \lambda_i \cdot g^l(1-g^a)(1-g^b)(1-g^c)/(1-g^{\i})\Bigg)g^{-\beta} =\rc. \qedhere
\end{equation*}
\end{proof}

By Theorem \ref{thm:generating function}, the problem of computing Euler characteristics of moduli spaces of the rank 1 torsion free sheaves on $\PP$ is reduced to counting coloured partitions. Counting coloured partitions is a delicate topic. To the authors knowledge, no closed formula exists for the general case. When the action of the cyclic group $\bmu_k$ on $\C^2$ is \emph{balanced}, i.e.~is of the form $\mu \cdot (x,y)=(\mu x, \mu^{k-1}y)$, there is an elegant formula appearing in the physics literature \cite{DS}. The formula in this case is\footnote{Assuming all coloured partitions have colour 0 at the origin.}
\begin{align} \label{balanced refined formula}
&\sum_{\text{coloured partitions } \lambda}q_0^{\#_0 \lambda}\cdots q_{k-1}^{\#_{k-1} \lambda}=\\ \nonumber &
\frac{1}{\prod_{j > 0}(1-(q_0\cdots q_{k-1})^j)^k}\sum_{n_1,\dots, n_{k-1}\in \mathbb{Z}}(q_0\cdots q_{k-1})^{\sum_{i=0}^{k-2} n_i^2-n_in_{i+1}}\prod_{r=1}^{k-1}q_{k-r}^{r^2/2+n_1r-r/2}.
\end{align}

One can count coloured partitions keeping track of the number of boxes with colour 0 only by setting $q_0=q$ and $q_1=\cdots =q_{k-1}=1$. Then formula (\ref{balanced refined formula}) is related to the character formula of the affine Lie algebra $\widehat{su}(k)$ (Equation (4.5) in [DS]) 
\begin{equation} \label{balanced formula}
\sum_{\text{coloured partitions } \lambda}q^{\#_0 \lambda}=\frac{q^{k/24}}{\eta(q)}\chi^{\widehat{su}(k)}(0),
\end{equation}
where $\eta(q)$ is the Dedekind eta function. \\

\noindent \textbf{Example 1}. 
When $(a,b,c)=(1,1,1),$ $(1,1,2),$ and $(1,2,3)$ (or permutations thereof), we only have balanced actions on all three charts of $\PP$. From Theorem \ref{thm:generating function} and formulae (\ref{balanced refined formula}), (\ref{balanced formula}), we get closed expressions for the generating functions.

\begin{enumerate} [i)] 
\item Let $(a,b,c)=(1,1,1)$. In this case $p_0=q_0=r_0$ by (\ref{equ:varrel}) and
\[
\rG_0 = \frac{1}{\prod_{k>0} (1-r_{0}^{k})^3}.
\]

\item Let $(a,b,c)=(1,1,2)$. In this case we have $p_0=q_0=r_0r_1$ by (\ref{equ:varrel}) and  
\[
\rG_0 = \frac{1}{\prod_{k>0} (1-(r_{0}r_{1})^{k})^4} \sum_{k \in \Z} r_{0}^{k^2} r_{1}^{k^2+k}.
\]
Putting $r_1=1$ and $r_0=q$ in $\rG_0$, we get  
$$\rG_0(q)=\frac{q^{1/6}}{\eta(q)^4}\chi^{\widehat{su}(2)}(0) = \frac{q^{1/6}}{\eta(q)^4} \theta_3(q),$$
where $\theta_3(q)$ is a Jacobi theta function. Geometrically, the power of $q$ keeps track of the Euler characteristic $\chi(\F^{**}/\F)$ of our rank 1 torsion free sheaves $\F$, where $\F^{**}$ is the reflexive hull of $\F$. This can be seen by noting that $$\chi(\O_{P_1} \otimes \mu^i) = \delta_{0i}, \ \chi(\O_{P_2} \otimes \nu^j) = \delta_{0j}, \ \chi(\O_{P_3} \otimes \xi^k) = \delta_{0k},$$ where $\delta_{ij}$ is the Kronecker delta.

\item  Let $(a,b,c)=(1,2,3)$. In this case $p_0=q_0q_1=r_0r_1r_2$ by  (\ref{equ:varrel}) and 
\[
\rG_0 = \frac{1}{\prod_{k>0} (1-(r_{0}r_{1}r_{2})^{k})^6} \sum_{k \in \Z} (r_0r_1r_2)^{k^2} q_{1}^{k} \sum_{k,l \in \Z} r_{0}^{k^2-kl+l^2} r_{1}^{k^2+2k+1-kl+l^2} r_{2}^{k^2+k-kl+l^2}.
\]
Setting $q_1=r_1=r_2=1$ and $p_0=q_0=r_0=q$ in $\rG_0$ as in the previous case, one gets  
$$\rG_0(q)=\frac{q^{1/4}}{\eta(q)^6}\chi^{\widehat{su}(2)}(0)\chi^{\widehat{su}(3)}(0) = \frac{q^{1/4}}{\eta(q)^6} \theta_3(q)(\theta_3(q) \theta_3(q^3) + \theta_2(q) \theta_2(q^3)),$$
where $\theta_2(q)$, $\theta_3(q)$ are Jacobi theta functions. 
\end{enumerate}

\noindent \textbf{Example 2}. Let $(a,b,c)=(1,c,c)$ with $c \geq 2$. In this case there is also a nice formula. Relations (\ref{equ:varrel}) give $p_0=q_0\cdots q_{c-1}$ and $q_i=r_i$. The following formula can be proved easily 
\[
\rG_0 = \frac{1}{\prod_{k>0} (1-(r_0 \cdots r_{c-1})^{k})} \frac{1}{(\prod_{k>0}\prod_{i=0}^{c-2} (1-r_{0} \cdots r_i (r_{0}\cdots r_{c-1})^{k-1}))^2}.
\]

\noindent \textbf{Example 3}. For $\PP(1,1,3)$, we encounter a non-balanced $\bmu_3$-action in the chart $\U_3$. We do not know a closed formula for $\rG_0$ in this example. Relations (\ref{equ:varrel}) give $p_0=q_0=r_0r_1r_2$. The first few terms of the expansion are 
$$\rG_0=\frac{(1+r_0+2r_0r_1+2r_0r_1r_2+r_0r_1^2+2r_0^2r_1r_2+3r_0r_1^2r_3+O(5))}{\prod_{k>0} (1-(r_{0}r_{1}r_2)^{k})^2}.$$

\subsection{Rank 2}

In this section, $\rc \in K_0(\PP)_\Q$ denotes the class of a rank 2 torsion free sheaf on $\PP:=\PP(a,b,c)$. We fix $\O_\mathbf{P}(1)$ and $\cE := \bigoplus_{n=0}^{E-1} \O_\PP(-n)$ as in the beginning of this section. The moduli scheme of $\mu$-stable sheaves on $\PP$ with class $\rc$ is denoted by  $M_\cE(\rc)$ as before. Unlike the rank 1 case, not all rank 2 torsion free sheaves on $\PP$ are $\mu$-stable. We denote by $$N_\cE(\rc) \subset M_\cE(\rc)$$ the open subset of the $\mu$-stable locally free sheaves.

Let $I\PP$ be the inertia stack. Using our explicit calculations of Section \ref{sec:K} together with the results of \cite{BCS, BH}, one can see that the Chern character map 
\[
\tch : K_0(\PP)_\Q \longrightarrow A^*(I\PP)_{\bmu_\infty}
\]
is a ring isomorphism. Therefore, fixing a class in $\rc \in K_0(\PP)_\Q$ is equivalent to fixing an element $\tch \in A^*(I\PP)_{\bmu_\infty}$. Recall that the set $$F = D \sqcup \coprod_{i,j} D_{ij} \sqcup \coprod_{i} D_i$$ indexes the components of the inertia stack $I\PP$ (Section 4.2). For a fixed $f \in F$ corresponding to component $Z$, let $\tch_f$ denote the part of $\tch$ taking values in $A^*(Z)_{\bmu_\infty}$. For notational convenience, we define $$(\tch_f)^k := (\tch_f)_{\dim Z-k} \in A^{\dim Z - k}(Z)_{\bmu_\infty}$$ to be its dimension $k$ part. In this context, we refer to $k$ as the codegree. The reason for this strange notational convention is that we are dealing with Chern characters of sheaves on components of \emph{different dimension} of the inertia stack $I\PP$ so it is more natural to keep track of dimension than codimension. Fix  $\alpha \in A^0(I\PP)_{\bmu_{\infty}}$ and $\beta \in A^1(I\PP)_{\bmu_{\infty}}$ to be (parts of) the Chern character of a rank 2 torsion free sheaf on $\PP$. Define the generating functions 
$$\rH_{\alpha,\beta}(q):=\sum_{\tiny \begin{array}{c} \tch^2(\rc)=\alpha \\ \tch^1(\rc)=\beta \end{array}}e(M_\cE(\rc))q^\rc, \quad \quad \rH^{\vb}_{\alpha,\beta}(q):=\sum_{\tiny \begin{array}{c} \tch^2(\rc)=\alpha \\ \tch^1(\rc)=\beta\end{array}}e(N_\cE(\rc))q^\rc.$$ 
So in terms of $\tch$, these generating functions sum over all 0-dimensional (i.e.~codegree 0) parts $(\tch_f)^0$. Here $\vb$ stands for vector bundle.

Our goal is to determine the generating function for the Euler characteristics $M_\cE(\rc)$. By the following lemma, this can be achieved by determining the generating functions for the Euler characteristics of $N_\cE(\rc)$ and of the moduli spaces of rank 1 torsion free sheaves on $\U_1, \U_2, \U_3$. The latter generating functions, which we denote by $\rG_{\U_i}(q)$, are known by the result of Section 7.1. The lemma is a small variation on a result by L.~G\"ottsche and K.~Yoshioka \cite[Prop.~3.1]{Got}.

\begin{lemma} The following holds: $\rH_{\alpha,\beta}(q)=\rH^{\vb}_{\alpha,\beta}(q)\prod_{i=1}^3\rG_{\U_i}(q)^2$. 
\end{lemma}
\begin{proof}\newcommand{\V}{\mathcal{V}}
\newcommand{\mQ}{\mathcal{Q}}
Embedding into the reflexive hull of a rank 2 torsion free sheaf $\F$ on $\PP$ gives the short exact sequence 
\begin{equation} \label{ses}
0\longrightarrow \F \longrightarrow \F^{**} \longrightarrow \mQ\to 0, 
\end{equation}
where $\mQ$ is a 0-dimensional sheaf on $\PP$. Since $\PP$ is a smooth DM stack of dimension two $\F^{**}$ is a rank 2 locally free sheaf. 

We denote by $\operatorname{Quot}_{\PP}(\V,\rc)$ the projective scheme of quotients of a locally free sheaf $\V$ on $\PP$ with class $\rc\in K_0(\PP)_\Q$. See \cite{OS} for the construction of these Quot schemes. 

Let $\rc \in K_0(\PP)_\Q$ be the class of a rank 2 torsion free sheaf on $\PP$. By short exact sequence \eqref{ses} $M_\cE(\rc)$ can be stratified (at the level of closed points) by $$\operatorname{Quot}_{\PP}(\V,\rc'), \ [\V] \in N_\cE(\rc'') \ \mathrm{with} \ \rc'' - \rc' = \rc.$$  Hence at the level of closed points, $M_\cE(\rc)$ is in bijective correspondence with the disjoint union
\[
\coprod_{\rc'' - \rc'  = \rc} \coprod_{[\V] \in N_\cE(\rc'')} \operatorname{Quot}_{\PP}(\V,\rc'). 
\]

Let $\V$ be a rank 2 $\T$-equivariant locally free sheaf on $\PP$. In order to determine $e(\operatorname{Quot}_{\PP}(\V,\rc'))$ for any 0-dimensional $\rc'$, we count $\T$-fixed points using the $\T$-action on $\PP$. This $\T$-action lifts to $\operatorname{Quot}_{\PP}(\V,\rc')$ (see proof of \cite[Prop.~4.1]{Koo1}). Since $\mQ$ is 0-dimensional with support at $P_1,P_2,P_3 \in \PP$, giving a $\T$-fixed quotient $$\V \longrightarrow \mQ \longrightarrow 0$$ is equivalent to giving quotients $$\O_{\U_i}^{\oplus 2} \longrightarrow \mQ |_{\U_i} \longrightarrow 0,$$ for all $i=1,2,3$. Therefore at the level of closed points, $\operatorname{Quot}_{\PP}^\T(\V,\rc')$ is in bijective correspondence with the disjoint union $$\coprod_{\rc_1+\rc_2+\rc_3=\rc'} \prod_{i=1}^3\operatorname{Quot}_{\U_i}^\T(\O_{\U_i}^{\oplus 2},\rc_i).$$

We are reduced to determine $e(\operatorname{Quot}_{\U_i}(\O_{\U_i}^{\oplus 2},\rc_i))$. 
For this we use the $\C^{*2}$-action on $\operatorname{Quot}_{\U_i}(\O_{\U_i}^{\oplus 2},\rc_i)$, which scales the factors of $\O_{\U_i}^{\oplus 2}$. At the level of closed points, the $\C^{*2}$-fixed locus of $\operatorname{Quot}_{\U_i}(\O_{\U_i}^{\oplus 2},\rc_i)$ is in bijective correspondence with the disjoint union $$\coprod_{\rc'_i+\rc''_i=\rc_i}\operatorname{Quot}_{\U_i}(\O_{\U_i},\rc'_i)\times \operatorname{Quot}_{\U_i}(\O_{\U_i},\rc''_i).$$ 
But  $\operatorname{Quot}_{\U_i}(\O_{\U_i},\rc''_i)$ is isomorphic to the moduli space of rank 1 torsion free sheaves on $\U_i$ with class $-\rc''_i$. This proves $$\sum_{\rc_i}e(\operatorname{Quot}_{\U_i}(\O_{\U_i}^{\oplus 2},\rc_i))q^{\rc_i}=\rG_{\U_i}(q)^2,$$ which finishes the proof of the lemma.
\end{proof}

The rest of the paper is devoted to the study of rank 2 toric locally free sheaves on $\PP$ and the generating function $\rH^{\vb}_{\alpha, \beta}(q)$. In Section 6, we classified three types of toric locally free sheaves on $\PP$: type I, II and III. Since types II and III are always decomposable, they are never $\mu$-stable. Hence we only need to consider type I rank 2 toric locally free sheaves $\F$ on $\PP$. 

\begin{definition}
The stacky $S$-families of a type I locally free sheaf $\F$ on $\PP$ are entirely determined by integers $u_1,u_2,u_3$ and $v_1,v_2,v_3 \geq 0$, 
and an element $(p_1,p_2,p_3) \in (\PP^1)^3$ (see Section 6). Recall that if $v_i=0$ then the corresponding $p_i$ does not occur. Our formulae can be more succinctly expressed if we use 
\begin{flalign*}
\qquad\qquad\qquad\qquad\qquad\qquad w_1 := bv_1, \ w_2:=cv_2, \ w_3:=av_3.  && \oslash 
\end{flalign*}  
\end{definition}

We want to express the $K$-group class of $[\F]$ in terms of $(u_1,u_2,u_3)$, $(w_1,w_2,w_3)$, $(p_1,p_2,p_3)$ only. 

\begin{proposition} \label{rank2}
Let $\F$ be a type I rank 2 toric vector locally free sheaf on $\PP$. Let its stacky $S$-families be described by $u_1,u_2,u_3 \in \Z$, $w_1,w_2,w_3 \in \Z_{\geq 0}$ satisfying $b \mid w_1$, $c \mid w_2$, $a \mid w_3$ and $(p_1,p_2,p_3) \in (\PP^1)^3$. Then the $K$-group class of $\F$ is
\begin{align*}
\Big(1+g^{w_1+w_2+w_3}&-(1-g^{w_1})(1-g^{w_2})(1-\delta_{p_1 p_2})-(1-g^{w_2})(1-g^{w_3})(1-\delta_{p_2 p_3}) \\
&-(1-g^{w_3})(1-g^{w_1})(1-\delta_{p_3 p_1})\Big) g^{u_1+u_2+u_3},
\end{align*}
where $g:=[\O_{\PP}(-1)]$ and 
the Kronecker delta $\delta_{pq}$ is 1 if $p=q$ and 0 if $p \neq q$.
\end{proposition}
\begin{proof}
We make use of the following observation. Let $\G$ be another toric torsion free sheaf with stacky $S$-families $\hat{G}_1$, $\hat{G}_2$, $\hat{G}_3$. Assume 
\begin{equation} \label{G}
\mathrm{dim}({}_{b} G_i(\l_1,\l_2)_l) = \mathrm{dim}({}_{b} F_i(\l_1,\l_2)_l),
\end{equation}
 for all $i=1,2,3$, $b \in \cB_\T$, $\l_1,\l_2 \in \Z$ and $l \in \Z_{\i}$. Then Lemma \ref{devissage} below shows $$[\F] = [\G] \in K_0(\PP)_\Q.$$ Assume without lost of generality $u_1=u_2=u_3=0$. The general case follows from tensoring by the line bundle $L_{(u_1,u_2,u_3)}$ introduced in Section 6. We construct a \emph{simpler} toric sheaf $\G$ satisfying Equation (\ref{G}). We take $\G$ to be an equivariant subsheaf of $L_{(0,0,0)} \oplus L_{(w_1,w_2,w_3)}$. In each chart, such a sheaf only has a non-zero box summand over $(0,0)$. For each $i=1,2,3$, define\footnote{Here $v_4:=v_1$, $w_4:=w_1$ and $p_4:=p_1$.}
\[
{}_{(0,0)}G_i(\l_1,\l_2) := {}_{(0,0)}L_{(0,0,0),i}(\l_1,\l_2) \oplus {}_{(0,0)}L_{(w_1,w_2,w_3),i}(\l_1,\l_2), \\
\]
in the following regions
\begin{align*}
&\{(\l_1,\l_2) \ | \ \l_1 \geq v_i \ \mathrm{or} \ \l_2 \geq v_{i+1}\}, \\ 
&\{(\l_1,\l_2) \ | \ \l_1 < v_i \ \mathrm{and} \ \l_2 < v_{i+1} \}, \ \mathrm{if} \ p_i = p_{i+1}
\end{align*}
and  ${}_{(0,0)}G_i(\l_1,\l_2)  = 0$ for all remaining $(\l_1,\l_2)$. From this description one can show that Equation (\ref{G}) is satisfied and see that $[\G]$ is equal to
\begin{align*}
1+g^{w_1+w_2+w_3}&- (1- \delta_{p_1 p_2})\sum_{I=0}^{w_1/b - 1} \sum_{J=0}^{w_2/c-1} [\O_{P_1} \otimes \mu^{Ib+Jc}] \\
&- (1- \delta_{p_2 p_3})\sum_{J=0}^{w_2/c-1} \sum_{K=0}^{w_3/a-1} [\O_{P_2} \otimes \nu^{Jc+Ka}] \\
&- (1- \delta_{p_3 p_1}) \sum_{K=0}^{w_3/a-1} \sum_{I=0}^{w_1/b-1} [\O_{P_3} \otimes \xi^{Ka+Ib}].
\end{align*}
The sums can be computed using Proposition \ref{point}. 
\end{proof}

\begin{lemma} \label{devissage}
Let $\F$, $\G$ be toric torsion free sheaves on $\PP$ and assume their stacky $S$-families only possibly differ in the maps between their fine-graded weight spaces. Then $[\F] = [\G] \in K_0(\PP)_\Q$.
\end{lemma}
\begin{proof}
We give the argument for toric torsion free sheaves on $\C^2$ with standard $\T$-action. The reader can easily generalize the argument to $\PP$. Take $(1,0)$, $(0,1)$ as a basis for the character group $M = X(\T)$. Let $\F$, $\G$ be toric torsion free sheaves on $\C^2$ and denote their $S$-families by $\hat{F}$, $\hat{G}$. Assume $\F$, $\G$ only possibly differ in the maps between their weight spaces. We apply \emph{equivariant d\'evissage}. Since both sheaves are finitely generated, there are integers $A,B$ such that all homogeneous generators of $H^0(\F)$ and $H^0(\G)$ lie in $(-\infty,u_1) \times (-\infty,u_2)$. In particular, for all $\l_1 \geq u_1$ and $\l_2 \geq u_2$ we have $F_1(\l_1,\l_2) = G_1(\l_1,\l_2) = \C^2$ and the maps between the weight spaces are the identity. Define the $S$-family
\begin{align*}
&S_1(\l_1,\l_2) := F(\l_1,\l_2), \ \mathrm{for} \ \l_1 \geq u_1 \ \mathrm{and} \ \l_2 \geq u_2, \\
&S_1(\l_1,\l_2) := 0, \ \mathrm{otherwise}.
\end{align*}
At the level of $S$-families, $S_1 \subset F$ and we denote the quotient by $Q_1$. This corresponds to a toric subsheaf $\S_1 \subset \F$ with quotient sheaf $\cQ_1$ and we get $$[\F] = [\S_1]+[\cQ_1].$$ Repeating the same argument for $\G$, we obtain $$[\G] = [\tilde{\S}_1]+[\tilde{\cQ}_1] = [\S_1]+[\tilde{\cQ}_1],$$ since $\tilde{\S}_1 = \S_1$. (Aside: $\S_1 \cong \O_{\C^2}^{\oplus r}$, where $r$ is the rank.)

We shift focus to $\cQ_1$, $\tilde{\cQ}_1$. Define the $S$-family
\begin{align*}
&S_2(\l_1,\l_2) := F(\l_1,\l_2), \ \mathrm{for} \ \l_1= u_1-1 \ \mathrm{and} \ \l_2 \geq u_2, \\
&S_2(\l_1,\l_2) := 0, \ \mathrm{otherwise}.
\end{align*}
This gives a toric sheaf $\S_2 \subset \cQ_1$ with quotient $\cQ_2$ and $$[\cQ_1] = [\S_2]+[\cQ_2].$$ Repeating the same argument for $\tilde{\cQ}_1$, we obtain $$[\tilde{\cQ}_1] = [\tilde{\S}_2]+[\tilde{\cQ}_2] = [\S_2]+[\tilde{\cQ}_2],$$ since $\tilde{\S}_2 = \S_2$. (Aside: $\S_2 \cong \O_{\C}^{\oplus s}$ for some $s$.) Repeating the argument a finite number of times, we ``peel off'' the region where $\l_1 \geq u_1$ or $\l_2 \geq u_2$. What is left are toric sheaves $\cQ_n$, $\tilde{\cQ}_n$ defined by the $S$-families
\begin{align*}
&Q_n(\l_1,\l_2) := F(\l_1, \l_2), \ \tilde{Q}_n(\l_1,\l_2) := G(\l_1, \l_2),  \ \mathrm{for} \ \l_1 \leq u_1-1 \ \mathrm{and} \ \l_2 \leq u_2-1, \\
&Q_n(\l_1,\l_2) := 0, \ \tilde{Q}_n(\l_1,\l_2) := 0, \ \mathrm{otherwise}.
\end{align*}
It suffices to show $[\cQ_n] = [\tilde{\cQ}_n]$.

For this, define the $S$-family
\begin{align*}
&S_{n+1}(\l_1,\l_2) := F(\l_1,\l_2), \ \mathrm{for} \ (\l_1,\l_2) = (u_1-1,u_2-1), \\
&S_{n+1}(\l_1,\l_2) := 0, \ \mathrm{otherwise}.
\end{align*}
This gives a toric sheaf $\S_{n+1} \subset \cQ_n$ with quotient $\cQ_{n+1}$ and $$[\cQ_n] = [\S_{n+1}]+[\cQ_{n+1}].$$ Repeating the same argument for $\tilde{\cQ}_{n+1}$, we obtain $$[\tilde{\cQ}_n] = [\tilde{\S}_{n+1}]+[\tilde{\cQ}_{n+1}] = [\S_{n+1}]+[\tilde{\cQ}_{n+1}],$$ since $\tilde{\S}_{n+1} = \S_{n+1}$. (Aside: $\S_{n+1} \cong \O_0^{\oplus t}$ for some $t$.) Repeating the final part a finite number of times, the lemma is proved. 
\end{proof}

We can now classify all rank 2 $\mu$-stable toric locally free sheaves on $\PP$. Let $P_\cE(\cdot,t)$ be the modified Hilbert polynomial with respect to $\O_\mathbf{P}(1)$ and $\cE$ as introduced at the beginning of this section. Note that $\O_\mathbf{P}(1)$ pulls back to $\O_\PP(m)$ on $\PP$. (Recall the definition of the common divisors $d$, $d_{ij}$ and the common multiple $m$ from Section 4.2.) The modified slope $\mu$ of any torsion free sheaf on $\PP$ is defined to be the ratio of the linear and quadratic terms of its modified Hilbert polynomial. 

\begin{proposition} \label{rank2classify} 
Let $\F$ be a type I rank 2 toric locally free sheaf on $\PP$ and let its stacky $S$-families be described by $u_1,u_2,u_3 \in \Z$, $w_1,w_2,w_3 \in \Z_{\geq 0}$ satisfying $b \mid w_1$, $c \mid w_2$, $a \mid w_3$ and $(p_1,p_2,p_3) \in (\PP^1)^3$. Then $\F$ is $\mu$-stable if and only if $w_1, w_2, w_3 > 0$, $p_1, p_2, p_3$ are mutually distinct and
$$w_1 < w_2+w_3, \ w_2 < w_1+w_3, \ w_3 < w_1+w_2.$$
\end{proposition}
\begin{proof}
We start by observing that $\F$ is $\mu$-semistable if and only if $\mu(\G) \leq \mu(\F)$ for any \emph{toric} subsheaf $\G$ with $0< \mathrm{rk}(\G) < \mathrm{rk}(\F)$. This is immediate by noting that the Harder-Narasimhan filtration of $\F$ is $\T$-equivariant. It is also true that $\F$ is $\mu$-stable if and only if $\mu(\G) < \mu(\F)$ for any \emph{toric} subsheaf $\G$ with $0< \mathrm{rk}(\G) < \mathrm{rk}(\F)$. This statement is more subtle and is proved in the case of smooth toric varieties in \cite[Prop.~4.13]{Koo1}. The proof extends to our current setting.

Note that $\F$ can only be indecomposable if $w_1, w_2, w_3 >0$ and $p_1, p_2, p_3$ are mutually distinct so we assume this from now on. From the toric description of $\F$ in terms of stacky $S$-families, it is clear that $\F$ has three toric line bundles $\L_1$, $\L_2$, $\L_3$ generated by $p_1,p_2,p_3$ as a subsheaf. Moreover, any rank 1 toric subsheaf is contained in one of them. Their stacky $S$-families and $K$-group classes are easily computed
\[
[\L_1] = g^{w_2 + w_3}g^{u_1+u_2+u_3}, \ [\L_2] = g^{w_1 + w_3}g^{u_1+u_2+u_3}, \ [\L_3] = g^{w_1 + w_2}g^{u_1+u_2+u_3}.
\]
The $K$-group class of $\F$ was computed in Proposition \ref{rank2}. Since all these $K$-group classes are expressed as sums of elements $g^r$, we can use Corollary \ref{Plinebundle} to compute the top two coefficients of the modified Hilbert polynomials and hence the modified slopes
\begin{align*}
\mu_{\F} &= \frac{1}{m}(a+b+c-(w_1+w_2+w_3) - 2(u_1+u_2+u_3)) + \frac{2}{em} \sum_{i=1}^e n_i, \\ 
\mu_{\L_i} &= \frac{1}{m} (a+b+c - 2(w_j+w_k + u_1+u_2+u_3)) + \frac{2}{em} \sum_{i=1}^e n_i,
\end{align*}
for any $\{i,j,k\}=\{1,2,3\}$. Since any other rank 1 toric subsheaf is contained in some $\L_i$ and has modified slope $\leq \mu_{\L_i}$, it suffices to test $\mu$-stability for $\L_1, \L_2, \L_3$ and the proposition follows. 
\end{proof}

From Proposition \ref{rank2}, it is easy to compute the Chern character of any type I rank 2 toric locally free sheaf on $\PP$. Since we are only interested in indecomposable locally free sheaves, we assume all $w_i > 0$ and $p_i$ are mutually distinct. Moreover, since we are only interested in $\T$-equivariant structures up to tensoring by a character, we suppose $u_1=u_2=0$. See Section 4.2 for a discussion of the Chern character map $\tch : K_0(\PP)_\Q \rightarrow A^*(I\PP)_{\bmu_\infty}$ and the indexing sets $D, D_{i}, D_{ij}$ (Definition \ref{indexsets}).

\begin{lemma} \label{rank2tch}
Let $\F$ be a type I rank 2 toric locally free sheaf on $\PP$. Let its stacky $S$-families be described by $u_1,u_2,u_3 \in \Z$, $w_1,w_2,w_3 \in \Z_{\geq 0}$ satisfying $b \mid w_1$, $c \mid w_2$, $a \mid w_3$ and $(p_1,p_2,p_3) \in (\PP^1)^3$. Assume all $w_i >0$, $p_i$ are mutually distinct and $u_1=u_2=0$, $u:=u_3$. Then $\tch(\F)$ is given by
\[
e^{-2 \pi \sqrt{-1} f u} \bigg\{ 2 - \big(2u + w_1+w_2+w_3\big) x + \big(u^{2} + (w_1+w_2+w_3)u + \frac{1}{2}(w_{1}^{2}+w_{2}^{2}+w_{3}^{2})\big)x^2 \bigg\},
\]for any $f \in D$;
$$
e^{-2 \pi \sqrt{-1} f u} \bigg\{(1+ e^{-2 \pi \sqrt{-1} f w_j})- \big((1+e^{-2 \pi \sqrt{-1} f w_j}) u + w_i + e^{-2 \pi \sqrt{-1} f w_j} w_ j + w_k\big)x \bigg\},
$$ for any $(i,j)=(1,2),(2,3),(3,1)$, $f \in D_{ij}$ and $k\in \{1,2,3\} \setminus \{i,j\}$;
$$
e^{-2 \pi \sqrt{-1} f u} \left\{e^{-2 \pi \sqrt{-1} f w_i}+e^{-2 \pi \sqrt{-1} f w_{i+1}}\right\},$$ for any $i=1,2,3$, $f \in D_i$ (and $w_4:=w_1$). Here $x$ always denotes the first Chern class of $\O_{\PP}(1)$ pulled-back along $\pi : I\PP \rightarrow \PP$ and restricted to the relevant component of $I\PP$.
\end{lemma}
\begin{proof}
By direct computation from Proposition \ref{rank2}.
\end{proof}

Fix $\alpha \in A^0(I\PP)_{\bmu_{\infty}}$ and $\beta \in A^1(I\PP)_{\bmu_{\infty}}$ (parts of) the Chern character of a rank 2 locally free sheaf on $\PP$. Let $$\tch(u, w_1, w_2, w_3) \in A^*(I\PP)_{\bmu_\infty}$$ be the formula for the Chern character as appearing in the previous lemma. Consider the following system of equations in variables $u \in \Z$, $w_1,w_2,w_3 \in \Z_{>0}$ 
\begin{align*}
\tch^2(u, w_1, w_2, w_3) = \alpha, \ \tch^1(u, w_1, w_2, w_3) = \beta.
\end{align*}
Note that $\tch^2$ only has components in $D$ and therefore the $\tch^2$ equation does not depend on $w_1,w_2,w_3$. The $\tch^2$ equation is automatically satisfied in the case of a trivial gerbe structure, i.e.~when $d=1$. In the case of a general gerbe structure the $\tch^2$ equation becomes
\begin{equation}
\alpha = \left\{ \begin{array}{cc} 2d & \mathrm{if} \ u \equiv 0 \mod d \\ 0 & \mathrm{otherwise}. \end{array}\right..
\end{equation}
so imposes a divisorial condition on $u$.
Next, look at the $f = 0 \in D$ part of the $\tch^1$ equation and let $\beta_0$ be the $f=0 \in D$ part of $\beta$. This equation has a solution if and only if $$2 \ | \ \beta_0 + w_1 + w_2 + w_3,$$ in which case
\[
u = -\frac{1}{2} \left(\beta_0 + w_1 + w_2 + w_3 \right).
\]
What remains are the $\tch^1$ equations for $f \in D \setminus 0$, which we analyze further in examples.

For each $f \in D$ we introduce a formal variable $p_f$. Likewise, for each $f \in D_{ij}$ we introduce a formal variable $q_{ij,f}$. Finally, for each $f \in D_{i}$ we introduce a formal variable $r_{i,f}$. Note that by working with Chern character instead of $K$-group class, we know these variables are \emph{independent}. We can now explicitly compute the generating function $\rH^{\vb}_{\alpha,\beta}$.
\begin{theorem} \label{mainrank2}
For any $\alpha \in A^0(I\PP)_{\bmu_{\infty}}$ and $\beta \in A^1(I\PP)_{\bmu_{\infty}}$ (part of) the Chern character of a rank 2 locally free sheaf on $\PP$, the generating function $\rH^{\vb}_{\alpha,\beta}$ is given by
\[
\sum_{(u,w_1,w_2,w_3) \in C_{\alpha,\beta}} \prod_{f \in D} p_{f}^{\tch^0(u,w_1,w_2,w_3)_f} \prod_{{\scriptsize{\begin{array}{c} i<j \\f \in D_{ij}\end{array}}}} q_{ij,f}^{\tch^0(u,w_1,w_2,w_3)_f}  \prod_{i, f \in D_{i}} r_{i,f}^{\tch^0(u,w_1,w_2,w_3)_f},
\]
where
\begin{align*}
C_{\alpha,\beta} := \big\{ (u,w_1,w_2,w_3) &\in \Z \times \Z_{> 0}^{3} \ : \ b \mid w_1, \ c \mid w_2, \ a \mid w_3, \ \tch^2(u,w_1, w_2, w_3) = \alpha, \\
&\tch^1(u,w_1, w_2, w_3) = \beta, \ w_i < w_j + w_k \ \forall \{i,j,k\}=\{1,2,3\} \big\}.
\end{align*}
Here $\tch^k(u,w_1,w_2,w_3)$ is the explicit formula for the codegree $k$ part of the Chern character appearing in Lemma \ref{rank2tch} and $\tch^k(u,w_1,w_2,w_3)_f$, $f \in F$ are its various components.
\end{theorem}
\begin{proof}
We start as in the rank 1 case. Fix a Chern character $\tch \in A^*(I\PP)_{\bmu_{\infty}}$ of a rank 2 locally free sheaf on $\PP$. The moduli scheme $N(\tch)$ of $\mu$-stable locally free sheaves on $\PP$ with Chern character $\tch$ has a torus action, which at the level of closed points is given by
\[
t\cdot[\F] := [t^* \F], \ t \in \T, \ \F \in N(\tch).
\]
By localization, $e(N(\tch)) = e(N(\tch)^\T)$. As a set, $N(\tch)^\T$ is the collection of isomorphism classes of rank 2 $\mu$-stable locally free sheaves $\F$ on $\PP$ with Chern character $\tch$ satisfying $t^* \F \cong \F$ for all $t \in \T$. Since any such $\F$ is simple, it admits a $\T$-equivariant structure \cite[Prop.~4.4]{Koo1} and this $\T$-equivariant structure is unique up to tensoring by a character of $\T$ \cite[Prop.~4.5]{Koo1}. Consequently, as a set, $N(\tch)^\T$ is the collection of \emph{equivariant} isomorphism classes of $\mu$-stable rank 2 toric locally free sheaves on $\PP$ with class $\tch$, where two such classes $[\F], [\F']$ are identified whenever $\F$ and $\F'$ differ by a character, i.e. whenever $\F \cong \F' \otimes L_{(u_1,u_2,u_3)}$ for some $u_1+u_2+u_3 = 0$. 

Therefore, any element of $N(\tch)^\T$ can be \emph{uniquely} represented by stacky $S$-families $\{\hat{F}_i\}_{i=1,2,3}$ of a $\mu$-stable rank 2 toric locally free sheaf as in Proposition \ref{rank2} satisfying 
\begin{align*}
&\tch(\F) = \tch, \ u_2 = u_3 = 0, \ p_1 = (1:0), \ p_2 = (0:1), \ p_3 = (1:1), \\
&0< w_1 < w_2 + w_3, \ 0< w_2 < w_1 + w_3, \ 0< w_3 < w_1 + w_2.
\end{align*}
Here we have used Proposition \ref{rank2classify} to characterise $\mu$-stability. Defining $u:=u_3$, the result follows by summing over all $(u,w_1,w_2,w_3)$ subject to the above constraints.
\end{proof}

The generating function of the previous theorem is the most refined version possible. Easier formulae can be obtained by specializing. One such specialization is by grouping together moduli spaces with the same value of the ``modified holomorphic Euler characteristic'' $\chi_{\cE}(\cdot) = P_{\O_{\bf{P}}(1), \cE}(\cdot,0) \in \Z$. In this context, it is more natural to fix $$\alpha = \tch^2(\cdot \otimes \cE) \in A^0(I\PP)_{\bmu_{\infty}}, \ \beta = \tch^1(\cdot \otimes \cE) \in A^1(I\PP)_{\bmu_{\infty}}$$ as our topological invariants. In other words, we are interested in the Euler characteristic of the moduli space of $\mu$-stable rank 2 locally free sheaves $\F$ on $\PP$ with $\tch^2(\F \otimes \cE) = \alpha$, $\tch^1(\F \otimes \cE) = \beta$ and fixed $\chi_{\cE}(\F) \in \Z$. We denote by $q$ the formal variable corresponding to $\chi_{\cE}(\F)$ and denote the corresponding generating function by $$\rH^{\vb}_{\alpha,\beta}(q).$$ 

Let $\F$ be a $\mu$-stable rank 2 toric locally free sheaf on $\PP$ described, as above, by the data $u_1=u_2=0$, $u:=u_3 \in \Z$, $w_1, w_2, w_3 \in \Z_{>0}$ satisfying $$b \mid w_1, \ c \mid w_2, \ a \mid w_3.$$ Define the function $$\phi_E(t):= t+2t^2+3t^3+\cdots+(E-1)t^{E-1}.$$ Then
\begin{align*}
\tch^2(\F \otimes \cE)_0 &= 2E, \\
\tch^2(\F \otimes \cE)_f &= 0, \ \forall f \in D \setminus 0, \\
\tch^1(\F \otimes \cE)_0 &= -(2u+w_1+w_2+w_3 - (E-1))E, \\
\tch^1(\F \otimes \cE)_f &= 2 \phi_E(e^{2 \pi \sqrt{-1} f}) e^{-2 \pi \sqrt{-1} f u}, \ \forall f \in D \setminus 0, \\
\tch^1(\F \otimes \cE)_f &= 0, \ \forall f \in D_{12}, D_{13}, D_{23}.
\end{align*}
Here the vanishing in the second and last line are due to the presence of the generating sheaf $\cE$. From these equations, we see that the classes $\alpha$, $\beta$ must satisfy corresponding vanishing and divisorial properties in order for $\rH^{\vb}_{\alpha,\beta}(q)$ to be non-trivial. In particular $\beta_0$ has to be divisible by $E$.  Suppose this is the case and define $$c_1 := \frac{\beta_0}{E}-(E-1)E.$$ 
Then the above equations are equivalent to
\begin{align*} 
-2u-w_1-w_2-w_3 &= c_1, \\
2 \phi_E(e^{2 \pi \sqrt{-1} f}) e^{-2 \pi \sqrt{-1} f u} &= \beta_f \ \mathrm{for \ all \ } f \in D \setminus 0. 
\end{align*}
These equations should be thought of as fixing ``the first Chern class''. The second set of equations only depend on the value of $u$ modulo $d$ and can be thought of as fixing the $\bmu_d$-eigenvalue of the rank 2 locally free sheaves we are considering. Therefore instead of fixing $\alpha, \beta$ and imposing the equations above, we fix $c_1 \in \Z$, $\lambda \in \{0, \ldots, d-1\}$ and impose the equations 
\begin{align} 
\begin{split} \label{c1lambda}
-2u-w_1-w_2-w_3 &= c_1, \\
u &\equiv \lambda \mod d,
\end{split}
\end{align}
and refer to the resulting generating function as $\rH^{\vb}_{c_1, \lambda}(q)$.

We now compute $\chi_{\cE}(\F)$. First some necessary notation. For any integers $m$ and $n>0$, we denote by $[m]_n$ the unique value $k \in \{0, \ldots, n-1\}$ for which $m \equiv k \ \mathrm{mod} \ n$. Also, for integers $m_1,m_2,m_3$ and $n>0$ define $$\psi_E(m_1,m_2,m_3,n) := \sum_{{\scriptsize{\begin{array}{c} k=1 \\ \frac{n}{\mathrm{gcd}(m_1,n)} \nmid k \end{array}}}}^{n-1}\frac{1+e^{-\frac{2 \pi \sqrt{-1} k m_2}{n}}}{1 - e^{\frac{2 \pi \sqrt{-1} k m_1}{n}}} e^{-\frac{2 \pi \sqrt{-1} k m_3}{n}} \phi_E(e^{\frac{2 \pi \sqrt{-1} k}{n}}).$$ Lemma \ref{rank2tch} together with the expression for the Todd classes in Section 4.2 gives the following formula for the modified holomorphic Euler characteristic $\chi_{\cE}(\F)$
\begin{align*}
\chi_{\cE}(\F) = &\frac{E}{abc}\Bigg\{\frac{1}{4} c_{1}^{2} + \frac{1}{4} \sum_{i=1}^{3} w_{i}^{2} - \frac{1}{2} \sum_{1 \leq i<j \leq 3} w_i w_j + \frac{1}{2} c_1 \sum_{i=1}^{3} \hat{i} + \frac{1}{6} \sum_{i=1}^{3} \hat{i}^2 + \frac{1}{2} \sum_{1 \leq i < j \leq 3} \hat{i} \hat{j} + [u]_{d}^{2} \\
&+ \Big(c_1+E-d + \sum_{i=1}^{3} \hat{i}\Big) [u]_d + \frac{1}{2}(E-d)\Big(c_1+\sum_{i=1}^{3}\hat{i}\Big)  + \frac{1}{3} E^2 - \frac{1}{2} E d + \frac{1}{6} d^2 \Bigg\} \\
&+ \sum_{1 \leq i < j \leq 3} \frac{1}{\hat{i}\hat{j}}\psi_E(\hat{k},w_{k-1},u,d_{ij}),
\end{align*}
where $\hat{i}$ was defined in Section 4.1 and $w_0 := w_3$. Substituting
\[
u = -\frac{1}{2} \Big(c_1+\sum_{i=1}^{3} w_i \Big)
\]
in the above expression gives a function which only depends on $w_1, w_2, w_3$ (for fixed $\PP$, $E$, $c_1$). We denote this function by
\[
Q_{E, c_1}(w_1,w_2,w_3).
\]
From Theorem \ref{mainrank2} we deduce the following corollary.
\begin{corollary} \label{unrefined}
For any $c_1 \in \Z$ and $\lambda \in \{0,\ldots, d-1\}$
\[
\rH_{c_1, \lambda}^{\vb}(q) = \sum_{(w_1,w_2,w_3) \in C_{c_1, \lambda}} q^{Q_{E, c_1}(w_1,w_2,w_3)},
\]
where
\begin{align*}
C_{c_1, \lambda} := \Big\{&(w_1,w_2,w_3) \in \Z_{>0} \ : \ b \mid w_1, \ c \mid w_2, \ a \mid w_3, \ 2 \mid c_1 + \sum_{i=1}^{3} w_i, \\
&-\frac{1}{2}\Big(c_1+\sum_{i=1}^{3} w_i \Big) \equiv \lambda \ \mathrm{mod} \ d, \  w_i < w_j + w_k \ \forall \{i,j,k\} = \{1,2,3\} \Big\}
\end{align*}
and $Q_{E, c_1}(w_1,w_2,w_3)$ was defined above.
\end{corollary}

We end by computing rank 2 generating functions in several explicit examples. 
In examples 1--3, it is easy to see that if $\gcd(2,c_1)=1$, then there are no strictly $\mu$-semistables around. In example 4, a more complicated numerical criterion for the absence of strictly $\mu$-semistables is given. \\

\noindent \textbf{Example 1.}  Let $\PP = \PP^2$, then $I\PP = \PP$. We take $\cE := \O_\PP$. Fixing $\alpha = 2$ and $\beta = c_1 = (\tch^1)_0$, the generating function $\rH^{\vb}_{\alpha,\beta}$ is given by
\[
\sum_{(w_1,w_2,w_3) \in C_{c_1}} p^{Q_{c_1}(w_1,w_2,w_3)},
\]
where
\begin{align*}
&C_{c_1} := \Big\{(w_1,w_2,w_3) \in \Z_{>0} \ : \ 2 \mid c_1 + \sum_{i=1}^{3} w_i, \ w_i < w_j + w_k \ \forall \{i,j,k\} = \{1,2,3\} \Big\}, \\
&Q_{c_1}(w_1,w_2,w_3) := \frac{1}{4}c_{1}^{2} + \frac{1}{4} \sum_{i=1}^{3} w_{i}^{2} - \frac{1}{2} \sum_{1 \leq i < j \leq 3} w_i w_j. 
\end{align*}
Replacing $p$ by $q$, this is also equal to the generating function $\rH^{\vb}_{c_1,0}(q)$ (up to an additional constant $\frac{3}{2}c_1+2$ in the power of $q$ making it integer-valued). 
This generating function was first computed (also by torus localization) by Klyachko \cite{Kly2}. He expressed his answer in terms of Hurwitz class numbers as explained at the end of this section in the proof of Theorem \ref{intro}. \\

\noindent \textbf{Example 2.} Let $\PP = \PP(1,1,2)$, then $I\PP = \PP \sqcup \PP(2)$. We take $\cE := \O_\PP \oplus \O_\PP(1)$. The codegree 0 Chern character over the twisted sector $\PP(2)$ is (Lemma \ref{rank2tch}) 
\[
(\tch^0)_{1/2} = (-1)^{u}((-1)^{w_1}+(-1)^{w_3}).
\]
Recall that $2 \mid w_2$. This part of the Chern character can only take values $-2,0,2$. We denote the variable keeping track of $(\tch^0)_{0}$ by $p$ and the variable keeping track of $(\tch^0)_{1/2}$ by $r$. For any subset $X \subset \Z^3$, we denote by $\mathbbm{1}_X$ the characteristic function taking value 1 on $X$ and 0 elsewhere. Fixing $\alpha = 2$ and $\beta = c_1 = (\tch^1)_0$, the generating function $\rH^{\vb}_{\alpha,\beta}$ is
\[
\sum_{(w_1,w_2,w_3) \in C_{c_1}} X_{c_1}(r) p^{Q_{c_1}(w_1,w_2,w_3)},
\]
where
\begin{align*}
&C_{c_1} := \Big\{(w_1,w_2,w_3) \in \Z_{>0} \ : \ 2 \mid c_1 + \sum_{i=1}^{3} w_i, 2 \mid w_2, w_i < w_j + w_k \ \forall \{i,j,k\} = \{1,2,3\} \Big\}, \\
&Q_{c_1}(w_1,w_2,w_3) := \frac{1}{4}c_{1}^{2} + \frac{1}{4} \sum_{i=1}^{3} w_{i}^{2} - \frac{1}{2} \sum_{1 \leq i < j \leq 3} w_i w_j, 
\end{align*}
\begin{align*}
&X_{c_1}(r) := \Bigg(\mathbbm{1}_{{\scriptsize{\left\{\begin{array}{c}  4 \nmid c_1 + \sum_{i=1}^{3} w_i \\ 2 \mid w_1, 2 \mid w_3 \end{array}\right\} }}} + \mathbbm{1}_{{\scriptsize{\left\{\begin{array}{c}  4 \mid c_1 + \sum_{i=1}^{3} w_i \\ 2 \nmid w_1, 2 \nmid w_3 \end{array}\right\} }}} \Bigg) r^{-2} \\
&\qquad\qquad +\Big(\mathbbm{1}_{{\scriptsize{\left\{\begin{array}{c} 2 \mid w_1, 2 \nmid w_3 \end{array}\right\} }}} + \mathbbm{1}_{{\scriptsize{\left\{\begin{array}{c}  2 \nmid w_1, 2 \mid w_3 \end{array}\right\} }}} \Big) r^0 \\
&\qquad\qquad +\Bigg(\mathbbm{1}_{{\scriptsize{\left\{\begin{array}{c}  4 \mid c_1 + \sum_{i=1}^{3} w_i \\ 2 \mid w_1, 2 \mid w_3 \end{array}\right\} }}} + \mathbbm{1}_{{\scriptsize{\left\{\begin{array}{c}  4 \nmid c_1 + \sum_{i=1}^{3} w_i \\ 2 \nmid w_1, 2 \nmid w_3 \end{array}\right\} }}} \Bigg) r^2. 
\end{align*}
Note that $X_{c_1}(1)=1$. The specialised generating function $\rH_{c_1,0}^{\vb}(q)$ is obtained by setting $r=1$ in the above expression (up to an addition constant $\frac{5}{2} c_1 + 6$ in the power of $q$). This is the same generating function as in example 1 (up to a constant $c_1+4$ in the power of $q$) \emph{except} for the additional divisorial condition $2 \mid w_2$ in the definition of $C_{c_1}$! \\

\noindent \textbf{Example 3.} Let $\PP = \PP(1,2,2)$, then $I\PP = \PP \sqcup \PP(2,2)$. We take $\cE := \O_\PP \oplus \O_\PP(1)$. This time, the part of $A^*(I\PP)_{\bmu_{\infty}}$ coming from the twisted sector $\PP(2,2)$ has a codegree 1 and 0 part. As before, we fix $\alpha = 2$, $\beta_0 = c_1 = (\tch^1)_0$, but this time we also fix $\beta_{1/2} = (\tch^1)_{1/2}$. The formula for $\beta_{1/2}$ is
\[
\beta_{1/2} = (-1)^{u}(1+(-1)^{w_3}),
\]
so it can take values $-2,0,2$. Recall that $2 \mid w_1$ and $2 \mid w_2$. Each of the possible values of $\beta_{1/2}$ gives a different generating function. We denote the variable keeping track of $(\tch^0)_{0}$ by $p$ and the variable keeping track of $(\tch^0)_{1/2}$ by $q$. \\

\noindent \emph{Case (1):} Fix $\beta_{1/2} = -2$. This is equivalent to $2 \mid w_3$ and $4 \nmid c_1 + \sum_{i=1}^{3} w_i$. In this case 
\begin{align*}
(\tch^0)_{1/2} &= -(-1)^{u}\big((1+(-1)^{w_3}) u + w_1 + w_2 + (-1)^{w_3} w_3 \big) = 2u + \sum_{i=1}^{3} w_i = -c_1,
\end{align*}
so the generating function $\rH^{\vb}_{\alpha,\beta}$ is given by
\[
\sum_{(w_1,w_2,w_3) \in C_{c_1}^{(1)}} p^{Q_{c_1}^{(1)}(w_1,w_2,w_3)} q^{R_{c_1}^{(1)}(w_1,w_2,w_3)},
\]
where
\begin{align*}
&C_{c_1}^{(1)} := \Big\{(w_1,w_2,w_3) \in \Z_{>0} \ : \ 2 \mid c_1 + \sum_{i=1}^{3} w_i, \ 4 \nmid c_1 + \sum_{i=1}^{3} w_i, \ 2 \mid w_1, \ 2 \mid w_2, \ 2 \mid w_3, \\ 
&\qquad\qquad\qquad\qquad\qquad\qquad\qquad \!\!\! w_i < w_j + w_k \ \forall \{i,j,k\} = \{1,2,3\} \Big\}, \\
&Q_{c_1}^{(1)}(w_1,w_2,w_3) := \frac{1}{4}c_{1}^{2} + \frac{1}{4} \sum_{i=1}^{3} w_{i}^{2} - \frac{1}{2} \sum_{1 \leq i < j \leq 3} w_i w_j, \\
&R_{c_1}^{(1)}(w_1,w_2,w_3) := -c_1.
\end{align*}

\noindent \emph{Case (2):} Fix $\beta_{1/2} = 0$. This is equivalent to $2 \nmid w_3$. In this case, the formula for $(\tch^0)_{1/2}$ is
\begin{align*}
(\tch^0)_{1/2} &= -(-1)^{u}\big((1+(-1)^{w_3}) u + w_1 + w_2 + (-1)^{w_3} w_3 \big) \\
&= (-1)^{\frac{1}{2}(c_1 + \sum_{i=1}^{3} w_i)}(-w_1-w_2+w_3).
\end{align*}
The generating function $\rH^{\vb}_{\alpha,\beta}$ is given by
\[
\sum_{(w_1,w_2,w_3) \in C_{c_1}^{(2)}} p^{Q_{c_1}^{(2)}(w_1,w_2,w_3)} q^{R_{c_1}^{(2)}(w_1,w_2,w_3)},
\]
where
\begin{align*}
&C_{c_1}^{(2)} := \Big\{(w_1,w_2,w_3) \in \Z_{>0} \ : \ 2 \mid c_1 + \sum_{i=1}^{3} w_i, \ 2 \mid w_1, \ 2 \mid w_2, \ 2 \nmid w_3, \\ 
&\qquad\qquad\qquad\qquad\qquad\qquad\qquad \!\!\! w_i < w_j + w_k \ \forall \{i,j,k\} = \{1,2,3\} \Big\}, \\
&Q_{c_1}^{(2)}(w_1,w_2,w_3) := \frac{1}{4}c_{1}^{2} + \frac{1}{4} \sum_{i=1}^{3} w_{i}^{2} - \frac{1}{2} \sum_{1 \leq i < j \leq 3} w_i w_j, \\
&R_{c_1}^{(2)}(w_1,w_2,w_3) := (-1)^{\frac{1}{2}(c_1 + \sum_{i=1}^{3} w_i)}(-w_1-w_2+w_3).
\end{align*}

\noindent \emph{Case (3):} Fix $\beta_{1/2} = 2$. Similar to case I, we obtain
\[
\sum_{(w_1,w_2,w_3) \in C_{c_1}^{(3)}} p^{Q_{c_1}^{(3)}(w_1,w_2,w_3)} q^{R_{c_1}^{(3)}(w_1,w_2,w_3)},
\]
where
\begin{align*}
&C_{c_1}^{(3)} := \Big\{(w_1,w_2,w_3) \in \Z_{>0} \ : \ 4 \mid c_1 + \sum_{i=1}^{3} w_i, \ 2 \mid w_1, \ 2 \mid w_2, \ 2 \mid w_3, \\ 
&\qquad\qquad\qquad\qquad\qquad\qquad\qquad \!\!\! w_i < w_j + w_k \ \forall \{i,j,k\} = \{1,2,3\} \Big\}, \\
&Q_{c_1}^{(3)}(w_1,w_2,w_3) := \frac{1}{4}c_{1}^{2} + \frac{1}{4} \sum_{i=1}^{3} w_{i}^{2} - \frac{1}{2} \sum_{1 \leq i < j \leq 3} w_i w_j, \\
&R_{c_1}^{(3)}(w_1,w_2,w_3) := c_1.
\end{align*}
The specialized generating function $\rH^{\vb}_{c_1,0}(q)$ is given by
\[
\sum_{(w_1,w_2,w_3) \in C_{c_1}} q^{Q_{c_1}(w_1,w_2,w_3)},
\]
where
\begin{align*}
&C_{c_1} := \Big\{(w_1,w_2,w_3) \in \Z_{>0} \ : \ 2 \mid c_1 + \sum_{i=1}^{3} w_i, \ 2 \mid w_1, \ 2 \mid w_2, \\ 
&\qquad\qquad\qquad\qquad\qquad\qquad\qquad \!\!\!\!\! w_i < w_j + w_k \ \forall \{i,j,k\} = \{1,2,3\} \Big\}, \\
&Q_{c_1}(w_1,w_2,w_3) := \frac{1}{8}c_{1}^{2} + \frac{3}{2} c_1 + \frac{17}{4} + \frac{1}{8} \sum_{i=1}^{3} w_{i}^{2} - \frac{1}{4} \sum_{1 \leq i < j \leq 3} w_i w_j \\
&\qquad\qquad\qquad\quad\quad - \frac{1}{8}(-1)^{\frac{1}{2}(c_1+\sum_{i=1}^{3} w_i)}(1+(-1)^{w_3}). 
\end{align*}

\noindent \textbf{Example 4.} Let $\PP = \PP(2,2,2)$, then $I\PP = \PP \sqcup \PP$. We take $\cE := \O_\PP \oplus \O_\PP(1)$. This time, we fix $(\tch^2)_f$ and $(\tch^1)_f$ for $f \in D = \{0,1/2\}$. We take $\alpha_0 = 2$. Fixing $\alpha_{1/2} = (\tch^2)_{1/2} = \pm 2$ is equivalent to fixing the parity of $u$. The generating function is zero unless $$\beta_0 = c_1 =(\tch^1)_0 = (-1)^u(\tch^1)_{1/2} = (-1)^u \beta_{1/2},$$ so we assume this is the case. We use the formal variable $p_0$ to keep track of $(\tch^0)_0$ and $p_{1}$ to keep track of $(\tch^0)_{1/2}$. \\

\noindent \emph{Case (1):} Fix $\alpha_{1/2} = -2$. Then $\rH^{\vb}_{\alpha,\beta}$ is
\[
\sum_{(w_1,w_2,w_3) \in C_{c_1}^{(1)}} (p_0 p_{1}^{-1})^{Q_{c_1}^{(1)}(w_1,w_2,w_3)},
\]
where
\begin{align*}
&C_{c_1}^{(1)} := \Big\{(w_1,w_2,w_3) \in \Z_{>0} \ : \ 2 \mid c_1 + \sum_{i=1}^{3} w_i, \ 4 \nmid c_1 + \sum_{i=1}^{3} w_i, \ 2 \mid w_1, \ 2 \mid w_2, \ 2 \mid w_3, \\ 
&\qquad\qquad\qquad\qquad\qquad\qquad\qquad \!\!\! w_i < w_j + w_k \ \forall \{i,j,k\} = \{1,2,3\} \Big\}, \\
&Q_{c_1}^{(1)}(w_1,w_2,w_3) := \frac{1}{4}c_{1}^{2} + \frac{1}{4} \sum_{i=1}^{3} w_{i}^{2} - \frac{1}{2} \sum_{1 \leq i < j \leq 3} w_i w_j.
\end{align*}

\noindent \emph{Case (2):} Fix $\alpha_{1/2} = 2$. Then $\rH^{\vb}_{\alpha,\beta}$ is
\[
\sum_{(w_1,w_2,w_3) \in C_{c_1}^{(2)}} (p_0 p_{1})^{Q_{c_1}^{(2)}(w_1,w_2,w_3)},
\]
where
\begin{align*}
&C_{c_1}^{(2)} := \Big\{(w_1,w_2,w_3) \in \Z_{>0} \ : \ 4 \mid c_1 + \sum_{i=1}^{3} w_i, \ 2 \mid w_1, \ 2 \mid w_2, \ 2 \mid w_3, \\ 
&\qquad\qquad\qquad\qquad\qquad\qquad\qquad \!\!\! w_i < w_j + w_k \ \forall \{i,j,k\} = \{1,2,3\} \Big\}, \\
&Q_{c_1}^{(2)}(w_1,w_2,w_3) := \frac{1}{4}c_{1}^{2} + \frac{1}{4} \sum_{i=1}^{3} w_{i}^{2} - \frac{1}{2} \sum_{1 \leq i < j \leq 3} w_i w_j.
\end{align*}
For the specialized generating function $\rH^{\vb}_{c_1,\lambda}(q)$, we fix $c_1 \in \Z$ and a $\bmu_2$-eigenvalue $\lambda \in \{0,1\}$. Note that the generating function is zero unless $c_1$ is even, so assume this is the case. We get the following expression for $\rH^{\vb}_{c_1,\lambda}(q)$ 
\[
\sum_{(w_1,w_2,w_3) \in C_{c_1}} q^{Q_{c_1}(w_1,w_2,w_3)},
\]
where
\begin{align*}
&C_{c_1} := \Big\{(w_1,w_2,w_3) \in \Z_{>0} \ : \ -\frac{1}{2}\big( c_1 + \sum_{i=1}^{3} w_i \big) \equiv \lambda \ \mathrm{mod} \ 2, \ 2 \mid w_1, \ 2 \mid w_2, 2 \mid w_3, \\ 
&\qquad\qquad\qquad\qquad\qquad\qquad\qquad \!\!\! w_i < w_j + w_k \ \forall \{i,j,k\} = \{1,2,3\} \Big\}, \\
&Q_{c_1}(w_1,w_2,w_3) := \frac{1}{16}c_{1}^{2} + \frac{7}{8}c_1 + \frac{23}{8} - \frac{1}{8} (-1)^{\lambda} (c_1 + 7) + \frac{1}{16} \sum_{i=1}^{3} w_{i}^{2} - \frac{1}{8} \sum_{1 \leq i < j \leq 3} w_i w_j . 
\end{align*}
For $\lambda$ even, this equals the generating function $\rH^{\vb}_{\frac{c_1}{2},0}(q)$ of $\PP^2$. For $\lambda$ odd, this equals the generating function $\rH^{\vb}_{\frac{c_1}{2}+1,0}(q)$ of $\PP^2$. It is not hard to see that for the choices $(\frac{1}{2}c_1,\lambda)=(\mathrm{odd},0), (\mathrm{even},1)$, there are no strictly $\mu$-semistables.  

From these examples, we can prove Theorem \ref{intro} of the introduction, which expresses the generating function $\rH^{\vb}_{c_1,\lambda}(q)$ of $\PP(a,b,c)$ with $a,b,c \leq 2$ in terms of Hurwitz class numbers $H(\Delta)$ and the sum of divisors function $\sigma_0(n)$, and discusses modularity. 

\begin{proof} [Proof of Theorem \ref{intro}]
Most of the proof is a variation on the reasoning in Klyachko's paper \cite{Kly3}. For fixed $c_1 \in \Z$ and $\Delta \in \Z$ consider the system of equalities and inequalities
\begin{align}
\begin{split} \label{P2case}
&(w_1, w_2, w_3) \in \Z_{>0}^{3} \ \mathrm{such \ that}  \\
&w_1 < w_2 + w_3, w_2 < w_1 + w_3, w_3 < w_1 + w_2 \ \mathrm{and} \\
&-\Delta = \sum_{i=1}^{3} w_{i}^{2} - 2 \sum_{i<j} w_i w_j, \ -\Delta \equiv c_{1}^{2} \ \mathrm{mod} \ 4.
\end{split}
\end{align}
Klyachko shows the number of solutions to this system is $3H(\Delta)$ when $\Delta \equiv -1 \ \mathrm{mod} \ 4$ and $3(H(\Delta) - \frac{1}{2} \sigma_0(\Delta/4))$ when $\Delta \equiv 0 \ \mathrm{mod} \ 4$. Note that the equations imply that $\Delta$ is congruent $-1$ or $0$ mod $4$. Also note that the equations imply\footnote{This gives a short toric proof of the Bogomolov inequality for toric rank 2 $\mu$-stable locally free sheaves on $\PP^2$.} $\Delta > 0$. The way this is proved is as follows. Split up 13 disjoint cases
$$
w_i < w_j < w_k, \ w_i = w_j < w_k, \ w_i < w_j = w_k, \ w_1 = w_2 = w_3,
$$
for each choice of $\{i,j,k\}=\{1,2,3\}$ and solve \eqref{P2case} in each of these 13 cases. In each case define $A:=w_i$, $B:=w_i+w_j - w_k$ and $C:=w_k$. For example, the case $w_1 < w_2 < w_3$ is equivalent to counting integer forms $AX^2+B XY + C Y^2$ of discriminant $B^2 - 4AC = -\Delta$ satisfying $B>0$, $C>A$, $-A<B<A$. Adding all 13 cases gives: 3 times the number of integer forms $AX^2+BXY+CY^2$ of discriminant $-\Delta$ satisfying $B \neq 0$, $C > A$, $-A < B \leq A$ or satisfying $B \neq 0$, $C=A$, $0 \leq B \leq A$ and counted by the size of the automorphism group (see footnote 3). We call this number $\tilde{H}(\Delta)$. In the case $c_1$ is odd, it is easy to see that $B \neq 0$ is automatically satisfied. Therefore $\tilde{H}(\Delta) = H(\Delta)$ because every equivalence class of binary integer positive definite quadratic forms has a unique representative of this form (i.e.~its Gauss reduced form \cite{BS}). For $c_1$ even $\tilde{H}(\Delta)$ and $H(\Delta)$ differ by the number of (Gauss reduced) integer forms with $B=0$, i.e.
$$
-\frac{1}{2} \sigma_0(\Delta/4).
$$

We now discuss how the formulae for $\PP(1,1,2)$ are obtained. The case of $\PP(1,2,2)$ is similar and the case of $\PP(2,2,2)$ follows directly from the discussion in Example 4. By Example 2 of this section, the system which needs to be solved is again \eqref{P2case} together with the \emph{additional} requirement $2 \ | \ w_2$. Moreover $Q_{E, c_1}(w_1,w_2,w_3)$ and $\Delta$ are related as follows
$$
Q_{E, c_1}(w_1,w_2,w_3) = - \frac{1}{4} \Delta + \frac{1}{4} c_{1}^{2}.
$$
Going through all 13 cases gives the following answer to the system
\begin{align*}
&\Big| \big\{\mathrm{Gauss \ red. \ forms} \ AX^2+B XY +CY^2 \ \mathrm{with} \ B \neq 0, \ \mathrm{disc.} -\Delta \ \mathrm{and} \ 2 \ | \ A \big\} \Big| \\
&+\Big| \big\{\mathrm{Gauss \ red. \ forms} \ AX^2+B XY +CY^2 \ \mathrm{with} \ B \neq 0, \ \mathrm{disc.} -\Delta \ \mathrm{and} \ 2 \ | \ A-B+C \big\} \Big| \\
&+\Big| \big\{\mathrm{Gauss \ red. \ forms} \ AX^2+B XY +CY^2 \ \mathrm{with} \ B \neq 0, \ \mathrm{disc.} -\Delta \ \mathrm{and} \ 2 \ | \ C \big\} \Big|,
\end{align*}
where $| \cdot |$ is the number of elements counted by the size of the automorphism group. In the case $c_1$ is odd, $B$ is odd so $\Delta \equiv -1 \ \mathrm{mod} \ 4$. This implies $\Delta \equiv -1  \ \mathrm{mod} \ 8$ or $-5  \ \mathrm{mod} \ 8$. In the former case the above reduces to $2H(\Delta)$ and in the latter case the above is $0$. In the case $c_1$ is even, the above reduces to $\tilde{H}(\Delta) + 2 \tilde{H}(\Delta/4)$. Here $\tilde{H}(\Delta) = H(\Delta) - \frac{1}{2} \sigma_0(\Delta/4)$ as in Klyachko's case. Moreover, $\tilde{H}(\Delta/4) = H(\Delta/4) - \frac{1}{2} \sigma_0(\Delta/16)$ when $\Delta \equiv 0 \ \mathrm{mod} \ 16$ and $\tilde{H}(\Delta/4) = H(\Delta/4)$ otherwise.

For the last part of the theorem it suffices to show $\sum_{n>0}H(8n-1)q^{n-1/8}$ ($q=e^{2 \pi i z}$, $\mathrm{Im}(z)>0$) is the holomorphic part of a modular form of weight $3/2$. Let $$\curly{H}(z)= -\frac{1}{12} + \sum_{n>0} H(n) q^n, \ q = e^{2 \pi i z}, \ \mathrm{Im}(z) > 0.$$ By work of Zagier \cite{Zag} (see also \cite{HZ}), this is the holomorphic part of a modular form of weight $3/2$. Our generating function can be written as\footnote{A similar argument is used in \cite[p.~92]{HZ}. Also note that, in the case of $\PP^2$, modularity of $\sum_{n>0}H(4n-1)q^{n-1/4}$ follows along the same lines.} $$\sum_{n>0}H(8n-1)q^{n-1/8}=\frac{1}{8}\sum_{r=0}^7e^{2\pi i r/8}\curly{H}\Big(\frac{z+r}{8}\Big)$$ and is therefore the holomorphic part of a modular form of weight $3/2$.  
\end{proof}

\noindent {\tt{amingh@math.umd.edu}}, {\tt{y.jiang@ku.edu}},  {\tt{mkool@math.ubc.ca}} 



\end{document}